\documentclass[11pt]{amsart}
\usepackage[english]{babel}
\usepackage[usenames,dvipsnames]{color}  
\usepackage{array,mathpazo,amsmath,xcolor}
\usepackage{amscd}
\usepackage{stmaryrd, mathtools, enumerate, amssymb}
\usepackage[colorlinks=true,linkcolor=Cerulean, citecolor=LimeGreen, pagebackref]{hyperref} 
\usepackage{stackrel}

\usepackage{ytableau} 

\usepackage{curves} 
\usepackage{epic} 
\usepackage{eepic}  

\usepackage[all]{xy}
\SelectTips{eu}{} 
\xyoption{curve}

\usepackage{tikz}
\usetikzlibrary{arrows,decorations.pathmorphing,decorations.pathreplacing,positioning,shapes.geometric,shapes.misc,decorations.markings,decorations.fractals,calc,patterns,matrix}

\tikzset{>=stealth',
     cvertex/.style={circle,draw=black,inner sep=1pt,outer sep=3pt},
     vertex/.style={circle,fill=black,inner sep=1pt,outer sep=3pt},
     star/.style={circle,fill=yellow,inner sep=0.75pt,outer sep=0.75pt},
     tvertex/.style={inner sep=1pt,font=\criptsize},
     gap/.style={inner sep=0.5pt,fill=white}}

\newcommand{\arrowrl}[3][20]
{
\hspace{-5pt}
\begin{tikzpicture}
\node (A) at (0,0) {};
\node (B) at (1,0) {};
\draw[->] ($(A)+(0,0.2)$) -- node [above] {$\scriptstyle f^*$} ($(B)+(0,0.2)$);
\draw [->] ($(B)+(0,0.2)$) -- node [below] {$\scriptstyle f_*$} ($(A)+(0,0.2)$);
\end{tikzpicture}
\hspace{-5pt}
}
\newcommand{\adjj}[2][20]{\arrowrl}

\tikzset{
    ncbar angle/.initial=90,
    ncbar/.style={
        to path=(\tikztostart)
        -- ($(\tikztostart)!#1!\pgfkeysvalueof{/tikz/ncbar angle}:(\tikztotarget)$)
        -- ($(\tikztotarget)!($(\tikztostart)!#1!\pgfkeysvalueof{/tikz/ncbar angle}:(\tikztotarget)$)!\pgfkeysvalueof{/tikz/ncbar angle}:(\tikztostart)$)
        -- (\tikztotarget)
    },
    ncbar/.default=0.5cm,
}

\tikzset{square left brace/.style={ncbar=0.5cm}}
\tikzset{square right brace/.style={ncbar=-0.5cm}}

\tikzset{round left paren/.style={ncbar=0.5cm,out=120,in=-120}}
\tikzset{round right paren/.style={ncbar=0.5cm,out=60,in=-60}}

\def\lto{{\longrightarrow}}
\def\into{{\hookrightarrow}}
\def\xto{\xrightarrow}
\def\onto{\twoheadrightarrow}
\newcommand{\qurep}[2]{\operatorname{\rightleftarrows}^{#1}_{#2}}  

\newcommand{\fc}{{\mathfrak c}}

\newcommand{\fm}{{\mathfrak m}}

\newcommand{\fp}{{\mathfrak p}}

\newcommand{\cala}{{\mathcal A}}

\newcommand{\cali}{{\mathcal I}}

\newcommand{\CC}{{\mathbb C}}

\newcommand{\KK}{{\mathbb K}}

\newcommand{\MM}{{\mathbb M}}
\newcommand{\NN}{{\mathbb N}}

\newcommand{\RR}{{\mathbb R}}

\newcommand{\boldf}{{\mathbf f}}

\newcommand{\bx}{{\mathbf x}}

\newcommand{\bdot}{\bullet}

\newcommand{\vp}{\varphi}

\DeclareMathOperator{\adj}{adj}
\DeclareMathOperator{\ann}{ann}
\DeclareMathOperator{\Ann}{Ann}

\DeclareMathOperator{\Aut}{Aut}
\DeclareMathOperator{\car}{char}

\DeclareMathOperator{\coker}{coker}
\DeclareMathOperator{\cok}{coker}
\DeclareMathOperator{\Coker}{Coker}

\DeclareMathOperator{\End}{End}

\DeclareMathOperator{\Ext}{Ext}

\DeclareMathOperator{\ev}{ev}

\DeclareMathOperator{\GL}{GL}
\DeclareMathOperator{\gl}{gldim}
\DeclareMathOperator{\gldim}{gldim}
\DeclareMathOperator{\gr}{gr}

\DeclareMathOperator{\Hom}{Hom}
\DeclareMathOperator{\id}{id}
\DeclareMathOperator{\Imm}{Im}

\DeclareMathOperator{\jac}{jac}
\DeclareMathOperator{\Ker}{Ker}

\DeclareMathOperator{\Maps}{\mathsf{Maps}}

\DeclareMathOperator{\rank}{rank}
\DeclareMathOperator{\red}{red}

\DeclareMathOperator{\Sing}{Sing}
\DeclareMathOperator{\SL}{SL}

\DeclareMathOperator{\Spec}{Spec}
\DeclareMathOperator{\supp}{ supp}
\DeclareMathOperator{\Sym}{{\mathbb S}ym}

\newcommand{{\sbullet}}{{\scriptstyle\bullet}}

\newcommand{\diag}{\mathsf {diag}}
\newcommand{\st}{\,\mid\,}

\newcommand{\trv}{\mathrm{trv}}

\DeclareMathOperator{\add}{\mathbf{add}}

\DeclareMathOperator{\Mod}{\ensuremath{ \mathbf{Mod}}}

\DeclareMathOperator{\fmod}{\ensuremath{ \mathbf{mod}}}

\DeclareMathOperator{\CM}{\ensuremath{ \mathbf{CM}}}

\def\Abar{\overline A}

\theoremstyle{definition}
\newtheorem{defn}{Definition}[section]

\newtheorem{rem}[defn]{Remark}
\newtheorem{remark}[defn]{Remark}
\newtheorem{rems}[defn]{Remarks}

\newtheorem{question}[defn]{Question}
\newtheorem{sit}[defn]{}
\newtheorem{example}[defn]{Example}

\newtheorem{exam}[defn]{Example}

\theoremstyle{plain}

\newtheorem{prop}[defn]{Proposition}
\newtheorem{proposition}[defn]{Proposition}
\newtheorem{theorem}[defn]{Theorem}
\newtheorem{lem}[defn]{Lemma}
\newtheorem{lemma}[defn]{Lemma}
\newtheorem{cor}[defn]{Corollary}
\newtheorem{corollary}[defn]{Corollary}

\newtheorem*{thma}{Theorem A}
\newtheorem*{thmb}{Theorem B}
\newtheorem*{thmc}{Theorem C}
\newtheorem*{thmd}{Theorem D}

\def\bexa{\begin{exam}}
\def\eexa{\end{exam}}

\def\bpro{\begin{prop}}
\def\epro{\end{prop}}

\def\bcor{\begin{cor}}
\def\ecor{\end{cor}}

\def\bthm{\begin{theorem}}
\def\ethm{\end{theorem}}

\def\bdfn{\begin{eefn}}
\def\edfn{\end{eefn}}

\def\brem{\begin{rem}}
\def\erem{\end{rem}}

\def\brems{\begin{rems}}
\def\erems{\end{rems}}

\def\bsit{\begin{sit}}
\def\esit{\end{sit}}

\def\blem{\begin{lem}}
\def\elem{\end{lem}}

\def\bdi{\pdfsyncstop\begin{eiagram}}
\def\edi{\end{eiagram}\pdfsyncstart}

\def\ba{\begin{array}}
\def\ea{\end{array}}

\def\bnum{\begin{enumerate}}
\def\enum{\end{enumerate}}

\def\be{\begin{equation}}
\def\ee{\end{equation}}

\def\bproof{\begin{proof}}
\def\eproof{\end{proof}}

\begin{document}
\title[McKay for reflections]
{A McKay correspondence for reflection groups}

\author[Ragnar-Olaf Buchweitz]{Ragnar-Olaf Buchweitz$^\dagger$}
\address{Dept.\ of Computer and Math\-ematical Sciences,
University  of Tor\-onto at Scarborough, 
1265 Military Trail, 
Toronto, ON M1C 1A4,
Canada}

\author{Eleonore Faber}
\address{School of Mathematics, University of Leeds, Leeds LS2 9JT, UK}
\email{e.m.faber@leeds.ac.uk}

\author{Colin Ingalls}
\address{
School of Mathematics and Statistics,
Carleton University, 
Ottawa, ON K1S 5B6,
Canada}
\email{cingalls@math.carleton.ca}

\thanks{
R.-O. \!B.~was partially supported by an NSERC Discovery grant, E.F.~was partially supported by an Oberwolfach Leibniz Fellowship and by a Marie Sk{\l}odowska-Curie fellowship under grant agreement No 789580, C.I.~was partially supported by an NSERC Discovery grant.
} 
\thanks{${}^\dagger$The first author passed away on November 11, 2017.}
\date{\today}

\dedicatory{For Ragnar}

\begin{abstract} 
We construct a noncommutative desingularization of the discriminant of a finite reflection group $G$ as a quotient of the skew group ring $A=S*G$. If $G$ is generated by order two reflections, then this quotient identifies with the endomorphism ring of the reflection arrangement $\mathcal{A}(G)$ viewed as a module over the coordinate ring $S^G/(\Delta)$ of the discriminant of $G$. This yields, in particular, a correspondence between the nontrivial irreducible representations of $G$ to certain maximal Cohen--Macaulay modules over the coordinate ring $S^G/(\Delta)$. These maximal Cohen--Macaulay modules are precisely the nonisomorphic direct summands of the coordinate ring of the reflection arrangement $\cala (G)$ viewed as a module over $S^G/(\Delta)$. We identify some of the corresponding matrix factorizations, namely the so-called logarithmic (co-)residues of the discriminant.  \end{abstract}
\maketitle

\section{Introduction}

The classical McKay correspondence relates representations of a finite subgroup $G \leqslant \SL(2, \CC)$ to exceptional curves on the minimal resolution of singularities of the Kleinian singularity $\CC^2/G$.  By a theorem of Maurice Auslander \cite{Auslander86}, this correspondence can be extended to maximal Cohen--Macaulay (=CM)-modules over the invariant ring of the $G$-action. In particular, Auslander's version of the correspondence holds more generally for \emph{small} finite subgroups $G \leqslant \GL(n,\CC)$.  We study the case where $G$ is a pseudo-reflection group, that is, a group that is \emph{generated} by pseudo-reflections. \\

To this end, let $G \leqslant \GL(n,\CC)$ be a finite group acting on $\CC^n$. By the theorem of Chevalley--Shephard--Todd the quotient $\CC^n/G$ is smooth if and only if $G$ is a pseudo-reflection group, that is, it is generated by pseudo-reflections. Thus, if $G$ is a pseudo-reflection group, at first sight there are no singularities to resolve and it is impossible to ``see'' the irreducible representations as CM-modules over the invariant ring $R$ of the group action: $R$ is a regular ring and it is well-known that in this case all CM-modules are isomorphic to some $R^n$! However, the key idea of this work is to consider the irregular orbits of the group action, on $\CC^n$ this is the \emph{reflection arrangement} $\cala(G)$ (the set of mirrors of $G$) and in the quotient $\CC^n/G$ this is the projection of $\cala(G)$, the so-called \emph{discriminant} of $G$. \\
The group $G \leqslant \GL(n,\CC)$ also acts on $S:=\Sym_\CC(\CC^n)$, then $\CC^n=\Spec(S)$, the quotient $\CC^n/G=\Spec(R)$, where $R:=S^G$ is the invariant ring. If $G$ is a pseudo-reflection group, then $R$ is itself isomorphic to a polynomial ring, and $\cala(G)$ is defined by the \emph{Jacobian} $J \in S$, a (not necessarily reduced) product of linear forms in $S$. The discriminant is given by a polynomial $\Delta \in R$ and its coordinate ring is $R/(\Delta)$. \\

Let us follow this train of thought further: Auslander's theorem states that for a small subgroup $G \leqslant \GL(n,\CC)$ acting on the polynomial ring $S$ the twisted group ring $A=S*G$ is isomorphic to the endomorphism ring $\End_R(S)$, where $R=S^G$. In particular, $\gl A = \dim R =n$, $A$ is a CM-module over $R$ and the nonisomorphic $R$-direct summands of $S$ correspond to the indecomposable projectives of $A$ and consequently to the irreducible representations of $G$, as these correspond to the simple modules over the group ring $\CC G$. For $G$ a pseudo-reflection group, the twisted group ring $A$ still has global dimension $n$ and is a CM-module over the invariant ring $R$. Following our idea, we would like to write $A$ as endomorphism ring over the discriminant, whose coordinate ring is $R/(\Delta)$, but an easy computation shows that the centre of $A$ is in some sense too large: $Z(A)=R$. In order to remedy this, we will consider the quotient $\Abar = A/AeA$, where $e=\frac{1}{|G|}\sum_{g \in G }g \in A$ is the idempotent for the trivial representation. This quotient has nice properties: 

\begin{thma}[=Thm.~\ref{thm:Abar}, Cor.~\ref{Cor:gdimA}, and Prop.~\ref{JisAnnBarA}]
Let $G \leqslant \GL(n,\CC)$ be a finite group (more generally: $G \leqslant \GL(n,K)$, where $K$ is an algebraically closed field such that $|G|$ is invertible in $K$) and assume that $G$ is generated by pseudo-reflections. Denote $A=S*G$ the twisted group ring and set $\Abar=A/AeA$.  Then $\Abar$ is a CM-module over $S/(J)$, the coordinate ring of the reflection arrangement, as well as over $R/(\Delta)$. Moreover, $\Abar$ is Koszul,  and $\gl \Abar \leq n$. If $G \not \cong \mu_2$, then $\gl \Abar =n$.   
\end{thma}

The referee noted that this implies that the map $A \to \Abar$ is a homological epimorphism in the sense of~\cite{MarksVitoria}, and their Theorem 3.3 implies that $\Abar$ is a universal localization of $A$.

In particular, interpreting $A$, $AeA$ and $\Abar$ geometrically, we exhibit a matrix factorization $(\varphi, \psi)$ of $J \in S$ whose cokernel is $\Abar$ as left $S$-module. Curiously, this matrix factorization comes from the group matrix of $G$ (see Section  \ref{Sub:groupmatrix}) and it is (skew-)symmetric in that the $S$--dual (or transpose matrix) $\psi^*$ is equivalent to $\vp$. Further, we determine the decomposition of $\Abar$ into indecomposable summands over $R/(\Delta)$ and the rank of $\Abar$ over the discriminant (Prop.~ \ref{Prop:groupmatrix}). We can also deduce that $\Abar$ is not an endomorphism ring over the discriminant if $G$ has generating pseudo-reflections of order $\geq 3$ (Cor.~\ref{Cor:AbarEndomorphismring}).

The next step is to show that the quotient $\Abar$ is isomorphic to an endomorphism ring over $R/(\Delta)$ if $G$ is generated by reflections of order two.  First we generalize Auslander's theorem ``noncommutatively'': For any $G \leqslant \GL(n,\CC)$ consider the small group $\Gamma:=G \cap \SL(n,\CC)$ and its invariant ring $S^\Gamma$. Then $\Gamma \leqslant G$ is a normal subgroup and 
$$1 \xrightarrow{} \Gamma \xrightarrow{} G \xrightarrow{} G/\Gamma \xrightarrow{} 1 $$
is a short exact sequence of groups. Assume that $H:=G/\Gamma$ is complementary to $\Gamma$, as will be the case for $H$ cyclic of prime order.  From this we obtain the following generalization of Auslander's theorem:

\begin{thmb}[see Prop.~\ref{Prop:BintermsofA} for a more general formulation] In this situation we have $\CC$-algebra isomorphisms 
$$A= S*G  \cong (S*\Gamma)*H \cong \End_{S^\Gamma *H}(S*H) \ ,$$
and $S*H \cong Ae_\Gamma$ as right $S^\Gamma*H\cong e_\Gamma A e_\Gamma$-module, where $e_\Gamma \in A$ is the idempotent $\frac{1}{|\Gamma|}\sum_{\gamma \in \Gamma}\gamma$.
In particular, if $G=\Gamma$ is in $\SL(n,\CC)$, then the above homomorphism reduces to the one considered by Maurice Auslander. \end{thmb}

In order to show that $\Abar$ is an endomorphism ring, we first view $A$ as a CM-module over the (noncommutative) ring $S^\Gamma*H$ and will use the functor $$i^*: \Mod(S^\Gamma*H) \xrightarrow{} \Mod(R/(\Delta)) \ ,$$ 
coming from a standard recollement. For this part we will need that $G$ is a \emph{true reflection group}, that is, generated by reflections of order $2$. Then clearly $H \cong \mu_2$. \\

In order to use the recollement, we consider more generally a regular ring $R$ that is an integral domain, a non-zero divisor $f \in R$, and define the path algebra  
\[
{\begin{tikzpicture}[baseline=(current  bounding  box.center)] 
\node (B) at (-0.2,0) {$B:=R$};

\node (C1) at (0.9,0) {$e_+$} ;
\node (C2) at (5.1,0)  {$e_- $};

\draw [thick ] (0.7,-0.7) to [round left paren ] (0.7,0.7);
\draw [thick ] (5.3,-0.7) to [round right paren] (5.3,0.7);

\draw [->,bend left=25,looseness=1,pos=0.5] (C1) to node[]  [above]{$v$} (C2);
\draw [->,bend left=20,looseness=1,pos=0.5] (C2) to node[] [below] {$u$} (C1);

\node (DD) at (5.7,0) {,};
\end{tikzpicture}} 
\]
with relations ,
\[e_{\pm}^2=e_{\pm}, \  e_+ + e_- =1, \ u=e_+ue_-, \ v=e_{-}ve_{+}, \ uv=fe_{+}, \text{ and } vu=fe_{-} \ .\]
Then matrix factorizations over $B/Be_{-}B \cong R/(f)$ (Lemma \ref{Cor:Bprops}) can be seen as CM-modules over $B$, which leads to the following reformulation of Eisenbud's theorem on matrix factorizations \cite{Eisenbud80}:

\begin{thmc}[Thm.~\ref{Thm:Eisenbud}]
Let $f \in R$ and $B$ as above and let $i^*: \Mod(B) \xrightarrow{} \Mod(B/Be_-B)$ be the functor $i^*=-\otimes_{B}B/Be_-B$ from the standard recollement. Then $i^*$ induces an equivalence of categories
$$\CM(B)/\langle e_-B \rangle \simeq \CM(R/(f)) \ ,$$
where $\langle e_-B \rangle$ is the ideal in the category $\CM(B)$ generated by the object $e_-B$. (Here $\CM(\Lambda)$ stands for the category of CM-modules over a ring $\Lambda$ as defined in Section \ref{Sub:MFandCM}).
\end{thmc}

In particular, consider $T:=R[Z]/(Z^2-f)$, so that $\Spec(T)$ is the double cover of $\Spec(R)$ ramified over $V(f)$. Then this theorem implies Kn\"orrer's result \cite{KnoerrerCohenMacaulay} that $\CM(T*\mu_2)\simeq MF(f)$, where $MF(f)$ stands for the category of matrix factorizations of $f$. 

The last ingredient comes from Stanley's work on semi-invariants: for a true reflection group $G \leqslant \GL(n,\CC)$ acting on $S$, set $R=S^G$, $T=S^\Gamma$, and $f=\Delta$ and $B=T*H$ in the above theorem. Then using that $T \cong R[J]/(J^2-\Delta)$ as $R$-modules (see Lemma \ref{Lem:RasTmod}), one can calculate $i^*(S*H) \cong S/(J)$ as $R/(\Delta)$-module (see Prop.~\ref{Prop:Smodz}). This leads directly to the main theorem: 

\begin{thmd}[=Thm.~\ref{Thm:main} and Corollaries]
Let $G$ be a true reflection group. Then with notation as just introduced, the quotient algebra $\Abar = A/AeA$ is isomorphic to the endomorphism ring $\End_{R/(\Delta)}(S/(J))$. \\
In particular, we have established a correspondence between the indecomposable projective $\Abar$-modules and the nontrivial irreducible $G$-representations on the one hand and the non-isomorphic $R/(\Delta)$-direct summands of $S/(J)$ on the other hand.\\
Moreover, $\Abar$ constitutes a noncommutative resolution of singularities (=NCR) of $R/(\Delta)$ of global dimension $n=\dim R +1$ for $G \neq \mu_2$. 
\end{thmd}

For  a true reflection group $G \leqslant \GL(2,\CC)$ this implies that $S/(J)$ is a representation generator of $\CM(R/(\Delta))$, and so recovers the fact that $R/(\Delta)$ is an ADE-curve singularity \cite{Bannai}, (cf.~Cor.~\ref{Thm:dim2}). \\


The remainder of the paper is dedicated to a more detailed study of  $S/(J)$ as $R/(\Delta)$-module, for any pseudo-reflection group $G \leqslant \GL(n,\CC)$: we determine the ranks of the isotypical components of $S/(J)$ over $R/(\Delta)$ using Hilbert--Poincar{\'e} series and can give precise formulas in terms of Young diagrams in the case $G=S_n$ (Prop.~\ref{RankisotypicalDiscriminant}). 
Then, using Solomon's theorem and results from Kyoji Saito and Hiroaki Terao we can identify some of  the isotypical components of $S/(J)$ (again for any pseudo-reflection group $G$): the isotypical component of the defining representation $V$ of $G$ and its higher exterior powers $\Lambda^lV$ are given by the cokernels of the natural inclusions $\Lambda^l \Theta_R(-\log \Delta) \xto{} \Lambda^l \Theta_R$ of the module of logarithmic derivations into the derivations on $R$, dubbed the \emph{ logarithmic co-residues}. In particular, for $l=1$ one gets that $j_{\Delta}$, the Jacobian ideal of the discriminant viewed as a module over $R/(\Delta)$, is a direct summand of $S/(J)$, see Thm.~\ref{logRes}. We also obtain the \emph{logarithmic residues} $\coker(\Lambda^{n-i}\mu: \Omega_R^{n-i} \rightarrow \Omega_R^{n-i}(\log \Delta))$ as direct summands of $S/(z)$.    The other isotypical components have yet to be determined in general. \\
The paper ends with the example of the discriminant of $G=S_4$ acting on $\CC^3$, the well-known \emph{swallowtail}. Here we can explicitly determine all matrix factorizations for the nonisomorphic direct summands of $S/(J)$.

\section{Dramatis Personae} \label{Sec:Basics}

\begin{tabular}{ll}
$K$\dotfill & an algebraically closed field, mostly $\CC$\\
$\car K$\dotfill &the characteristic of $K$\\
$V$ \dotfill& a finite dimensional vector space over $K$\\
$n=\dim_{K}V$ \dotfill& the dimension of $V$ over $K$\\
$G\leqslant \GL(V)\cong \GL(n,K)$ \dotfill& a finite subgroup of $K$--linear automorphisms of $V$\\
$\Gamma = G\cap \SL(V)$\dotfill& the kernel of the determinant homomorphism restricted to $G$ \\
$|G|$\dotfill & the order of $G$, assumed not to be divisible by $\car K$\\
$KG$\dotfill & the group algebra on $G$ over $K$. According to our assumption,\\
& a semi--simple $K$--algebra, product of matrix algebras over $K$ \\
$S=\Sym_{K}V$\dotfill & the symmetric algebra on $V$ over $K$\\
$R=S^{G}$\dotfill & the invariant subring of the action of $G$ on $V$ \\

$S^{G}=K[f_{1},\ldots, f_{n}]$\dotfill& the invariant subring when $G\leqslant\GL(V)$ is a subgroup\\ 
&generated by pseudo-reflections \\ 
$d_{i}=\deg f_{i}$\dotfill &the degrees of basic invariants, so that $|G|=d_{1}\cdots d_{n}$\\
$m=\sum_{i=1}^{n}(d_{i}-1)$\dotfill& the number of pseudo-reflections in $G$\\
$J=\det\left(\frac{\partial f_{i}}{\partial x_{j}}\right)_{i,j=1,\ldots,n}$\dotfill & the Jacobian determinant of the basic 
invariants that is\\
&a polynomial in $S$ of degree $m$\\
$z$\dotfill & the squarefree polynomial underlying $J$\\
$m_1=\deg z$\dotfill& the degree of $z$, that is, the number of mirrors in $G$ \\
$\Delta=zJ\in S^{G}$\dotfill &the discriminant of the reflection group $G$ that is thus\\
& of degree $m+m_1$\\
$V_{i}, i=0,\ldots,r,$\dotfill & representatives of the isomorphism classes of irreducible\\
&$G$--representations.\\
$V_{0}=K_{\trv}=\mathrm{triv}$\dotfill & the trivial representation\\
$V_{1}=V$\dotfill& the defining representation $G \hookrightarrow \GL(V)$ if that is irreducible\\
$V_{\mathrm{\det}}=\det V=|V|$\dotfill &the linear one-dimensional representation of $G$ afforded by the\\
&determinant of the
defining representation $V$\\
$\rank_{C} M$\dotfill &the rank function on the minimal primes in $\Spec C$ for a module\\ 
&$M$ over a reduced commutative ring $C$ \\

\end{tabular}

\subsection{Conventions} 
Throughout the paper let $K=\CC$,\footnote{Most of our results also hold if the characteristic of the field $K$ does not divide the order $|G|$ of the group $G$. However, in order to facilitate the presentation, we restrict to $K=\CC$.} if not explicitly otherwise specified.  
Let $V$ be a finite dimensional vector space over the field $K$ and
$\GL(V)$ the group of invertible linear transformations on it. 
If we choose a basis to identify $V\cong K^{n}$, we identify, as usual, 
$\GL(n)=\GL(n,K)\cong \GL(V)$ with the group of invertible $n\times n$ matrices over $K$.  
Further, let $G$ be a finite subgroup of linear transformations on $V$. The group $G$ acts then linearly and faithfully on the polynomial ring $S=\Sym_{K}V\cong K[x_{1},\ldots,x_{n}]$ over $K$,
where $x_{1},\ldots, x_{n}$ constitutes a $K$--basis of $V$. We may consider $S$ as a graded ring with standard grading $|x_i|=1$ for all $i$. If $s=f(\bx)\in S$, then we write $g(s) = f(g\bf x)$ for the action 
of $g\in G$ on $s$, with $\bx=(x_{1},\ldots, x_{n})$ and 
$g\bx =(g(x_{1}),\ldots, g(x_{n}))$. Note that if $g=(a_{ij})_{i,j=1,\ldots,n}\in \GL(V)$, then $g\bx = (a_{ij})(x_{1},\ldots, x_{n})^{t}$, where
$(-)^{t}$ denotes the transpose\footnote{Let us point out that many authors use $S=\Sym_{K}(V^*)$ with $g$ acting on $s=f(x)$ as $g(s)=f(g^{-1}(x))$.}. 

The \emph{invariant ring} of the action of $G$ on $V$ will be denoted by $R:=S^{G}=\{ s \in S: g(s)=s$ for all $g \in G \}$.

 \subsection{Twisted group rings} 

Assume that $G \leqslant \GL(V)$ is any finite subgroup. The group ring of $G$ will be denoted by $KG$. We denote by $Q=Q(S)$ the field of fractions of $S$ and note that $G$ acts on $Q$ as well. 
We consider the following $K$--algebras.

\begin{defn}
\label{skewgroupring}
Assume $G$ acts on a $K$--algebra $S$ through $K$--algebra automorphisms. 
The {\em twisted\/} or {\em skew group ring\/} defined by these data is 
$A=S{*}G$, where $A=S\otimes_{K}KG$ as a left $S$--, right $KG$--module, but the multiplication is twisted by the action of $G$ on $S$. 

In more detail, $A$ is the free left $S$--module with basis indexed by $G$, thus, 
$A=\bigoplus_{g\in G}S\delta_{g}$, where $\delta_{g}$ stands 
for the basis element parametrized by $g\in G$. 

The multiplication is defined by $\delta_{g}s = g(s)\delta_{g}$, for $s\in S, g\in G$. In particular the multiplication of two elements $s'\delta_{g'}, s\delta_{g}$ is given by
\[
(s'\delta_{g'})(s\delta_{g})=(s' g'(s)) \delta_{g'g} \in S\delta_{g'g}\quad \text{for}\quad g',g\in G, s',s\in S\,.
\] 
\end{defn}

Our notation here follows \cite{KKu2} and is meant to clearly distinguish, say, the element $\delta_gs\in A$ from the element 
$g(s)\in S$.

However, even if $S$ is commutative, its image is usually not in the centre of
$A$, whence the ring homomorphism $S\to A$ only endows $A$ with an $S$--bimodule structure over $K$,
with the action from the left simply multiplication in $S$, while the action from the right is determined by
$\delta_{g} s= g(s) \delta_{g}$ for $g\in G, s\in S$. In particular, each left $S$--module direct summand 
$S\delta_{g}\subseteq A$ is already an $S$--bimodule direct summand of $A$.

Similarly  $Q * G \cong Q\otimes_{S}A$ and we have ring homomorphisms 
$Q\to Q*G$ and $QG\to Q*G$, where $Q=Q(S)$. As noted in \cite[p.515]{Auslander86}  or in \cite[Sect.2]{KKu1}, \cite[4.1($I_{23}$)]{KKu2} the map
\begin{align*}
\tau\colon Q*G \lto Q *G \quad, \quad \tau(f \delta_{g}) = g^{-1}(f) \delta_{g^{-1}} \quad g\in G, f\in Q
\end{align*}
is an involutive algebra anti-isomorphism that restricts to an anti-isomorphism, denoted by the same symbol, 
$\tau\colon A\xto{\ \cong\ }A$. In particular, $A \cong A^{op}$ as $K$-algebras.

 \label{Sub:twistedgroupring}
If $|G|$ is invertible in $S$, we can set  $e=\frac{1}{|G|}\sum_{g \in G}\delta_g$. 
It is an idempotent element of $A$ and 
$A\left(\sum_{g\in G}\delta_{g}\right)A=AeA \subseteq A$ is an idempotent ideal in $A$.

\begin{lemma} \label{Lem:Se}
Let $e$ be the idempotent just introduced.
\begin{enumerate}[\rm(a)]
\item  The left multiplication $e(\ )\colon S\to A, s\mapsto es,$ 
yields an isomorphism of right $A$--modules $S\xto{\ \cong\ } eS=eA$.
\item   \label{Se2} The right multiplication $(\ )e\colon S\to A, s\mapsto se,$ 
yields an isomorphism of left $A$--modules $S\xto{\ \cong\ } Se=Ae$.
\item The (two--sided) multiplication $e(\ )e\colon R\to A, r\mapsto ere=er=re,$ yields an 
isomorphism of rings $R\xto{\ \cong\ } eAe$, where $R=S^G$ as defined above.
\item \label{Se4}
In the commutative squares
\begin{align*}
\xymatrix{
S\times R \ar[rrr]^-{(s,r)\mapsto sr}\ar[d]_-{(\ )e\times e(\ )e}^{\cong}
&&&S\ar[d]^{(\ )e}_{\cong}
&&&R\times S\ar[rrr]^-{(r,s)\mapsto rs}\ar[d]_-{e(\ )e\times e(\ )}^{\cong}
&&&S\ar[d]^{e(\ )}_{\cong}
\\
Ae\times eAe\ar[rrr]^-{(ae,ea'e)\mapsto aea'e}&&&Ae&&&
eAe\times eA\ar[rrr]^-{(ea'e,ea)\mapsto ea'ea}&&&eA
}
\end{align*}
the vertical maps are bijections, thereby identifying the right $eAe$--module $Ae$ with the (right)
$R$--module $S$ and the left $eAe$--module $eA$ with the (left) $R$--module $S$.
In particular, the induced map
\begin{align*}
S\otimes_{R}S\xto[\cong]{(\ )e\otimes e(\ )} Ae\otimes_{eAe}eA
\end{align*}
is an isomorphism of $A$--bimodules.\qed
\end{enumerate}
\end{lemma}
In this way, there is a natural homomorphism of rings
\begin{align} \label{Eq:AtoHomS}
- \otimes_{A}Ae : A &\cong \Hom_{A}(A,A)\longrightarrow \Hom_{eAe}(Ae,Ae)\cong \Hom_{R}(S,S)\, ,
\end{align}
also cf.~the calculations in \cite[p.515]{Auslander86}.

Moreover, taking invariants with respect to the above action of $G$ defines a functor $\Mod A\to \Mod R$ as the $G$--invariants 
form a (symmetric) $R$--module.  \\
For any left $A$-modules $M, N$, one has $\Hom_A(M,N) \cong \Hom_S(M,N)^G$, where $g \in G$ acts on an $S$-linear map $f: M \xrightarrow{} N$ through $(g \cdot f)(m)=g(f(g^{-1}(m)))$.   Taking invariants $(-)^G$ is an exact functor, whence also $\Ext^i_A(M,N) = \Ext^i_S(M,N)^G$ for all $i$. In particular, an $A$-module $M$ is projective if and only if the underlying $S$-module is projective.

\begin{lemma}[see Section 1 of \cite{Auslander86}] \label{Lem:ProjCorresp}
Let $S$ be a regular complete local ring or a graded polynomial ring. One has a functor 
\[ \alpha: P(A) \xto{}  \Mod KG, \ P \mapsto S/\fm_S \otimes_S P \ , \]
where $P(A)$ denotes the category of projective $A$-modules and $\fm_S$ denotes the maximal ideal (in case $S$ is local) or the maximal ideal $(x_1,\ldots, x_n)$ in $S$ (in case $S$ is a polynomial ring in $n$ variables 
One also has a functor $\beta$ in the other direction that sends a $KG$-module $V$ to $S \otimes_K V$. This pair of functors induces inverse bijections on the isomorphism classes of objects.\qed
\end{lemma}

\begin{remark} Auslander proved this result in the case where $S=K[[x,y]]$ the power series ring in two variables, a proof for the $n$-dimensional complete case can be found e.g.~in \cite{LeuschkeWiegand}. However, the correspondence also holds in the graded case, i.e., for graded modules over $S=K[x_1,\ldots, x_n]$ with $\deg x_i=1$. For this one uses Swan's theorem, see e.g.~\cite[XIV, Thm.~3.1]{Bass69}.
\end{remark}

\subsubsection{Quotients of $A$ by idempotent ideals} Let $\chi$ be the  character of an irreducible $G$--rep\-resentation. This defines the central primitive idempotent associated to this representation as $e_{\chi}=\frac{1}{|G|}\sum_{g \in G} \chi(g^{-1}) g$ in $KG \subseteq A$. If we want to stress that $e_\chi \in A$, then we write $e_{\chi}=\frac{1}{|G|}\sum_{g \in G} \chi(g^{-1}) \delta_g$.  In particular, denote $e:=e_{\mathrm{triv}}=\frac{1}{|G|}\sum_{g\in G}\delta_g\in A$ the idempotent associated to the trivial representation of $G$, $\overline{e}:=e_{\det^{-1}}=\frac{1}{|G|}\sum_{g\in G}\det(g)\delta_g$, the idempotent associated to the inverse determinantal representation. \\In the following we will be interested in the quotient algebra $A/Ae_{\chi}A$, where $e_\chi$ is an idempotent associated to a linear character $\chi$. The next two results show that the choice of the one-dimensional character does not matter and thus we will sometimes switch between $A/AeA$ and $A/A\overline{e}A$.

With $\Hom_{\mathrm{gps}}(G, K^*)$ the group of linear characters, consider the map $\alpha: \Hom_{\mathrm{gps}}(G, K^*) \xto{} \Aut_{K-\mathrm{Alg}}(A)$,
\[ \lambda \mapsto \alpha_\lambda \]
\[\alpha_\lambda\left(\sum_{g \in G} s_g \delta_g\right)= \sum_{g \in G} s_g \lambda(g^{-1})\delta_g  \  . \]

\begin{lemma} \label{Lem:charIdempot}
The map $\alpha$ is a homomorphism of groups. If $L$ is the one-dimensional representation defined by $\lambda$ and $\chi$ the character of some $G$-representation $W$, then $L \otimes W$ has character $\lambda \cdot \chi$. If $W$ is irreducible and $e_{\chi}=\frac{1}{|G|}\sum_{g \in G} \chi(g^{-1})g$ the corresponding idempotent in $KG$, then $\alpha_\lambda(e_{\chi})=e_{\lambda \cdot \chi}$. \qed
\end{lemma}

\begin{corollary} \label{Cor:quotientsiso}
Let $\lambda, \lambda'$ be  one-dimensional characters of $G$ with respective idempotents $e_{\lambda}, e_{\lambda'}$. Then the  quotient algebras $A/Ae_{\lambda}A$ and $A/Ae_{\lambda'}A$ are isomorphic $K$-algebras.  \qed
\end{corollary}

In the next lemma we state some useful properties of the quotient.  For this we recall the  following notion: Let $G \leqslant \GL(V)$ be a finite group and let $\chi$ be a linear character. An element $f \in S$ is a \emph{relative invariant for $\chi$} if $g(f)=\chi(g) f$ for all $g \in G$.  The set of relative invariants for $\chi$ is denoted by $S^G_\chi = \{ f \in S: g(f)=\chi(g) f$ for all $g \in G\}$, cf.~\cite{StanleyInvariants}. Clearly one has $S^G_{\mathrm{triv}}=S^G=R$.

\begin{lemma} \label{Lem:quotienrelative}
Let $G \leqslant \GL(V)$ be a finite group and let $\chi$ be a linear character. Assume that $S^G_\chi$ is a free $R$-module of rank $1$, that is, there exists a $f_\chi \in S$ such that $S^G_\chi=f_\chi R$. Then 
\[S/(f_\chi) \cong (A/Ae_\chi A)e\]
 as $S$-modules.
\end{lemma}

\begin{proof}
Denote $\overline{A}:=A/Ae_{\chi}A$. Applying $- _{A}\otimes Ae$ to the exact sequence 
$$ 0 \xrightarrow{} Ae_{\chi}A \xrightarrow{} A \xrightarrow{} \overline{A} \xrightarrow{} 0$$
yields the exact sequence (since $Ae$ is a flat $A$-module)
\begin{equation} \label{Eq:exactSz} 0 \xrightarrow{} Ae_{\chi}Ae \xrightarrow{} Ae \xrightarrow{}  \overline{A}e \xrightarrow{} 0.
\end{equation}

We have seen in Lemma \ref{Lem:Se} \eqref{Se2} that $Ae \cong Se$. 
Moreover $Ae_{\chi}Ae = (Sf_\chi)e \cong Sf_\chi$: for this we first use $Ae_{\chi}Ae = Ae_{\chi}Se$. Then using that $\delta_g e_{\chi}=\chi(g) e_\chi$, for an element $(\sum_{g \in G}  t_g \delta_g) e_{\chi} s e$ in $Ae_{\chi}Se$  we get
\begin{align*}
\sum_{g \in G} t_g \delta_g e_{\chi}se & = \left(\sum_{g \in G}\chi(g)t_{g}\right)e_{\chi}se=\left(\sum_{g \in G}\chi(g)t_{g}\right) \frac{1}{|G|}\sum_{h \in G}\chi(h^{-1}) h(s) \delta_{h}e \\ & = \left(\sum_{g \in G}\chi(g)t_{g}\right) \left(\frac{1}{|G|}\sum_{h \in G} \chi(h^{-1}) h(s)\right) e. 
\end{align*}
The element $\sum_{h \in G} \chi(h^{-1}) h(s)$ is a semi-invariant for $\chi$, so it is in the ideal in $R$ generated by $f_{\chi}$. Thus it follows that $Ae_\chi Ae \subseteq Sf_{\chi}e$. 
And  the element $f_{\chi}e=e_{\chi} f_{\chi} e$ is in $Ae_{\chi}Ae$, thus $Ae_{\chi}Ae \supseteq Sf_{\chi}e$.
This means that the sequence \eqref{Eq:exactSz} is isomorphic to
\[ 0 \xrightarrow{} Sf_{\chi}e \rightarrow Se \xrightarrow{}  \overline{A}  e \xrightarrow{} 0,
\]
which implies that  $ \overline{A} e \cong (S/(f_\chi))e\cong S/(f_\chi)$  as $S$-modules.
\end{proof}

\subsection{Reflection groups}
Here we recall some useful facts about complex reflection groups; see, for example, \cite{BourbakiLIE4-6,LehrerTaylor,OTe}. We mostly follow the notation in \cite{OTe}.

Recall that an element $g$ in $\GL(V)$, is
\begin{enumerate}[\rm(a)]
\item a {\em (true) reflection}, if it is conjugate to a diagonal matrix $\diag(-1,1,\ldots,1)$. 
In other words, as a linear transformation $g$ fixes a unique hyperplane $H\subset V$ pointwise and has additionally $-1\neq 1$ 
as an eigenvalue. We call any nonzero eigenvector for the eigenvalue $-1$ a {\em root\/} of the reflection. We think of it as a vector
``perpendicular'' to the hyperplane $H$. We call the hyperplane $H$ the \emph{mirror} of $g$.
\item a {\em pseudo-reflection}, if it is conjugate to a diagonal matrix $\diag(\zeta,1,\ldots,1)$, where $\zeta\neq -1$ is a root of unity in $K$.
Again we call the hyperplane $H=\ker(g - \mathrm{Id}_V)$ the {\em mirror\/} of $g$. 
\end{enumerate}

For a finite subgroup $G\leqslant \GL(V)$, the subgroup $G' \leqslant G$ generated by the pseudo-reflections is normal in $G$
as the conjugate of a pseudo-reflection is again a pseudo-reflection. For the same reason the subgroup $G'' \leqslant G$
generated by (true) reflections is normal in $G$, contained, of course, in $G'$.

One distinguishes now the extreme possibilities.
\begin{defn} \label{Def:subgroups}
Given a finite subgroup $G\leqslant \GL(V)$,
\begin{enumerate}[\rm(a)]
\item $G$ is {\em small\/} if it contains {\em no pseudo-reflections}, thus, $G'=1$.
\item $G$ is a {\em (true) reflection\/} group if it is {\em generated by its (true) reflections}, thus,
$G''=G$.
\item $G$ is a {\em complex reflection\/} or {\em pseudo-reflection\/} group if it is 
{\em generated by its pseudo-reflections}, thus, $G'=G$.
\end{enumerate}
\end{defn}

\begin{example}
Any finite subgroup of $\SL(V)$ is small, since it only contains elements with determinant $1$, that is, it does not contain any pseudo-reflections.
\end{example}

The ring $S^{G}$ is a  normal Cohen--Macaulay domain by the Hochster--Roberts Theorem \cite{HochsterRoberts}. 
If $G\leqslant \SL(V)$, then $S^{G}$ is Gorenstein and, conversely, if $G$ is small, then $S^{G}$ is Gorenstein only if 
${G}\leqslant \SL(V)$ according to a theorem by Kei-Ichi Watanabe \cite{WatanabeGorenstein}. Invariant rings of pseudo-reflection groups are distinguished by the following:

\begin{theorem}[Chevalley--Shephard--Todd] \label{Thm:CST} Let $G \leqslant \GL(V)$ be a finite group acting on $S$. Then the invariant ring $R=S^{G}$ is a polynomial ring itself, that is,
$R=K[f_{1},\ldots, f_{n}]\subseteq S$, where the $f_{i}$ are  algebraically independent homogeneous polynomials of degree 
  $d_{i}\geqslant 1$, if and only if $G$ is a pseudo-reflection group. Note that, equivalently, the $f_{i}$ form a homogeneous regular sequence in $S$.  \\
Moreover, if $G$ is a pseudo-reflection group, then $S$ is free as an $R$-module, more precisely $S \cong R \otimes_K KG$, as $G$-modules, where $KG$ denotes the group ring of $G$.
\end{theorem}

This was the second theorem in \cite{Che} and was as well generalized for pseudo-reflections in the separable case, see
 \cite[Thm. 6.19]{OTe}.

 \subsubsection{Reflection arrangement and discriminant}  \label{defnms} Let us now recall some facts regarding pseudo-reflection groups $G \leqslant \GL(V)$:
\begin{enumerate}[\rm (a)]
\item Finite pseudo-reflection groups over the complex numbers have been classified by 
Geoffrey C.~Shephard and John Arthur Todd \cite{STo}. They contain true reflection groups and thus all finite Coxeter groups, i.e., all finite groups that admit a realization as a reflection group over the real numbers. Coxeter groups are precisely those true reflection groups that have an invariant of degree $2$, see \cite{OTe}.
  Coxeter groups are moreover the pseudo-reflection groups for which $V$ is isomorphic to its dual $V^*$, see e.g.~\cite[Thm.~31]{Serre}.
\item  The polynomials $f_i$ in Theorem \ref{Thm:CST} are called the \emph{basic invariants} of $G$. They are not unique but their degrees $d_i$ are uniquely determined by $G$ and one has an equality $|G| = d_{1}\cdots d_{n}$ (for a proof of this fact see e.g. \cite{STo} or \cite[Ch.~5, \S 5, no.~3, Corollary to Theorem 3]{BourbakiLIE4-6}). Note that $S^G=K[f_1, \ldots, f_n]$ is a graded  polynomial $K$-algebra, with $\deg f_i=d_i$.

\item Let $H\subset V$ be the mirror of a pseudo-reflection $g_{H}\in G$ of order $\rho_{H}>1$. So for any $v \in V$, one has $g_H(v)=v+L_H(v)a_H$, where $a_H \in V$ and $L_H(v)$ is a linear form such that $H = \{v \in V \st L_H(v) =0 \}$.
The \emph{Jacobian} is defined as
\begin{align*}
J=Jac(f_{1},\ldots, f_{n}) &=\det
\left(\left(
\frac{\partial f_{i}}{\partial x_{j}}
\right)_{i,j=1,\ldots,n}\right) \ .
\end{align*}
One can show  that 
$$J=u \prod_{\mbox{mirrors }H}L_{H}^{\rho_{H}-1} \ , $$
where $u\in K^{*}$.  Therefore, each linear form $L_{H}$ occurs with multiplicity $\rho_{H}-1$. The degree of the Jacobian
is $m=\sum_{i=1}^{n}(d_{i}-1)$, which equals the number of pseudo-reflections in $G$ (see e.g. \cite[Ch.~5, \S 5, no.~5, Prop.~6]{BourbakiLIE4-6}). 

\item 
The differential form
\begin{align*}
df_{1}\wedge\cdots\wedge d_{f_{n}} = Jdx_{1}\wedge\cdots\wedge dx_{n}
\end{align*}
is $G$--invariant, whence $J$ transforms according to $gJ =(\det g)^{-1}J$. Thus, $JK$ is the one dimensional
{\em inverse determinant representation\/} of $G$.
 The element $z =\prod_{H} L_H$ is the reduced defining equation of the \emph{reflection arrangement} $\mathcal{A}(G)$ associated to $G$. It is easy to see that $z$ is a relative invariant for the linear character $\chi = \det$, that is, for all $g \in G$ we have $gz=\det(g) z$. The degree of $z$ is $m_1$, the number of mirrors of $G$. 
\item \label{defn:zJ}
The \emph{discriminant} of the group action is given by
\begin{align*}
\Delta = zJ = u \prod_{H \subset \mathcal{A}(G)}L_{H}^{\rho_{H}}\,,
\end{align*}
where $u \in K^*$. The polynomial $\Delta$ is an element of $S^G$ of degree $\sum_{\kappa}\rho_{\kappa}=m+m_1$, see e.g. \cite[Def.~6.44]{OTe}. The discriminant polynomial $\Delta \in S^G$ is always reduced (this follows e.g.~from Saito's criterion and the fact that $\Theta_S^G \cong \Theta_S(-\log \Delta)$, see \cite[Chapter 6]{OTe} for statements and notation). In particular, if $G$ is a true reflection group, then $\rho_H=2$ for all $H$, and thus $J=z$ (up to unit) and $z^{2}=\Delta$ 
represents the discriminant (also see Remark \ref{Rmk:unitDisc}). 
\item The preceding in geometric terms: if $G$ is a pseudo-reflection group, then the quotient $V/G=\Spec(S^G)$ is an affine regular variety isomorphic to $V \cong \mathbb{A}^n(K)$. Under the natural projection
\[ \pi: V\cong \Spec(S) \longrightarrow V/G \cong \Spec(S^G) \]
 the image of the hyperplane arrangement $\mathcal{A}(G)$ is the discriminant  hypersurface $V(\Delta) \subseteq V/G$. 
 \item  The discriminant $V(\Delta)$ in $V/G$ and the hyperplane arrangement $\cala(G)$ in $V$ are both \emph{free divisors}. This means that the module of logarithmic derivations $\Theta_{R}(-\log \Delta)=\{ \theta \in \Theta_R: \theta(\Delta) \in (\Delta)R \}$ is a free $R=S^G$-module and accordingly $\Theta_S(-\log z)$ is a free $S$-module. This was first shown by Kyoji Saito for Coxeter groups, cf.~\cite{SaitoReflexion} and by Hiroaki Terao for complex reflection groups \cite{Terao-free-complex} .
\end{enumerate}

\begin{remark} \label{Rmk:unitDisc}
For ease of notation we will consider the polynomials $$J':=u^{-1}J=\prod_{\mbox{mirrors }H}L_{H}^{\rho_{H}-1} \quad  \text{ and } \quad \Delta':=zJ'=\prod_{\mbox{mirrors }H}L_{H}^{\rho_{H}}$$ instead of $J$ for the Jacobian and $\Delta$ for the discriminant, since they generate the same ideals in $S$ resp. $S^G$. In abuse of notation we will also denote them with $J$ and $\Delta$. For true reflection groups we will then have $z=J$ and $J^2=\Delta$.
\end{remark}

\begin{example}
The true reflection groups $G \leqslant \GL(2,\CC)$ are classified via the ADE-Coxeter-Dynkin diagrams. The discriminant $\Delta$ of such a $G$ is the corresponding ADE-curve singularity, cf.~e.g. \cite[Section 3]{KnoerrerCurves}. For example, the $A_2$-curve singularity $K[x,y]/(x^3-y^2)$ is the discriminant of the group $S_3$ acting on $\CC^2$, see Fig.~\ref{Fig:cusp}.

\begin{figure}[!h]   
\begin{tabular}{c@{\hspace{1.cm}}c}
\includegraphics[width=0.4 \textwidth]{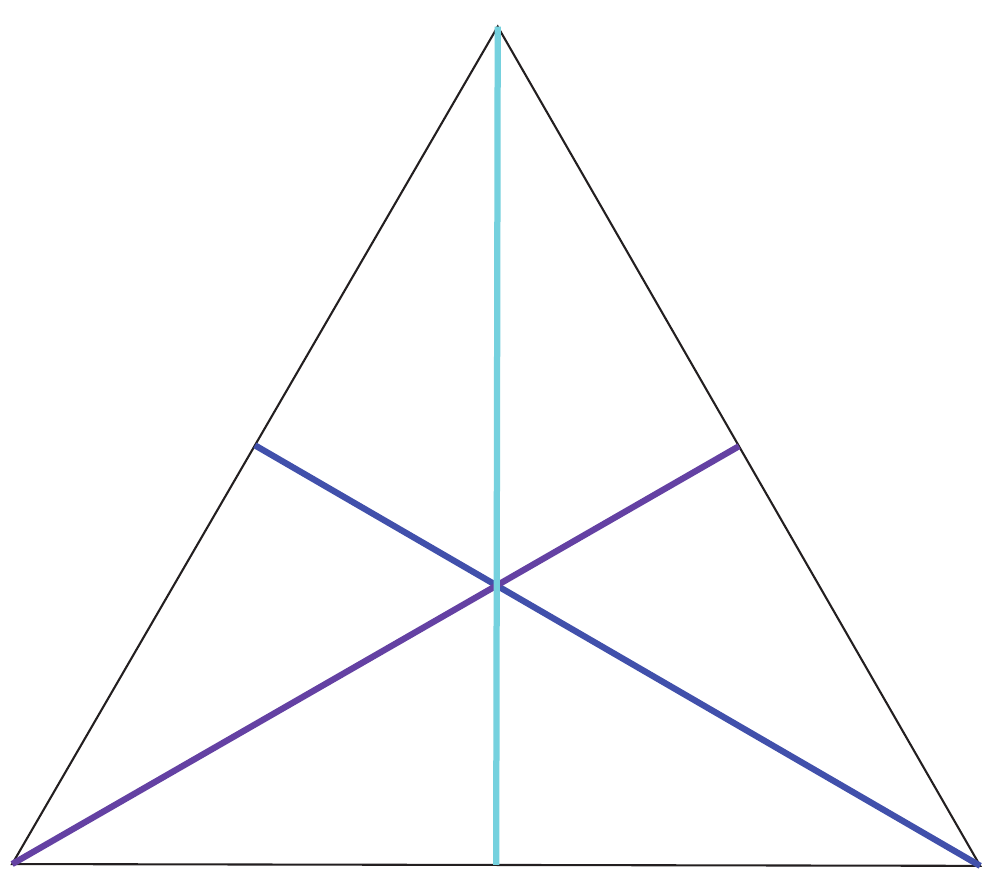}& 
\setlength{\unitlength}{1mm}
\begin{picture}(50,40)
\put(30,8){\textcolor{blue}{\qbezier(8,0)(6,16)(0,16)\qbezier(0,16)(6,16)(8,32)}}
\put(25,13){\makebox(0,0){\footnotesize$\displaystyle \Delta=y^2-x^3$}}
\end{picture}
\end{tabular}
\caption{The three lines of the hyperplane arrangement of $S_3$ and the discriminant $\Delta$ on the right.}
\label{Fig:cusp}
\end{figure}
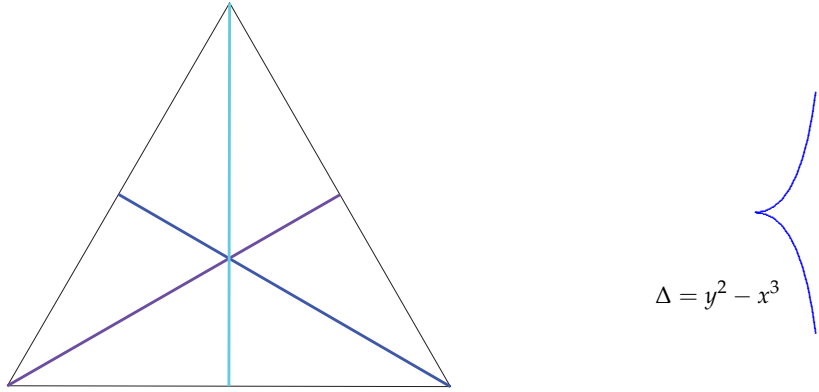
 
Etsuko Bannai calculated all discriminants for complex reflection groups $G \leqslant \GL(V)$, for $\dim V =2$ in \cite{Bannai}. In particular one sees from this list that all discriminants of reflection groups in $\GL(V)$ are curves of type ADE.
\end{example}

\begin{example}
The true reflection group $G_{24} \leqslant \GL(3,\CC)$ is a complex reflection group of order $336$ that comes from Klein's simple group, see \cite{OTe} ex. 6.69, 6.118\footnote{In Ex.~6.118 in \cite{OTe} the sign in front of $256x^7z$ is erroneous.} for more details. The reflection arrangement $\mathcal{A}(G_{24})$ consists of $21$ hyperplanes. In loc.~cit.~the basic invariants for this group, and the discriminant matrix are determined. One obtains the equation of the discriminant $\Delta$ as the determinant of the discriminant matrix, see Fig.~\ref{Fig:G24}. The discriminant $V(\Delta)$ is a non-normal hypersurface in $\CC^3$, whose singular locus consists of two singular cubic curves meeting in the origin.
\end{example}

\begin{figure}[!h]   \label{Fig:G24}
\begin{tabular}{c@{\hspace{1.cm}}c}
\includegraphics[width=0.4 \textwidth]{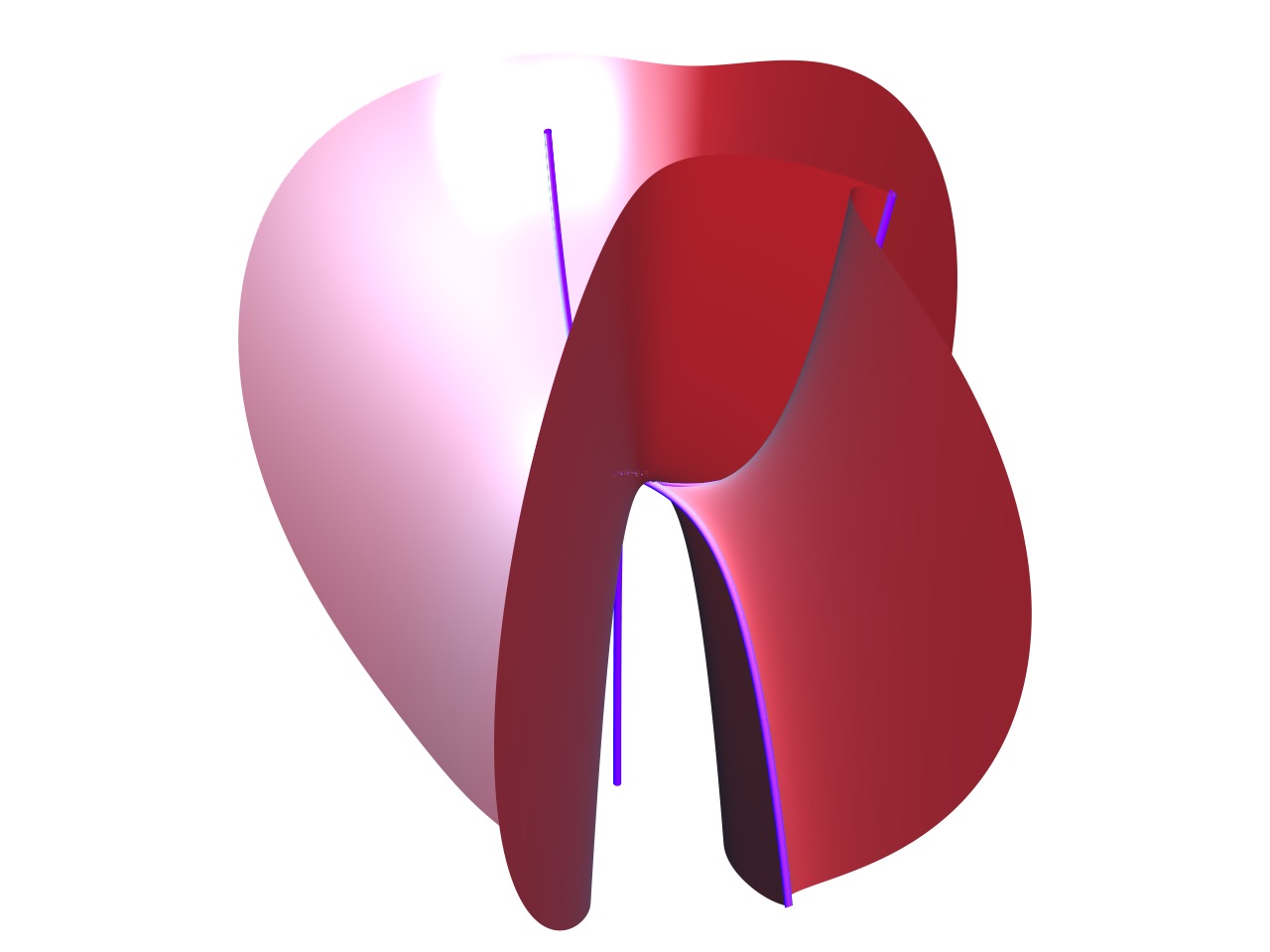}& 
\includegraphics[width=0.4 \textwidth]{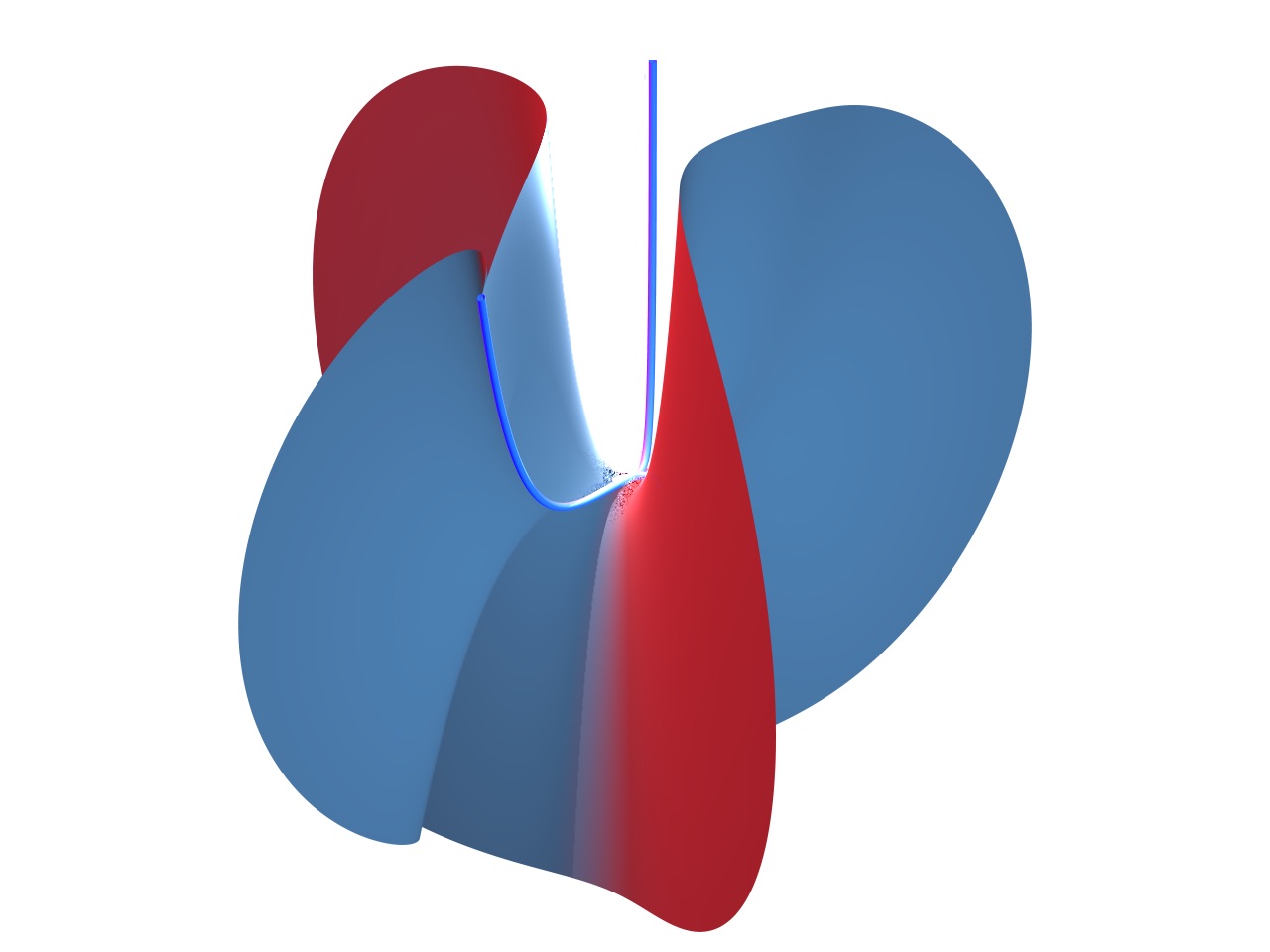}\end{tabular}
\caption{Two views of the discriminant of the group $G_{24}$ realized in $\RR^3$ with equation $\Delta=-2048x^9y+22016x^6y^3-256x^7z-60032x^3y^5+1088x^4y^2z+1728y^7+1008xy^4z-88x^2yz^2+z^3=0$. 
}
\end{figure}

\subsection{Isotypical components}  \label{Sub:isotypical}
\label{Sub:isotypical} Let $G \leqslant \GL(V)$ be a pseudo-reflection group, and adopt the notation from the last subsection for $R, S, z,J,$ and $\Delta$. 
Note that $R=S^G=K[f_1,\ldots,f_n]$ and $R/(\Delta)$ are graded rings with $\deg f_i = d_i$, the degrees of the basic invariants. The decomposition of $S$ as an $R$-module is given as follows: let $R_{+}$ be the set of invariants of $G$ with zero constant term, sometimes called the \emph{Hilbert ideal}. Then $S/(R_{+})$ is called the coinvariant algebra (here $(R_{+})$ denotes the ideal in $S$ generated by elements in $R_+$) and by the Theorem of Chevalley--Shephard--Todd (Thm.~\ref{Thm:CST}) one has
\[ S \cong R \otimes_K S/(R_{+})  \]
as graded $R$-modules. As $KG$-modules: 
\[ S\cong R \otimes_K KG \ .\]

With notation as above, one has the following simple fact.
\begin{lemma}\label{isotypical}
  Let $G$ be a finite group and $M$ a $K G$--module. Suppose that $r$ is the class number of $G$, i.e. the number of conjugacy class of $G$ or equivalently the number of isomorphism classes of irreducible representations of $G$. For $V_{i}$ an irreducible $G$--representation, the functors
  $\Hom_{K G}(V_{i},-)$ and $(-) \otimes_{K}V_{i}$ are adjoint.
We write  
$$\ev_{V_{i}}:\Hom_{K G}(V_{i}, M)\otimes_{K}V_{i}\to M$$
 for the evaluation map, which is the natural transformation of the composition of these functors to the identity functor.
The map $\ev_{V_{i}}$ is a split monomorphism of $K G$--modules, where $G$ acts on $\Hom_{KG}(V_{i}, M)\otimes_{K}V_{i}$ 
through the second factor. Its image is the isotypical component of $M$ of type $V_{i}$.
The sum of the evaluation maps, 
\[
\sum_{i=1}^{r}\ev_{V_{i}}:\bigoplus_{i=1}^{r}\Hom_{KG}(V_{i}, M)\otimes_{K}V_{i}\xto{\ \cong\ } M\,,
\]
is an isomorphism of $KG$--modules.\qed
\end{lemma}\label{lemma:directsumDecomp}
\begin{lemma} \label{Lem:SMCM}
If $M$ is a projective module over the skew group ring $S*G$, then
each $\Hom_{KG}(V_{i}, M)$ is a maximal Cohen--Macaulay module over $R=S^{G}$. 
\end{lemma}

\begin{proof}
If $M$ is projective over $S*G$, then by definition, $M$ is also a projective $S$-module. Since $S$ is a CM-module over $R$ (see e.g., \cite[Prop.~5.4]{LeuschkeWiegand}), also $M$ is CM over $R$. By Lemma \ref{isotypical}, $M \cong \bigoplus_{i=1}^{r}\Hom_{KG}(V_{i}, M)\otimes_{K}V_{i}$ and each of the $\Hom_{KG}(V_{i}, M)\otimes_{K}V_{i}$ is a module over $R$. This implies that each $\Hom_{KG}(V_{i}, M)$ is CM over $R$.  \end{proof}

  Thus, we recover the well known
  decomposition as $G$-representations, see \cite{Auslander86}:
\[\label{eqn:Sdecomp} S  \cong \bigoplus_{i=1}^r \Hom_{KG}(V_i,S) \otimes_K V_i = \bigoplus_{i=1}^r S_i \otimes_K V_i \ , \]
with notation $S_i:=\Hom_{KG}(V_i,S)$. By Lemma \ref{Lem:SMCM}, each $S_i \otimes_K V_i$ is CM over $R$.

 The Jacobian $J\in S$ 
is an element of the isotypical component of $S$ to the inverse determinantal representation $\det^{-1}$, while
$z\in S$  is an element of the isotypical component of $S$ of the determinantal representation $\det$
of $G$, and, as $S$ is a free $R$--module, the pair $(J,z)$ constitutes, trivially, a matrix factorization of 
$\Delta\in R$. 

As $J$ and $z$ are relative invariants for $G$, multiplication with these elements on $S$ is
$G$--equivariant. More precisely, multiplication with $J, z$, respectively, yields for each $V_{i}$
a graded $G$--equivariant matrix factorization. For compact notation, set $V_{i}'= V_{i}\otimes \det$, which is 
again an irreducible $G$--representation. Further recall that the degrees of $J$ and $z$ resp., are $m$ and $m_1$. Now look at the exact sequence 
\[0 \xrightarrow{}S (-m) \otimes \textrm{det$^{-1}$}  \xto{J} S \xto{} S/(J) \xto{} 0.
\]
Apply $ \Hom_{KG}(V_i,-) $ 
to get
\[ 0 \xrightarrow{}\Hom_{KG}(V_i,S (-m) \otimes \textrm{det$^{-1}$})  \xto{} S_i \xto{} \Hom_{KG}(V_i,S/(J)) \xto{} 0 \ .\]
Here $\Hom_{KG}(V_i,S (-m) \otimes \det^{-1}) \cong \Hom_{KG}(V_i \otimes \det,S)(-m)$. If we set as well $S_{i}'= \Hom_{K G}(V_{i}',S)$ this is $S_i'(-m)$.  Now denoting $\Hom_{KG}(V_i,S/(J))=M_i$, 
 we have short exact sequences of graded $R$--modules 
\begin{align}  \label{Eq:isotypical}
\xymatrix{
0\ar[r]&S_{i}'(-m)\ar[r]^-{J}&S_{i}\ar[r]&M_{i}\ar[r]&0 \\
0\ar[r]&S_{i}(-m-m_1)\ar[r]^-{z}&S_{i}'(-m)\ar[r]&N_{i}\ar[r]&0
}
\intertext{with $N_i=\Hom_{KG}(V_i, S/(z))(-m)$. Here the second one comes from the exact sequence}
\xymatrix{
0\ar[r]&S(-m-m_1) \ar[r]^-{z}& S  \otimes \mathrm{\det^{-1}} \ar[r]&S/(z) \ar[r]&0
}
\intertext{We also have the exact sequences} \label{eqn:ses}
\xymatrix{
0\ar[r]&N_{i}\ar[r]&S_{i}\otimes_{R}R/(\Delta)\ar[r]&M_{i}\ar[r]&0\\
0\ar[r]&M_{i}(-m-m_1)\ar[r]&S_{i}'\otimes_{R}R/(\Delta)\ar[r](-m)&N_{i}\ar[r]&0
}
\end{align}
which are already short exact sequences of maximal Cohen--Macaulay $R/(\Delta)$--modules.

To sum up this discussion, we can state the following
\begin{lemma} 
We have the direct sum decompositions: 
\[
S/(J) \cong \bigoplus_{i=0}^{r}M_{i}\otimes_{K}V_{i}\quad\text{and}\quad
S/(z)\cong \bigoplus_{i=0}^{r}N_{i}(m)\otimes_{K}V_{i}'
\]
as graded $R/(\Delta){-}KG$--modules. If $\Delta$ is irreducible it follows that 
\[
\rank_{R/(\Delta)}M_{i}+\rank_{R/(\Delta)}N_{i} = \dim_{K}V_{i}=\rank_{R}S_{i}=\rank_{R}S_{i}'\,.
\] 
\end{lemma}
\begin{proof} The direct sum decomposition follows from Lemma~\ref{lemma:directsumDecomp}, and the above discussion.  The second statement follows from the short exact sequences~\eqref{eqn:ses}
  above.
\end{proof}
 \begin{example} \label{Ex:isotypical-triv}
Consider the representation $V_{\mathrm{triv}}$ (instead of indexing the representations by $V_i$ we index $V_{\rho}$ by a specific representation $\rho$) and thus $V'_{\mathrm{triv}}=V_{\det}$. Then the exact sequence \eqref{Eq:isotypical} looks as follows
 \[ 0 \xto{} Rz(-m)\xto{J} R \xto{} R/(\Delta) \xto{} 0 \ , \]
 since $S'_{\mathrm{triv}} \cong Rz$ and $S_{\mathrm{triv}} \cong R$. This means that $M_{\mathrm{triv}}=R/(\Delta)$ and shows that $R/(\Delta)$ is a direct summand of $S/(J)$. \\
 For $V_{\det^{-1}}$  on the other hand we obtain from \eqref{Eq:isotypical}
  \[ 0 \xto{} R(-m) \xto{J} RJ \xto{}  0 \ . \]
  Thus $M_{\det^{-1}}=0$ and the inverse determinantal representation does not contribute a $R$-direct summand of $S/(J)$. Note that if $G$ is a true reflection group, then $\det=\det^{-1}$.
 \end{example}

\subsection{Endomorphism rings and Auslander's theorem}
One of the key results by Maurice Auslander in \cite[p.515]{Auslander86} asserts that the ring homomorphism \eqref{Eq:AtoHomS} from $A \xto{} \End_{R}(S)$
 is an isomorphism if 
$G$ is small, for a detailed proof see e.g. \cite[Ch.~5,Thm 5.15]{LeuschkeWiegand} or \cite[Thm 3.2]{IyamaTakahashi}:

\begin{theorem}[Auslander] \label{thm:Auslander}
Let $S$ be as above and assume that $G \leqslant \GL(V)$, with $\dim V=n$, is small and set $R=S^G$. Then we have an isomorphism of algebras:
$$ A=S * G \xrightarrow{\cong} \End_R(S) \ , \  s \delta_g \mapsto (x \mapsto sg(x)) \ .$$
Moreover, $S* G$ is a CM-module over $R$ and $\gl (S*G)=n$.
\end{theorem}

\begin{remark} 
By an obvious calculation, one sees that the centre $Z(A)=R$.
\end{remark}

\subsection{Noncommutative resolutions of singularities and the McKay correspondence} \label{Sec:ClassicalMcKay}


A resolution of singularities of an affine scheme $X=\Spec(R)$ is a proper birational map $\pi: \widetilde X \xto{} X$ from a smooth scheme $\widetilde X$ to $X$ such that $\pi$ is an isomorphism over the smooth points of $X$. Noncommutative resolutions of singularities of a ring $R$ (or of $\Spec(R)$) are certain noncommutative $R$-algebras that should provide an algebraic analog of this geometric notion. For the rationale behind the definition and more background about noncommutative (crepant) resolutions see \cite{Leuschke12, vandenBergh04,BFI-ICRA}. 

\begin{defn}
Let $R$ be a commutative noetherian ring. Let $M$ be a finitely generated $R$-module with $\supp M = \Spec(R)$. Then $\Lambda=\End_RM$ is called a \emph{noncommutative resolution (NCR)} of $R$ if $\gl \Lambda < \infty$.  \\
If $\Lambda$ is any finitely generated $R$-algebra that is faithful as $R$-module and $\gl \Lambda < \infty$, then we call $\Lambda$ a \emph{weak NCR} of $R$.
Note that in the case of a weak NCR we do not require that $\Lambda$ is an Endomorphism ring or even an $R$-order.
\end{defn}

\begin{remark} \label{Rmk:NCR}
In Michel Van den Bergh's original treatment \cite{vandenBergh04}, a \emph{noncommutative crepant resolution (=NCCR)} was defined over a Gorenstein domain. With our definition above, a NCCR over a commutative noetherian ring $R$ is an NCR that is additionally a nonsingular order over $R$. The (weak) NCRs constructed in this paper are (almost) never nonsingular orders: by definition if a finitely generated $R$-algebra $\Lambda$ is a nonsingular $R$-order, then $\gl(\Lambda)_{\mathfrak{p}}=\dim R_{\mathfrak{p}}$ for all $\mathfrak{p} \in \Spec(R)$. This implies in particular that $\gl \Lambda = \dim R$. But our NCRs are of global dimension $\dim R+1$. For more detail see Remark \ref{Rmk:mu2} and Cor.~\ref{Cor:NCRdisc}.   \\
NCRs were first defined in \cite{DaoIyamaTakahashiVial} over normal rings, we use here the more general definition of \cite{DFI}.
\end{remark}

In particular, Auslander's theorem can be reformulated in terms of noncommutative resolutions, cf.~\cite{vandenBergh04, IyamaWemyss10}:

\begin{theorem}
Let $G \leqslant \GL(V)$ small. Then $A=S*G$ yields a NCCR over $R=S^G$, that is, $A \cong \End_RS$ has global dimension $n$ and is a nonsingular order over $R$. 
\end{theorem}

For more details and information on the classical McKay correspondence we refer to the literature \cite{BuchweitzMFO,BFI-ICRA,GonzalezSprinbergVerdier,ReidBourbaki}.

\section{The Geometry} \label{Sec:Geometry}

\subsection{Some general facts on group actions}

We begin with the following general results on group actions that we quote from Bourbaki.
\begin{proposition}[{\cite[V.1.9 Cor.]{BourbakiAC}}]
Let $G$ be a finite group that acts through ring automorphisms on a commutative integral 
domain $S$. 
The group then acts as well through automorphisms on the field of fractions $Q(S)$ of $S$ 
and the fixed field $Q(S)^G$ is the field of fractions of the invariant integral subdomain 
$R=S^{G}$, that is, $Q(R)\cong Q(S)^{G}$.\qed
\end{proposition}

In the setting of the preceding Proposition, a crucial role will be played by the map 
$\vp:S\otimes_{R}S\to \Maps(G,S)$ given by 
$$\vp\left(\sum_{i=1}^{m}x_{i}\otimes y_{i}\right)(g) = \sum_{i=1}^{m}x_{i}g(y_{i})\in S $$
 with $(x_{i}, y_{i})\in S\times S, $ for  $i=1,
 \ldots,m,$ a finite family of pairs from $S$.  
Both source and target of this map are naturally $R$--modules and $\vp$ is $R$--linear 
with respect to these structures. 

Moreover, identifying naturally $Q(R)\otimes_{R}(S\otimes_{R}S) \cong Q(S)\otimes_{Q(R)}Q(S)$ 
and $Q(R)\otimes_{R}\Maps(G,S)\cong \Maps(G,Q(S))$, the induced map 
$\psi=Q(R)\otimes_{R}\vp$ of vector spaces over $Q(R)$ identifies with 
$\psi\left(\sum_{i=1}^{m}x_{i}\otimes y_{i}\right)(g) = \sum_{i=1}^{m}x_{i}g(y_{i})\in Q(S)$ for 
$(x_{i}, y_{i})\in Q(S)\times Q(S)$ a finite family of pairs from $Q(S)$.

Galois descent then yields the following fact%
\footnote{This result has also been called
``a strong form of Hilbert's Theorem $90$''; see \url{https://math.berkeley.edu/~ogus/Math_250A/Notes/galoisnormal.pdf}}:

\begin{proposition}[{\cite[V.\S10, no.4, Cor.~of Prop.~8]{BourbakiALG}}]
\label{prop:Galoisdescent}
If $G$ is a finite subgroup of the group of ring automorphisms of a commutative integral 
domain $S$, 
then the map 
\[
\psi\colon Q(S)\otimes_{Q(R)}Q(S)\lto \Maps(G,Q(S))
\] 
is bijective.\qed
\end{proposition}

\subsection{The structure of $\vp$} (See also \cite{Watanabe} for the material of this subsection.)\\
To study $\vp$ further, note next that with respect to the natural $R$--algebra structure on 
$S\otimes_{R}S$ and the diagonal $R$--algebra structure on $ \Maps(G,S)\cong S^{|G|}$, 
endowed with the componentwise operations, the map $\vp$ is a homomorphism of $R$--algebras.\\
Let $\ev_{g}:  \Maps(G,S)\to S$ be the evaluation at $g\in G$, so that
\[
\ev_{g}\vp\left(\sum_{i=1}^{m}x_{i}\otimes y_{i}\right) =  \sum_{i=1}^{m}x_{i}g(y_{i})\in S
\]
yields an $R$--algebra homomorphism $\ev_{g}\vp\colon S\otimes_{R}S\to S$.
\begin{lemma}
\label{kervp}
With notation as just introduced and with hypotheses as in Proposition {\em \ref{prop:Galoisdescent}}, one has
\begin{enumerate}[\rm (a)]
\item
\label{vp1}
For each $g\in G$, the $R$--algebra homomorphism $\ev_{g}\vp$ is surjective with kernel the prime ideal
\begin{align*}
I_{g} =(1\otimes s - g(s)\otimes 1; s\in S)\subseteq S\otimes_{R}S\,.
\end{align*}
\item 
\label{vp2}
The family of $R$--algebra homomorphisms $(\ev_{g})_{g\in G}$ identifies $\Maps(G,S)$ with the 
$S\otimes_{R}S$--algebra $\prod_{g\in G}(S\otimes_{R}S)/I_{g}$.

\item
\label{vp3} 
The kernel of $\vp$ equals $\cap_{g\in G}I_{g}$. The image of $\vp$ is isomorphic to the reduced $R$--algebra 
$\Imm\vp \cong (S\otimes_{R}S)/\cap_{g\in G}I_{g}$.

\item
\label{vp4} 
The ideal $\cap_{g\in G}I_{g}$ is the nilradical of the ring $S\otimes_{R}S$.

\item
\label{vp5} 
The kernel and cokernel of $\vp$ are $R$--torsion modules.
\end{enumerate}
\end{lemma}

\begin{proof}
(\ref{vp1}) For any $x\in S$, one has $\ev_{g}\vp(x\otimes 1) = x\in S$, whence 
$\ev_{g}\vp$ is surjective. 
Its kernel is as claimed due to the following standard argument: 
Clearly, $I_{g}\subseteq \Ker(\ev_{g}\vp)$, and if 
$\ev_{g}\vp\left(\sum_{i=1}^{m}x_{i}\otimes y_{i}\right) =  \sum_{i=1}^{m}x_{i}g(y_{i}) =0$ 
in $S$, then 
\begin{align*}
\sum_{i=1}^{m}x_{i}\otimes y_{i}&= \sum_{i=1}^{m}(x_{i}\otimes y_{i} - x_{i}g(y_{i})\otimes 1) =
 (x_{i}\otimes 1)\sum_{i=1}^{m}(1\otimes y_{i}-g(y_{i})\otimes 1) \in I_{g}\,.
\end{align*}
Because $S\otimes_{R}S/I_{g}\cong S$ is a domain, $I_{g}\subseteq S\otimes_{R}S$ is prime.

Regarding (\ref{vp2}), note that $(\ev_{g})_{g\in G}:\Maps(G,S)\to \prod_{g\in G}S$ is 
bijective by definition and (\ref{vp1}) reveals the $S\otimes_{R}S$--algebra structure 
induced by $\vp$ on $\Maps(G,S)$.

The first part of assertion (\ref{vp3}) is an immediate consequence of (\ref{vp1}) as $\vp=(\ev_{g}\vp)_{g\in G}$. 
The second assertion then follows.

As concerns (\ref{vp4}), Proposition \ref{prop:Galoisdescent} shows that $S\otimes_{R}S$ 
and its image have the same reduction. As the image is reduced, the claim follows 
--- see alternatively \cite[Lemma 2.5]{Watanabe} for a direct argument.

Likewise, (\ref{vp5}) follows from Proposition \ref{prop:Galoisdescent}.
\end{proof}

\begin{remark}
It seems worthwhile to point out the following consequence of (\ref{vp2}) above:
For $f\in \Maps(G,S)$, the map $sfs' =\vp(s\otimes s')f\in \Maps(G,S)$ is given by
$sfs'(g)= sg(s')f(g)$ for $g\in G$. In particular, even though $\Maps(G,S)\cong S^{|G|}$ as a ring,
it is not a symmetric $S$--bimodule when viewed as a $S$--bimodule via $\vp$.
\end{remark}

\subsection{The geometric interpretation of $\vp$}
With $X=\Spec S$, the reduced and irreducible affine scheme defined by the integral domain $S$, 
the scheme $Y=\Spec R$ identifies with the orbit scheme $Y=X/G$ of $X$ modulo the action of $G$. 
The canonical map $X\to Y$ corresponds to the inclusion $R\subseteq S$ and 
$\Spec(S\otimes_{R}S) \cong X\times_{Y}X\subseteq X\times X$ 
is the (schematic) graph of the equivalence relation defined by the action of $G$ on $X$. 

For $g\in G$, one may identify $\Spec(S\otimes_{R}S/I_{g})$ with the image of the map $(g,x)\mapsto (x,g(x))$ for $x\in X$, that is
$$\Spec(S\otimes_{R}S/I_{g}) \cong \Imm(\Spec(\ev_{g}\vp): 
\{g\}\times X  \to X\times_{Y} X) \ .$$
The map $\vp$ corresponds then to 
$$\Spec(\vp):G\times X =\coprod_{g\in G}\{g\}\times X\to X\times_{Y}X \ . $$
Proposition \ref{prop:Galoisdescent} says that this morphism of schemes is generically an 
isomorphism, and its image is the graph of the group action, 
$$GX := \bigcup_{g\in G}\Imm(\Spec(\ev_{g}\vp): \{g\}\times X \subseteq X\times_{Y}X)$$ 
with its reduced structure. Moreover, $GX=(X\times_{Y}X)_{\red}$ is the reduced underlying 
scheme of $X\times_{Y}X$, so that the only difference between $GX$ and $X\times_{Y}X$ 
can be embedded components in $X\times_{Y}X$.

\subsection{Interpretation in terms of the twisted group algebra}
The foregoing facts admit the following interpretation in terms of the twisted group algebra 
$A=S*G$, where $G$ is still a finite subgroup of the group of ring automorphisms of the
commutative domain $S$.

\begin{proposition} \label{Prop:commdiagA}
With the assumptions just made, the map $x\otimes y\mapsto \sum_{g\in G}x\delta_{g} y$ 
defines a surjective 
homomorphism $\alpha: S\otimes_{R}S \onto A\left(\sum_{g\in G}\delta_{g}\right)A$ of 
$S$--bimodules, while
$$\beta:\Maps(G,S)\to A, (s_{g})_{g\in G}\mapsto \sum_{g\in G}s_{g}\delta_{g}$$ is an 
isomorphism of $S$--bimodules over $R$ when $\Maps(G,S)$ is viewed as an 
$S$--bimodule via $\vp$. Note, however, that $\beta$ is clearly not an isomorphism of algebras.

There is a commutative diagram of $S$--bimodule homomorphisms
\begin{align} \label{Diag:AmAeA}
\xymatrix{
&&A\left(\sum_{g\in G}\delta_{g}\right)A\ar@{^(->}[r]&A\\
0\ar[r]&\bigcap_{g\in G}I_{g}\ar[r]&S\otimes_{R}S\ar@{{>>}}[u]^{\alpha}\ar[r]^-{\vp}&
\Maps(G,S)
\ar[u]^{\cong}_{\beta}\,,
}
\end{align}
with the bottom row an exact sequence.
In particular, as $S$--bimodules
\[
A\left(\sum_{g\in G}\delta_{g}\right)A\cong (S\otimes_{R}S)/\cap_{g\in G}I_{g}\,.
\]
\end{proposition}

\begin{proof}
It is clear that $\alpha$ is a homomorphism of $S$--bimodules with respect to the natural 
$S$--bimodule structures on $A$ and its ideal $A\left(\sum_{g\in G}\delta_{g}\right)A$. 
It is surjective as for $a= \sum_{h,h'\in G} s_{h}\delta_{h}$ and 
$a'= \sum_{h'\in G}\delta_{h'}s'_{h'}$ in $A$ one has 
\begin{align*}
a\left(\sum_{g\in G}\delta_{g}\right)a' &= 
\sum_{h,h'\in G} s_{h}\delta_{h}\bigg(\sum_{g\in G}\delta_{g}\bigg)\delta_{h'}s'_{h'}
=\sum_{h,h'\in G} s_{h}\bigg(\sum_{g\in G}\delta_{g}\bigg)s'_{h'}
=\alpha\bigg(\sum_{h,h'\in G}s_{h}\otimes s'_{h'}\bigg)\,,
\end{align*}
because $\delta_{h}\left(\sum_{g\in G}\delta_{g}\right)\delta_{h'}= \sum_{g\in G}\delta_{hgh'} 
= \sum_{g\in G}\delta_{g}$ for any $h,h'\in G$.

Note that establishing $\beta$ as an isomorphism of $S\otimes_{R}S$--modules uses 
Lemma \ref{kervp}(\ref{vp2}).

By definition of the various objects and morphisms we have
\begin{align*}
\alpha(x\otimes y)= \sum_{g\in G}x\delta_{g} y = \sum_{g\in G}xg(y)\delta_{g} =\beta\vp(x\otimes y)
\end{align*}
in $A$, whence the commutativity of the square in the diagram.

What we have established so far shows that the image of $\beta\vp$ equals 
$A\left(\sum_{g\in G}\delta_{g}\right)A\subseteq A$, isomorphic
as $S$--bimodule to $S\otimes_{R}S/\cap_{g\in g}I_{g}$.
\end{proof}

\begin{corollary}
If in the above setting $S\otimes_{R}S$ is reduced then $\alpha$ is a bijection and 
the bijections $\alpha,\beta$ identify the map $\vp$ with the inclusion of the two--sided ideal 
$A\left(\sum_{g\in G}\delta_{g}\right)A$  into $A$.\qed
\end{corollary}

\begin{proof}
By Lemma \ref{kervp} \eqref{vp4} $\bigcap_{g \in G}I_g$ is the nilradical of $S\otimes_R S$ and thus  equal to $0$ in this situation. Thus, the homomorphism $\alpha$ of the proposition is bijective. 
\end{proof}

If $|G|$ is invertible in $S$, let again $e=\frac{1}{|G|}\sum_{g\in G}\delta_{g}\in A$, as defined in Section \ref{Sub:twistedgroupring}.

\begin{corollary}
  \label{cor.3.9}Denote the homomorphism of $A$--bimodules 
  $$\mu:Ae\otimes_{eAe}eA \to A \quad  \text{ by } \quad
\mu(ae\otimes ea')=aea' \ .$$ 
  Then the diagram below is commutative up to a factor of $|G|$, 
  \begin{align} 
\xymatrix{
Ae\otimes_{eAe}eA\ar[r]^{\mu}&A\\
S\otimes_{R}S\ar@{{>>}}[u]^{()e\otimes e()}\ar[r]^-{\vp}&
\Maps(G,S)
\ar[u]^{\cong}_{\beta}\,,
}
  \end{align}
  i.e. $\beta \circ \phi =|G| \mu \circ (()e\otimes e()).$
If  $|G|$ is invertible in $K$,  then one can identify $\vp:S\otimes_{R}S\to \Maps(G,S)$ with $\mu$.
\end{corollary}
\begin{proof}
By Lemma \ref{Lem:Se} \eqref{Se4} $S\otimes_RS \cong Ae\otimes_{eAe}eA$ as $A$-bimodules, and by Proposition~\ref{Prop:commdiagA}, we have that $\beta$ is an isomorphism.
\end{proof}

\subsection{The structure of $\vp$ for reflection groups}
Now we return to the situation where $S=\Sym_{K}V$ and the finite group $G\leqslant\GL(V)$
acts linearly on $S$. 
In the following key result the equivalence (\ref{wat1})$\Longleftrightarrow$(\ref{wat2}) is due to Junzo Watanabe \cite[Cor.2.9, Cor.2.12, Lemma 2.7]{Watanabe}.
\begin{theorem}
\label{thm:watanabe}
For a finite subgroup $G\leqslant\GL(V)$ with $|G|$ invertible in $K$ the following are equivalent.
\begin{enumerate}[\rm(a)]
\item 
\label{wat1}
The group $G$ is generated by pseudo--reflections in $\GL(V)$.
\item
\label{wat2}
The ring $S\otimes_{R}S$ is Cohen--Macaulay.
\item
\label{wat3}
The ring $S\otimes_{R}S$ is a complete intersection in the polynomial ring $S\otimes_{K}S$.
\end{enumerate}
If these equivalent conditions are satisfied then $S\otimes_{R}S$ is reduced and
$\vp\colon S\otimes_{R}S\to \Maps(G,S)$ is injective. This ring homomorphism is the 
normalization morphism for $S\otimes_{R}S$.
\end{theorem}

\begin{proof}
As stated above, (\ref{wat1})$\Longleftrightarrow$(\ref{wat2}) is due to Watanabe and clearly
(\ref{wat3})$\Longrightarrow$(\ref{wat2}). It thus suffices to show 
(\ref{wat1})$\Longrightarrow$(\ref{wat3}).  
With $S=\Sym_{K}V\cong K[x_{1},\ldots, x_{n}]$ the polynomial ring,
$S\otimes_{K}S\cong \Sym_{K}(V\oplus V)\cong K[x'_{1},\ldots, x'_{n}; x''_{1},\ldots, x''_{n}]$ is
a polynomial ring in $2n$ variables, where we have set $x'_{i} = x_{i}\otimes 1$ and
$x''_{i}=1\otimes x_{i}$.

With $f_{i}\in R\subseteq S$ basic invariants, so that $R=K[f_{1},\ldots, f_{n}]\subseteq S$, 
one has the presentation
\begin{align*}
S\otimes_{R}S \cong K[x'_{1},\ldots, x'_{n}; x''_{1},\ldots, x''_{n}]/(f_{i}(\bx'')-f_{i}(\bx'); i=1,\ldots,n)\,.
\end{align*}
Since $S$ is flat (even free) over $R$ by the Chevalley--Shephard--Todd theorem, the 
$R$--regular sequence ${\bf f}=(f_{1},\ldots, f_{n})$ is also regular on $S$. 
As $S$ is flat over $K$, the sequence $(f_{i} \otimes 1)_{i}$ is regular in $S \otimes_{K} S$ with 
quotient $S/(\boldf) \otimes_{K} S$. 
As $S/(\boldf)$ is flat over $K$, it follows that $(1 \otimes f_{i} )_{i}$ forms a regular 
sequence in $S/(\boldf) \otimes_{K} S$. 
Hence $(f_{1} \otimes 1,\ldots, f_{n}\otimes 1, 1\otimes f_{1},\ldots, 1\otimes f_{n})$ is a regular 
sequence in $S \otimes_K S$. This implies that $(1\otimes f_{i} - f_{i}\otimes 1)_{i}$ is a regular 
sequence since it is part of the regular sequence $(f_i \otimes 1 - 1 \otimes f_i, 1 \otimes f_i)_{i}$.
Thus, 
\begin{equation} \label{eq:completeIntersection}
S \otimes_R S = S \otimes_{K} S /( f_i \otimes 1 - 1 \otimes f_i)_{i=1, \ldots, n}
\end{equation}
is a
complete intersection ring as claimed.

Concerning the remaining assertions, if $S\otimes_{R}S$ is Cohen--Macaulay it cannot contain 
any nontrivial torsion submodule, whence $\vp$ is injective by Lemma \ref{kervp}(\ref{vp5}). This means that $\ker(\vp)=\bigcap_{g \in G}I_g=0$ and hence by Lemma \ref{kervp}\eqref{vp4}, $S\otimes_RS$ is reduced.
As $\vp$ is injective and generically an isomorphism by Proposition \ref{prop:Galoisdescent},
it suffices to note that the ring $\Maps(G,S)\cong S^{|G|}$ is normal.
\end{proof}

\begin{question}
Can one strengthen Theorem \ref{thm:watanabe} by showing that $G$ is a group generated by 
pseudo--reflections if, and only if, $S\otimes_{R}S$ is reduced?
\end{question}

\subsection{A note on normalization and conductor ideals} If $\nu\colon C\to \widetilde C$ 
is the normalization homomorphism of a reduced commutative ring $C$, then applying
$\Hom_{C}(-,C)$ to $\nu$ yields an inclusion 
$\nu^{*}=\Hom_{C}(\nu, C)\colon \Hom_{C}(\widetilde C,C)\into C$. 
The image is the \emph{conductor ideal\/} $\fc\subseteq C$ (with respect
to its normalization.) It is also an ideal in the larger ring $\widetilde C$ and is the largest ideal of 
$C$ with this property. Alternatively, one may define the conductor ideal as the annihilator 
$\fc = \ann_{C}\widetilde C/C$.

Below we will use the following facts, for which we could not locate a compact reference. Therefore (and for the convenience of the reader) we include the proofs. Recall that a commutative ring is {\it equicodimensional} if all maximal ideals have the same height.

\begin{lemma}
\label{lem:cond}
Assume the commutative ring $C$ is noetherian, equicodimensional and Gorenstein with its normalization 
$\widetilde C$ a Cohen--Macaulay $C$--module. In this case,
\begin{enumerate}[\rm (a)]
\item
\label{lem:cond1}
The $C$--module $\widetilde C/C$ is Cohen--Macaulay of Krull dimension $\dim C -1$.
\item
\label{lem:cond2}
As $C$--modules $\Ext^{1}_{C}(\widetilde C/C,C)\cong C/\fc$.
\item
\label{lem:cond3} 
As $C$--modules $\Ext^{1}_{C}(C/\fc,C)\cong \widetilde C/C$.
\item
\label{lem:cond4} 
There are isomorphic short exact sequences of $C$--modules
\begin{align*}
\xymatrix{
0\ar[r]&C/\fc \ar[d]_{\cong}\ar[r]&\widetilde C/\fc\ar[d]_{\cong} \ar[r]
&\widetilde C/C\ar[r]\ar[d]^{\cong}&0\hphantom{\,.}\\
0\ar[r]&\Ext^{1}_{C}(\widetilde C/C,C)\ar[r]&\Ext^{1}_{C}(\widetilde C/\fc, C)\ar[r]
&\Ext^{1}_{C}(C/\fc, C) \ar[r]&0\,.
}
\end{align*}
\item
\label{lem:cond5} 
 The conductor ideal $\fc$ is a maximal Cohen--Macaulay $C$--module.
\item
\label{lem:cond6} 
If $\widetilde C$ is a regular ring, then the (reduced) vanishing locus 
$V(\fc)\subseteq \Spec C$ is the (reduced) singular locus of $C.$
\end{enumerate}
\end{lemma}

\begin{proof}
Because the normalization homomorphism is generically an isomorphism, the Krull dimension
of $\widetilde C/C$ is at most $\dim C-1$. Because both $\widetilde C$ and $C$ are Cohen--Macaulay of Krull dimension $\dim C$ by assumption, the short exact sequence
\begin{align}
\tag{$\dagger$}
\label{dis:C}
\xymatrix{
0\ar[r]&C\ar[r]^-{\nu}&\widetilde C\ar[r]&\widetilde C/C\ar[r]&0
}
\end{align}
shows that the depth of $\widetilde C/C$ is at least $\dim C-1$, whence the dimension 
and depth coincide and are equal to $\dim C-1$, thus establishing (\ref{lem:cond1}).

For (\ref{lem:cond2}) note that $\Ext^{i}_{C}(\widetilde C,C)=0$ for $i\neq 0$ as $\widetilde C$
is a (necessarily maximal) Cohen--Macaulay $C$--module. Applying $\Hom_{C}(-,C)$ to the short 
exact sequence (\ref{dis:C}) and noting that $\Hom_{C}(\widetilde C/C,C)=0$ one obtains the short 
exact sequence of $C$--modules
\begin{align}
\tag{$\ddagger$}
\label{dis:CC}
\xymatrix{
0&\ar[l] \Ext^{1}_{C}(\widetilde C/C,C)&\ar[l] C&\ar[l]\fc&\ar[l]0\,,
}
\end{align}
and so (\ref{lem:cond2}) follows.

As to (\ref{lem:cond3}), just apply $\Hom_{C}(-,C)$ to the short exact sequence 
\begin{align*}
\xymatrix{
0\ar[r]&\fc\ar[r]^-{\nu}& C\ar[r]&C/\fc\ar[r]&0
}
\end{align*}
and observe that $\Hom_{C}(\fc, C)\cong \widetilde C$ as $\widetilde C$, being maximal Cohen--Macaulay over $C$, is a reflexive $C$--module.

Finally, apply $\Hom_{C}(- ,C)$ to the short exact sequence 
\begin{align*}
\xymatrix{
0\ar[r]&\fc\ar[r]^-{\nu}& \widetilde C\ar[r]&\widetilde C/\fc\ar[r]&0
}
\end{align*}
to obtain first $\Ext^{1}_{C}(\widetilde C/\fc,C)\cong \widetilde C/\fc$ and then (\ref{lem:cond4}).

Item (\ref{lem:cond5}) follows from (\ref{lem:cond2}) as $\Ext^{1}_{C}(\widetilde C/C,C)$ is 
a Cohen--Macaulay $C$--module of Krull dimension $\dim C-1$. 
Now use the short exact sequence \eqref{dis:CC} to conclude.

Item  (\ref{lem:cond6}) follows from the fact that $V(\fc)$ describes the non--normal locus of 
$\Spec C$, thus, $V(\fc)\subseteq \Sing(C)$. Outside of $V(\fc)$, the normalization homomorphism is an isomorphism, thus, $\Spec C$ is regular there as this holds for $\widetilde C$ by assumption.
\end{proof}

Translating this into a statement for the twisted group algebra, we obtain the following 
structure theorem for the algebra $\overline A =A/AeA$.
Recall that a \emph{homological epimorphism\/} is a ring epimorphism $\varepsilon: \Lambda \rightarrow \Lambda'$  such that
 for any (left or right) $\Lambda'$--modules $M,N$, restriction of scalars along $\Lambda \onto \Lambda'$
 yields the natural map  $\Ext^{i}_{\Lambda'}(M,N)\xto{\ \cong\ } \Ext^{i}_{\Lambda}(M,N)$ for all integers $i$, see \cite[Thm.~4.4.(5), (5')]{GeigleLenzing}.
\begin{theorem}
\label{thm:Abar}
Assume the finite subgroup $G\leqslant\GL(V)$ with $|G|$ invertible in $K$ is generated by
pseudo--reflections%
\footnote{We allow $G$ to be the trivial group.}
and let $A=S*G$ be the twisted group algebra. 
\begin{enumerate}[\rm(a)]
\item 
\label{AeAprojective}
The ideal $AeA$ of $A$ is projective both as left and right $A$--modules.
\item \label{Extiso}
The ring homomorphism $A\onto \overline A$ is a 
homological epimorphism.
\item \label{finglodim} $\overline A$ is of finite global dimension at most $n=\dim V$.
\item \label{normalcok} As an $S\otimes_RS$--module, $\overline{A}$ identifies with the cokernel of the normalization homomorphism $\varphi:S\otimes_RS \to \Maps(G,S).$
  \item \label{AbarCM} $\overline A$ is a Cohen--Macaulay $R$--module of Krull dimension $n-1$.
\end{enumerate}
\end{theorem}

\begin{proof}
(\ref{AeAprojective}) In view of Corollary \ref{cor.3.9} and Theorem \ref{thm:watanabe}, 
the multiplication map $Ae\otimes_{eAe}eA\to AeA$ is an isomorphism of $A$--bimodules. 
It thus suffices to prove that $Ae\otimes_{eAe}eA$ is projective as (one--sided) $A$--module. 
Using again Corollary \ref{cor.3.9}, we have also the identification 
$S\otimes_{R}S\cong Ae\otimes_{eAe}eA$ as $A$--bimodules. Moreover, by the 
Chevalley--Shephard--Todd theorem, $S$ is a free $R$--module, whence $S\otimes_{R}S$ 
is free as left or right $S$--module. Now $S\cong Ae$ is a projective left $A$--module and 
$eA\cong S$ is a projective left $eAe\cong R$--module. 
Thus, $Ae\otimes_{eAe}eA$ is projective as left $A$--module. The statement for the right module 
structure follows by symmetry.

It is well known that (\ref{AeAprojective}) implies (\ref{Extiso}). This is shown in 
\cite{AuslanderPlatzeckTodorov} for 
Artin algebras, but their arguments apply to any rings. For a reference
that makes no such restrictive assumption, see \cite[Thm.~4.4]{GeigleLenzing} and
\cite[Lemma 2.7]{Krause}. 

As $\gldim A=n$, property (\ref{Extiso}) implies $\gldim\overline A \leqslant \gldim A$ giving (\ref{finglodim}). \\
Since $S\otimes_{R}S\cong Ae\otimes_{eAe}eA$ and $A\cong\Maps(G,S)$, as in Corollary~\ref{cor.3.9}, we have that Theorem~\ref{thm:watanabe} gives us (\ref{normalcok}). 
Since $S\otimes_{R}S$ is a complete 
intersection, thus, Gorenstein, and $\Maps(G,S)\cong S^{|G|}$ is Cohen--Macaulay,
Lemma~\ref{lem:cond} (\ref{lem:cond1}) applies to show that $\overline A$ is a 
Cohen--Macaulay module of Krull dimension $n-1$, equivalently as $S\otimes_{R}S$ or 
$R$--module, as claimed in (\ref{AbarCM}). 
\end{proof}

\begin{cor} \label{Cor:gdimA}
Let $S$ be as above and $G \leqslant \GL(V)$ be a finite group generated by pseudo-reflections. Let  $A=S*G$ and $e_\chi=e_{\chi}^2 \in A$ an idempotent associated to a linear character $\chi$. Then: \\
(1) The quotient algebra $\overline{A}=A/Ae_{\chi}A$ is Koszul. \\
(2) If $G \neq \mu_2$, then $\gl \overline{A} =n$. 
\end{cor}

\begin{proof} 
By Cor.~\ref{Cor:quotientsiso} we may assume that $\chi$ is the trivial character and thus $e_\chi=e=\frac{1}{|G|}\sum_{g \in G}\delta_g$. Denote by $V$ the defining representation of $G$. Following \cite{Auslander86}, the Koszul complex of $K$ 
$$0 \longrightarrow S(-n) \otimes_K \det V \longrightarrow S(-n+1) \otimes_K \bigwedge^{n-1} V \longrightarrow \cdots \longrightarrow S \longrightarrow K \xto{} 0 $$
yields an $A$-projective resolution of $K$. A minimal projective $A$-resolution for any simple $A$-module $W$ is obtained by tensoring this complex with $W$ over $K$, with $G$ acting diagonally:
\begin{align*}  0 & \longrightarrow  S(-n) \otimes_K \det V\otimes_K W \longrightarrow S(-n+1) \otimes_K \bigwedge^{n-1} V \otimes_K W \longrightarrow \cdots \\
 \cdots & \longrightarrow S \otimes_K W \longrightarrow W \longrightarrow 0 
\end{align*}
 Since $e$ is the idempotent for the trivial representation $K$, any irreducible representation $W\not \cong K$ gives rise to a nonzero projective $\overline{A}$-module $S \otimes_K W$. \\
By the theorem, part \eqref{AeAprojective}, tensoring the above  $A$-projective resolution of $W$ with $-\otimes_A \overline{A}$ yields an $\overline{A}$-projective resolution of $W$ (cf.~Thm.~1.6.~and Ex.~1 of \cite{AuslanderPlatzeckTodorov}). From this follows that $\overline{A}$ is Koszul and $\gl \overline{A} \leq n$. \\
For the equality we show that $\Ext^n_{\overline{A}}(W, \det V \otimes W) \neq 0$. Since $G \neq \mu_2$ in statement (2),  there exists $W \neq K$ and $W \neq \det(V)^{-1}$ such that $W$ and $\det V \otimes W$ yield nonzero $\overline{A}$-modules. For any $A$- modules $M, N$ it holds that $\Ext^i_A(M,N)=\Ext_S^i(M,N)^G$. 
Moreover, by the theorem, one has $\Ext^i_{\overline{A}}(M,N) \xrightarrow{\cong} \Ext^i_A(M,N)$ for all $M,N \in 
\Mod(\overline{A})$.  The projection $S \otimes_K \det V \otimes_K W \xrightarrow{} \det V \otimes_K W$ represents a nonzero element of $\Ext^n_S(W,\det V \otimes_K W)$ that is $G$-invariant. Thus $\Ext^n_S(W,\det V \otimes_K W)^G \neq 0$. \end{proof}

\begin{remark} \label{Rmk:mu2}
 If $G=\mu_2$, then since $G$ is generated by a reflection we can choose a basis of $V$ so that $G$ is generated by $\diag(-1,1,\ldots,1)$.  Let $C=K[x_2,\ldots,x_n]$.  Then an explicit calculation shows that the global dimension of $\overline{A}$ drops indeed: one may realize $A=(K[x_1,\ldots,x_n]*\mu_2)$ as the order
  $$\begin{pmatrix} K[x_1^2] & K[x_1^2] \\ x_1^2K[x_1^2] & K[x_1^2] \end{pmatrix} \otimes C.$$ In this description, $e$ is the idempotent matrix $e_{11}$ and $AeA$ is of the form
  $$\begin{pmatrix} K[x_1^2] & K[x_1^2] \\ x_1^2K[x_1^2] & x_1^2K[x_1^2] \end{pmatrix}\otimes C.$$
  Thus $\overline{A} \cong C$, of global dimension $n-1 < n= \gl A$. Note here that $R/(\Delta) \cong C$ is regular and moreover $\Abar \cong R/(\Delta)$.
\end{remark}

Until the end of this section we assume the hypotheses of Theorem~\ref{thm:Abar} that $G\leqslant \GL(V)$ is generated by pseudo--reflections and that $|G|$ is invertible in $K$.

Our next goal is to determine the annihilator of $\overline A$ as $S\otimes_{R}S$--module,
equivalently, in view of the preceding Theorem~\ref{thm:Abar}(\ref{normalcok}) and Lemma~\ref{lem:cond}, 
the conductor ideal of the normalization of $S\otimes_{R}S$. 

We write $f_i$ for the basic invariants so $R=k[f_1,\ldots,f_n] \subseteq S = k[x_1,\ldots,x_n]$
  and $x_i' = x_i \otimes 1, x_i'' = 1 \otimes x_i$ in $S\otimes_KS$ as introduced in proof of Theorem~\ref{thm:watanabe}.
Let $$I_\Delta =(x_i \otimes 1 -  1 \otimes x_i\st i=1,\ldots,n) = (x'_i - x''_i \st i=1,\ldots,n)$$
be the ideal of the diagonal in $\Spec S\otimes_KS$.  Recall that $\Omega^1_S = I_\Delta/I_\Delta^2$ is a free $S$--module with basis $dx_i = (x'_i - x''_i)$ modulo $I_\Delta^2$.
Since $df=\sum_i(\partial f/\partial x_i) dx_i$ in $\Omega^1_S$, we can find $\nabla_{i}^{j}(\bx',\bx'')$ in $S\otimes_KS$ such that
the elements $df_i = f_{i}(\bx'')-f_{i}(\bx')$ in $S\otimes_{K}S$ can be expressed as
\begin{equation}
\label{diffquotient}
f_{i}(\bx'')-f_{i}(\bx') =\sum_{j=1}^{n}\nabla_{i}^{j}(\bx',\bx'')(x''_{j}-x'_{j}).
\end{equation}
The elements $\nabla_{i}^{j}(\bx',\bx'')$ are uniquely defined modulo $I_\Delta$, so if we let $\nabla_{i}^{j}(\bx,\bx)$ to be the image of
$\nabla_{i}^{j}(\bx',\bx'')$ under the multiplication map $S\otimes_KS \to S$
then $\nabla_{i}^{j}(\bx,\bx) = \partial f_{i}/\partial x_{j}\in S$.  Further recall that
\begin{align*}
J = \det \left(\frac{\partial f_{i}}{\partial x_{j}}\right) \in S\,,
\end{align*}
is the Jacobian of the basic invariants $f_{i}$.

\begin{lemma}
\label{lem:Jincond}
For $g\in G$, one has $\beta(\vp(\det(\nabla_{i}^{j}(\bx',\bx'')))) = J\delta_{1}\in A$, where $\beta$ is the map defined in Proposition \ref{Prop:commdiagA}. 
\end{lemma}

\begin{proof}
By definition of $\vp$, and because it is a ring homomorphism, one has
\begin{align*}
\vp(\det(\nabla_{i}^{j}(\bx',\bx'')))(g)&= \det(\nabla_{i}^{j}(\bx,g\bx))\in S\,.
\end{align*}
For $g=1\in G$, this evaluates to $J$. For $g\neq 1$, as the $f_{i}$ are $G$--invariant, one has in $S\otimes_{K}S$
\begin{align*}
f_{i}(\bx'')-f_{i}(\bx') &=f_{i}(g(\bx''))-f_{i}(\bx')\\
&=\sum_{j=1}^{n}\nabla_{i}^{j}(\bx',g(\bx''))(g(x''_{j})- x'_{j})\,,
\end{align*}
and specializing $x'', x'\mapsto x$, this becomes
$0= \sum_{j=1}^{n}\nabla_{i}^{j}(\bx,g(\bx))(g(x_{j})- x_{j})$
in $S$, whence the linear system 
$\left(\nabla_{i}^{j}(\bx,g(\bx))\right)(v_{1}, \ldots, v_{n})^{T}=\mathbf 0$ 
has the nontrivial solution $(v_{j})_{j=1,\ldots,n}= (g(x_{j})- x_{j})_{j=1,\ldots,n}\neq \mathbf 0$ 
over the domain $S$. This forces the determinant $\det(\nabla_{i}^{j}(\bx,g\bx))$ to vanish. This means that $\beta(\vp(\det(\nabla_{i}^{j}(\bx',\bx'')))) = J\delta_{1}$ in $A$.
\end{proof}

\begin{corollary}
In $A$, one has the containment of ideals $A(J\delta_{1})A \subseteq AeA$. In particular,
$\overline A$ is annihilated by $J$ both as left or right $S$--module.
\end{corollary}

\begin{proof}
By Lemma \ref{lem:Jincond}, $\beta\vp\left(\det\left(\nabla_{i}^{j}(\bx',\bx'')\right)\right) =J\delta_{1}$ in $A$. Thus by the commutative diagram \eqref{Diag:AmAeA} we have that 
$J\delta_{1}=\alpha\left(\det\left(\nabla_{i}^{j}(\bx',\bx'')\right)\right) \in AeA$.
\end{proof}

In fact we have the following precise description of the annihilator of $\overline A$ as
$S\otimes_{R}S$--module.

\begin{proposition}\label{prop:normConductor}
The annihilator ideal of $\overline A$ in $S\otimes_{R}S$ is the conductor ideal $\fc$
of the normalization of $S\otimes_{R}S$ and
\begin{align*}
\fc=\ann_{S\otimes_{R}S}\overline A =
\left(\det(\nabla_{i}^{j}(g(\bx'),\bx'')); g\in G\right)\,. 
\end{align*}
where the right hand side indicates the ideal of $S\otimes_R S$ generated by
the $(\det(\nabla_{i}^{j}(g(\bx'),\bx'')$ as $g$ ranges over all elements in $G$.
The image of this ideal under $\vp$ is $\vp(\fc)=\Maps(G,JS)=
J\Maps(G,S)\subseteq \Maps(G,S)$, the principal ideal generated by the Jacobian $J{\cdot} 1$ in
$\Maps(G,S)$.
\end{proposition}

\begin{proof}
The first statement follows from Theorem~\ref{thm:Abar}(\ref{normalcok}).
  Observe that $C = S\otimes_{R}S \cong S\otimes_{K}S/(\text{regular sequence})$,
where the regular sequence is of length $n=\dim C$ as in \eqref{eq:completeIntersection} in the proof of Theorem \ref{thm:watanabe}. Combining this with Lemma \ref{lem:cond}(\ref{lem:cond2}) implies that naturally
\begin{align}\label{eqn:Cmodc}
   C/\fc \cong \Ext^{1}_{C}(\widetilde C/C,C) \cong \Ext^{1}_{C}(\overline A, C)\cong \Ext^{n+1}_{S\otimes_{K}S}(\overline A, S\otimes_{K}S)\,.
\end{align}
To determine the latter extension module, we make explicit the free $S\otimes_{K}S$--resolution
of $\overline A$ as mapping cone of the $S\otimes_{K}S$--resolutions of $C=S\otimes_{R}S$
and of $\widetilde C=\bigoplus_{g\in G}S\otimes_{R}S/I_{g}$, respectively.

As $S\otimes_{R}S$ is the complete intersection $S\otimes_{K}S$ modulo the regular sequence
$(f_{i}(\bx'')-f_{i}(\bx'))_{i}$, a free $S\otimes_{K}S$--resolution is given by the Koszul complex
on that regular sequence,
\begin{align*}
\KK_{\boldf}=\KK((f_{i}(\bx'')-f_{i}(\bx'))_{i},S\otimes_{K}S)\xto{\ \simeq\ } S\otimes_{R}S\,,
\end{align*}
where $\simeq$ is an isomorphism in the derived category since the complex is a resolution.

Now for $g\in G$ one has $I_{g}=(x_{i}''-g(x_{i}'); i=1,\ldots,n)\subseteq S\otimes_{R}S$ by 
Lemma \ref{kervp}(\ref{vp1}).
Applying $g$ to the first tensor factor in equation (\ref{diffquotient}) shows
\begin{align} \label{eqn:mat}
f_{i}(\bx'')-f_{i}(g(\bx')) &=\sum_{j=1}^{n}\nabla_{i}^{j}(g(\bx'),\bx'')(x''_{j}-g(x'_{j}))\,.
\end{align}
As $f_{i}$ is $G$--invariant, $f_{i}(g(\bx')) =f_{i}(\bx')$ and so there is a containment
\begin{equation*}
(f_{i}(\bx'')-f_{i}(\bx'); i=1,\ldots,n)\subseteq (x''_{j}-g(x'_{j}); j=1,\ldots,n)\subset S\otimes_{K}S
\end{equation*}
of ideals in $S\otimes_{K}S$. In particular, 
$S\otimes_{R}S/I_{g}\cong S\otimes_{K}S/(x''_{j}-g(x'_{j}); j=1,\ldots,n)$ as $S\otimes_{K}S$--modules.

The sequence $(x''_{j}-g(x'_{j}))_{j}$ is regular in $S\otimes_{K}S$ as it consists of linearly 
independent linear forms. Thus, $S\otimes_{R}S/I_{g}$ is also a complete intersection 
in $S\otimes_{K}S$ with free resolution the Koszul complex on that regular sequence,
\begin{align*}
\KK_{g}=\KK((x''_{j}-g(x'_{j}))_{j},S\otimes_{K}S)\xto{\ \simeq\ } S\otimes_{R}S/I_{g}\,.
\end{align*}
Let $M_{g} = \left(\nabla_{i}^{j}(g(\bx'),\bx'')\right)$ be the $n\times n$ matrix over 
$S\otimes_{K}S$ indicated in~(\ref{eqn:mat}).  Hence we have a commutative diagram of the form
\begin{align*}\xymatrixcolsep{5pc}
\xymatrix{
(S\otimes_KS)^n\ar[r]^-{(f_{i}(\bx'')-f_{i}(\bx'))_{i=1}^n}\ar[d]_{\prod_{g\in G}(M_{g})}&S\otimes_KS\ar[r]\ar[d]&S\otimes_{R}S\ar[d]^{\vp}\\
(S\otimes_KS)^{n|G|}\ar[r]_{\prod_{g\in G}(x''_{j}-g(x'_{j}))_{j=1}^n}&(S\otimes_KS)^{|G|}\ar[r]&\Maps(G,S)\,.
}
\end{align*}
So the matrices $(M_{g})_{g\in G}$ provide a morphism of the degree zero and one components of the Koszul complexes. Since the Koszul complex is functorial
the exterior powers $\prod_{g\in G}(\Lambda^{\bdot}M_{g})$ provide a lift of the evaluation homomorphism
$ev_{g}\vp:S\otimes_{R}S\to S\otimes_{R}S/I_{g}$ to a morphism between the resolutions.
So we obtain
\begin{align*}
\xymatrix{
\KK_{{\bf f}}\ar[r]^-{\simeq}\ar[d]_{\prod_{g\in G}(\Lambda^{\bdot}M_{g})}&S\otimes_{R}S\ar[d]^{\vp}\\
\prod_{g\in G}\KK_{g}\ar[r]^-{\simeq}&\Maps(G,S)\,.
}
\end{align*}
With $\Phi = (\Lambda^{\bdot}M_{g})_{g\in G}$ the indicated morphism between resolutions,
the mapping cone on $\Phi$ yields a resolution of $\overline A \cong \Maps(G,S)/\Imm(\vp)$ as 
$S\otimes_{K}S$--module. This mapping cone is a complex of free $S\otimes_{K}S$--modules of 
length $n+1$, whence we can calculate $\Ext^{n+1}_{S\otimes_{K}S}(\overline A, S\otimes_{K}S)$
simply as the cokernel of the last differential in the  $S\otimes_{K}S$--dual of that mapping cone.
The result is easily seen to be
\begin{align*}
\Ext^{n+1}_{S\otimes_{K}S}(\overline A, S\otimes_{K}S)&=
S\otimes_{R}S/\left(\det M_{g}; g\in G\right)\\
&=S\otimes_{R}S/\left(\det\left(\nabla_{i}^{j}(g(\bx'),\bx'')\right); g\in G\right)\,.
\end{align*}
Combining the above isomorphism of $C$-modules with~(\ref{eqn:Cmodc}) we 
conclude that
$$\fc = \left(\det\left(\nabla_{i}^{j}(g(\bx'),\bx'')\right); g\in G\right)\subseteq S\otimes_{R}S$$ as 
claimed.

By the same reasoning as in Lemma \ref{lem:Jincond}, it follows that 
$\beta(\vp\left(\det\left(\nabla_{i}^{j}(g(\bx'),\bx'')\right)\right) )= J\delta_{g}\in A$.
\end{proof}

\begin{proposition} \label{JisAnnBarA}
If $G$ is generated by (pseudo--)reflections of order $2$, then $J$ is a squarefree 
product of linear forms and so, with $C=S\otimes_RS$, $C/\fc\subseteq (S/(J))^{|G|}$ is reduced. Moreover, 
$$V(\fc)=\Sing(S\otimes_{R}S)\subseteq \Spec(S\otimes_{R}S) \ .$$
\end{proposition}

\begin{proof}
  The fact that $J$ is a square free product of linear forms is
  noted in Section~\ref{defnms}(\ref{defn:zJ}).  We always have
  a natural inclusion $C/\fc \subseteq \widetilde{C}/\fc$
  since $\fc$ is the conductor.  Proposition~\ref{prop:normConductor} shows that
  $\widetilde{C}/\fc \subseteq (S/J)^{|G|}.$  We know that $\widetilde{C}=\Maps(G,S)$ by the last statement of Theorem~\ref{thm:watanabe} and so is regular.  Now by Lemma~\ref{lem:cond}(\ref{lem:cond6}) we obtain that $V(\fc)=\Sing(S\otimes_{R}S)$.
\end{proof}

\begin{corollary}\label{MF}
For $G$ generated by pseudo--reflections, consider the map
$\psi:\Maps(G,S)\to S\otimes_{R}S$ given by
\begin{align*}
\psi((s_{g})_{g\in G}) &=
\sum_{g\in G}s_{g}\det\left(\nabla_{i}^{j}(g(\bx'),\bx'')\right)\,.
\end{align*}
This is $S$--linear on the left and $\vp\psi = J\id_{\Maps(G,S)}$.

As both $\Maps(G,S)$ and $S\otimes_{R}S$ are free (left) $S$--modules,
the pair $(\vp,\psi)$ constitutes a matrix factorization of $J\in S$ whose cokernel is
$\overline A$ as left $S$--module. In particular, $\overline A$ is a maximal 
Cohen--Macaulay module over the hypersurface ring $S/(J)$.\qed
\end{corollary}
\begin{proof}
That $\vp\psi = J\id_{\Maps(G,S)}$ is a variation on Lemma~\ref{lem:Jincond}.
  That $\cok =\overline A,$ is in Theorem~\ref{thm:Abar}(\ref{normalcok}).  That matrix factorizations give maximal Cohen-Macaulay modules is well known.
\end{proof}

\subsection{The map $\vp$ and the group matrix} \label{Sub:groupmatrix}

Let $G$ be a finite group.  Let
$\MM = (x_{gh^{-1}})_{g,h\in G}$, a matrix with entries from the polynomial ring $K[\{G\}]=K[x_{g}| g\in G]$, 
the variables thus indexed by the group elements. 
This is the classical
{\em group matrix\/} of $G$. This matrix essentially represents the multiplication table of the group written in commuting independent variables.

In a famous letter to Frobenius, Dedekind observed that in all examples he could handle, with $K=\CC$ the field
of complex numbers the {\em group determinant\/} $\det \MM$ decomposed as
\begin{align*}
\det \MM = \prod_{i=1}^{r}F_{i}^{d_{i}}
\end{align*}
for irreducible homogeneous polynomials $F_{i}\in P$ of degree $d_{i}\geqslant 1$. He asked 
Frobenius for an explanation. Dedekind had already himself established the case of finite abelian groups, 
for which he found each $d_{i}=1$.
That was the birth of the representation theory of general finite groups. See Section~4.11 of \cite{Etingofetal} for a beautiful 
short account of this story. We describe a direct relation between our matrix factorization of $J$ and the group matrix.

Let now $G \leqslant \GL(V)$ be a finite pseudo-reflection group acting on $S=\Sym_{K}V$ and let $R=S^G=K[f_1, \ldots, f_n]$ as above. Let $R_{+}\subset R$ the Hilbert ideal, cf.~Section \ref{Sub:isotypical}. 
It is known that $R_+ S \subseteq S$ has a $G$--stable complement $U\subseteq S$ that is isomorphic to the regular representation
$KG$ as a $G$--module. This complement can be realized as the vector space of all $G$--harmonic polynomials, cf.~ for example \cite[Cor.~9.37]{LehrerTaylor} and also see \cite{Steinberg}.
 Equivalently, see \cite[Thm.~9.38]{LehrerTaylor}, it is the subspace in $S$ generated by all (higher) partial derivatives of the Jacobian
$$J=\det\left(\left(\frac{\partial f_{i}}{\partial x_{j}}\right)_{i,j=1,\ldots,n}\right).$$ In particular, we can find a generic harmonic polynomial
$x\in U$ so that the set $\{x_{g}:= g(x)|g\in G\}$ forms a basis of $U$. Further, the multiplication map
$R\otimes_{K}U\to S$ is an isomorphism of $RG$--modules, thereby identifying $S$ with $RG$ as an $RG$--module.

\begin{proposition}
The matrix $X_{g,h}$ of the $S$-linear map $\varphi: S \otimes_RS \xrightarrow{} \Maps(G,S)$ is the group matrix $\MM$ of $G$ evaluated at the harmonic polynomials $x_g \in S$, $g \in G$.
\end{proposition}

\begin{proof}
The map $\vp\colon S\otimes_{R}S\to \Maps(G,S)$ is determined by $x\otimes y\mapsto \vp(x\otimes y)(g) = xg(y)$ 
with respect to the $S$--bases $\{1\otimes x_{h^{-1}}''\mid h\in G\}$ of $S\otimes_{R}S$ and 
$\{\delta_{g}\mid g\in G\}\subset \Maps(G,S)$. Thus its matrix is $X_{g,h} = \vp(1\otimes x''_{h^{-1}})(g) = g(x_{h^{-1}})=x_{gh^{-1}}$, that is, $\MM$ evaluated at the harmonic polynomials $x_{g}\in S, g\in G$.
\end{proof}

  We note that the matrix factorization of $J$ has a particularly nice form.
  Let the irreducible components of the discriminant $\Delta$ be $\Delta_j$ and let the ramification index of the cover $S$ over $R$ on $\Delta_j$
be $r_j,$ which is also the order of the cyclic subgroup of $G$ that fixes the mirror which is a component of the inverse image of $V(\Delta_j).$

\begin{proposition} \label{Prop:groupmatrix}
In the matrix factorization $(\vp,\psi)$ of $J$ with $\coker\vp =\overline A$,
  \begin{enumerate}[\rm(a)] 
    \item \label{phidual} The morphism $\psi$ is the transpose of $\vp$ up to base change.
      \item \label{syzdual} 
      $\Hom_S(A,S/(J)) \cong$ $\mathrm{syz}^1_{S/(J)}\overline{A},$ where  $\mathrm{syz}^1_{S/(J)}$ denotes the first $S/(J)$-syzygy .
              \item \label{detphi} $\det(\vp)=J^{|G|/2}$
\item \label{rankAeA} The rank of $\overline A$ along the component $\Delta_j$ of the discriminant is
$$  \rank_{\Delta_j}\overline{A} = \frac{(r_j-1)|G|^2}{2r_j}={ r_j \choose 2}\frac{|G|^2}{r_j^2}.$$
  \end{enumerate}
\end{proposition}

\begin{proof}
  The map $\varphi : S\otimes_R S \to \Maps(G,S)$ can be identified with
  the normalization map $C \to \widetilde{C}$ as in Theorem~\ref{thm:watanabe}.  As noted in the proof of Lemma~\ref{lem:cond} (b), the dual of this map $\Hom_C(-,C)$  is naturally the map
  $$\Hom_C(\varphi,C):\fc \to C.$$
  As in Prop.~\ref{prop:normConductor}, we note that the composition
  $$\varphi \circ \Hom_C(\varphi,C) : \fc \to C \to \widetilde{C}$$
  can be identified with the inclusion $J\Maps(G,S) \to \Maps(G,S)$ and
  so $\varphi \circ \Hom_C(\varphi,C)  =  J \id_{\Maps(G,S)}.$
  Since $C$ is a finitely generated free $S$-module, we obtain the same statement when dualizing over $S$ instead of $C$.  Lastly, by
 Prop.~\ref{MF} we have that
  $$ \varphi \circ \Hom_S(\varphi,S)  =  J \id_{S\otimes_RS} =  \varphi\circ\psi $$
  giving statement (\ref{phidual}).
  Statement (\ref{syzdual}) follows from \cite[Prop.~7.7]{Yos} together with statement (\ref{phidual}).
The equation $\det(\vp\psi) = J^{|G|}$ that follows from $\vp\psi=J\id_{\Maps(G,S)}$ then entails that 
\[J^{|G|}=\det(\vp)\det(\psi) = \det(\vp)^2 \, , \]
giving statement (\ref{detphi}).
To establish statement (\ref{rankAeA}) we first note that $J=\prod_{i=1}^{m_1}L_i^{r_i-1}$, where the $L_i$ are the linear forms defining the mirrors of $G$ and $r_i$ is the order of the cyclic group that leaves the mirror invariant. The hyperplanes $\{L_i=0\}$
are the irreducible components of the hyperplane arrangement and on such a component the rank of $\overline A$
is accordingly $(r_i-1)|G|/2$. Note that this is an integer, as $|G|$ odd implies that each $r_i$ is odd too.

Next note that $\Delta = zJ=\prod_iL_i^{r_i}$. Grouping the hyperplanes into orbits under the action of $G$, we get $\Delta=\prod_{j=1}^{q}\Delta_j$, where
$\Delta_j=\prod_{L_k\in O_j}L_k^{r_k}$ with $O_j$ an orbit and $q$ the number of such orbits. 
These $\Delta_j$ are the irreducible factors of $\Delta$ in $R$.
Note that the exponents $r_k$ are the same for each linear form in an orbit. We abuse notation and denote this common value for $O_j$ also by $r_j$, giving 
$\Delta_j=(\prod_{L_k\in O_j}L_k)^{r_j}$. Since the stabilizer of a hyperplane in $O_j$ has order $r_j$ we have that $|O_j| \cdot r_j  = |G|$. Hence we obtain the result that $\rank_{\Delta_j}\overline{A} = (r_j-1)|G||O_j|/2$ which gives statement (\ref{rankAeA}).
\end{proof}

\begin{remark}Recall that $m$ the number of pseudo-reflections in $G$
  and $m_1$ is the number of mirrors as discussed in Section \ref{defnms}.  If
  the rank of  $\overline A$ is $r$ on every component we obtain
\begin{eqnarray*}
\binom{r}{2}  \frac{|G|^2}{r^2}  & = &  \frac{r(r-1)|G|^2}{2r^2} \\
& = & \frac{(r-1)|G|^2}{2r} \\
& = & \frac{|G|^2}{2} \frac{m_1(r-1)}{m_1+m_1(r-1)}\\
& = & \frac{|G|^2}{2} \frac{m}{m_1+m} \ . \\
\end{eqnarray*}
\end{remark}

\begin{example}\label{Ex:ranks}
One can explicitly compute the rank of $\overline{A}$ for any finite unitary reflection group $G \leqslant \GL(V)$ with irreducible discriminant in the Shephard--Todd list  with the above: 
\[ \rank_{R/(\Delta)}(\overline{A})=\frac{|G|^2}{2}\frac{m}{m+m_1}.\]
The number $m$ of reflections is given as $\sum_{i=1}^n(d_i -1)$, where $n$ is the dimension of $V$, and $d_i$ are the degrees of the basic invariants. The number $m_1$, that is, the number of different mirrors is given by the sum of the \emph{co-exponents} of $G$. These are the degrees of the homogeneous generators minus $1$ of the logarithmic derivation module of the reflection arrangement corresponding to $G$ \cite[Cor.~6.63]{OTe}. All these numbers can be found in the literature, see e.g. \cite[Table VII]{STo} for the orders and degrees and \cite[Table B.1]{OTe} for the co-exponents. \\
Note that the groups labeled $G_1$ in the Shephard--Todd list are the symmetric groups, which are true reflection groups, so $$\rank_{R/(\Delta)}(\overline{A})=\left(\frac{|G|}{2}\right)^2 \ .$$ For the remaining groups one can determine in which cases the discriminant is irreducible, cf.~Appendix C in \cite{OTe}. 
\end{example}

\begin{corollary} \label{Cor:AbarEndomorphismring}
  If $G \leqslant \GL(V)$ is generated by pseudo-reflections, some of which have order $\geq 3$,
  then $\overline A$ is not an endomorphism ring over $R/(\Delta)$.
\end{corollary}

\begin{proof}
  Suppose there is an $R/(\Delta)$-module $M$ such that
  $\overline A \cong \End_{R/\Delta}(M)$.  Let $\fp$ be an associated prime ideal of $R/(\Delta)$ and let $L$ be the algebraic closure of its residue field.
Then since $R/(\Delta)$ is reduced, we see that
$$\End_{R/\Delta}(M) \otimes L \cong \End_L(L^n) \cong L^{n\times n}$$
where $n$ is the rank of $M$ on the component corresponding to $\fp$.  On the other hand, let $\fp$ be the image of a mirror with a pseudo-reflection of order $r$.  We know that there is an \'etale extension $R'$ of the DVR $R_\fp$ with residue field $L'$ such that $A=R'=A\otimes_{R_{\fp}} R'$ is a matrix algebra over a standard hereditary order with ramification index equal to $r$~\cite{Reiner75}.  A computation shows that we can move any rank one idempotent in a standard hereditary order to the matrix idempotent $e_{11}$.  So we see that $\overline{A}\otimes R'$ will be the algebra of matrices over upper triangular matrices of size $(r-1) \times (r-1)$ which has a nontrivial Jacobson radical unless the ramification index $r=2$.  Therefore $G$ must be a true reflection group.
\end{proof}

\section{Noncommutative resolutions of discriminants} \label{Sec:Main}

\subsection{Matrix factorizations as quiver representations, Kn\"orrer's functors} \label{Sub:MFandCM}

In order to compare modules over the discriminant and the skew group ring, we will reinterpret  Kn\"orrer's functors 
from \cite{KnoerrerCurves} and \cite{KnoerrerCohenMacaulay}.  This yields a reformulation of Eisenbud's theorem \cite{Eisenbud80} and a variation of Kn\"orrer's result (\cite[Prop.~2.1]{KnoerrerCohenMacaulay}) in Remark \ref{Rmk:Knoerrer}. Most of this section follows with standard proofs from the cited results - we will give a short account only and then will  work with the skew group ring $B=T*\mu_2$, which is the ring of our interest (the interested reader may skip to Section \ref{Subsub:Knoerrer} directly for this).

\subsubsection{Modules over path algebras:} Let $R$ be a commutative regular ring, $f \in R$ a non-zero divisor 
 and let
 \begin{equation} \label{Eq:Bquiver}
 {\begin{tikzpicture}[baseline=(current  bounding  box.center)] 
 \node (B) at (-0.1,0) {$B=R$};

 \node (C1) at (0.9,0) {$e_+$} ;
 \node (C2) at (5.1,0)  {$e_- $};

 \draw [thick ] (0.7,-0.7) to [round left paren ] (0.7,0.7);
 \draw [thick ] (5.3,-0.7) to [round right paren] (5.3,0.7);

 \draw [->,bend left=25,looseness=1,pos=0.5] (C1) to node[]  [above]{$v$} (C2);
 \draw [->,bend left=20,looseness=1,pos=0.5] (C2) to node[] [below] {$u$} (C1);

 \node (DD) at (5.7,0) {.};
 \end{tikzpicture}} 
 \end{equation}
 This stands for the associative $R$-algebra generated by $e_+$, $e_-$, $u$, $v$, modulo the relations 
 \begin{gather*}
 e_+^2=e_+   \,, \quad  e_-^2=e_-\, , \quad   e_-+e_+=1 \, , \\
 u  =e_+ue_- \, , \quad  v=e_-v e_+\,  ,   \\
   uv=fe_+ \,, \quad vu=fe_-  \, .
 \end{gather*}
Note that $B$ is free as a $R$-module with basis the four elements $e_-, e_+$, $u, v$. \\
  A right $B$-module $M$ corresponds to a quiver representation of the form 
 \[M:= (M_+ \stackrel[u_M]{v_M}{\rightleftarrows} M_-) \ , \]
 where $M_+=Me_+$ and $M_-=Me_-$ are $R$-modules and $u_M$ and $v_M$ are $R$-linear and must satisfy $u_M v_M = f \mathrm{Id}_{M_+}$ and $v_M u_M = f \mathrm{Id}_{M_-}$. Note here that $M$ is isomorphic to $M_+ \oplus M_-$ as $R$-modules via restriction of scalars.
 A morphism between $B$-modules $M=(M_+ \qurep{v_M}{u_M} M_-)$ and $M'=(M'_+ \qurep{v_M'}{u_M'} M'_-)$ corresponds to a pair $(\alpha_-, \alpha_+)$ of $R$-module homomorphisms such that the diagram 
 \begin{equation} 
 \xymatrix{
 M_+ \ar[r]^{v_M} \ar[d]^{\alpha_+} & M_-\ar[r]^{u_M} \ar[d]^{\alpha_-}    &   M_+ \ar[d]^{\alpha_+}   \\
 M_+' \ar[r]_{v_M'} & M_-'\ar[r]_{u_M'} & M_+'  
 } 
 \end{equation}
 commutes.

Conversely, if we start with a quiver representation $(M_+ \qurep{v_M}{u_M} M_-)$, then $M:=M_+ \oplus M_-$ is naturally a right $B$-module. 
  If $M$ is  finitely generated projective as a $R$-module, then the pair $(u_M,v_M)$ is called a \emph{matrix factorization of $f$}  over $R$ and the $B$-module $(M_+ \qurep{v_M}{u_M} M_-)$ is called a \emph{(maximal) Cohen--Macaulay module} over $B$. The category of such modules is denoted $\CM(B)$.

\begin{lemma}  \label{Cor:Bprops} 
Let $B$ be an algebra of the form \eqref{Eq:Bquiver}. Then:
\begin{enumerate}[\rm(a)]
 \item $B=e_+ B \oplus e_- B$ is the sum of two projective $B$-modules.
 \item $B/Be_+B  \cong B/Be_-B \cong R/(f)$. In particular, there is a natural algebra surjection $B \onto R/(f)$. \label{DiskQuiv}
 \item $e_+ B e_+ \cong e_- B e_- \cong  R$ . \label{TQuiv}
 \item The centre of $B$ is $R$. \label{Bcentre}
 \end{enumerate}
 \end{lemma}

 \begin{proof}
 All four assertions follow from straightforward calculations.
 \end{proof}

In order to relate modules over $B$ and over $R/(f)$ we look at the standard recollement 
 induced by $e_-$. It is given by 
 \[
 \begin{tikzpicture}
 \node at (4,0) {\begin{tikzpicture} 
 \node (C1) at (0,0) {$\Mod B/Be_- B$} ;
 \node (C2) at (5,0)  {$\Mod B$};
 \node (C3) at (10,0)  {$\Mod e_-Be_-$ \ .};

 \draw [->,bend right=38,looseness=1,pos=0.5] (C2) to node[]  [above]{$i^*=-\otimes_B B/Be_-B$} (C1);
 \draw [->,bend left=30,looseness=1,pos=0.5] (C2) to node[] [below] {$i^!=\Hom_B(B/Be_-B,-)$} (C1);
 \draw[->]  (C1) edge [] node[left] [above]{$i_*$}(C2);

 \draw [->,bend right=38,looseness=1,pos=0.5] (C3) to node[]  [above]{$j_!=-\otimes_{e_-Be_-} e_-B$} (C2);
 \draw [->,bend left=30,looseness=1,pos=0.5] (C3) to node[] [below] {{\small $j_*=\Hom_{e_-Be_-}(Be_-,-)$}} (C2);
 \draw[->]  (C2) edge [] node[left] [above]{{\small $j^*=\Hom_B(e_-B,-)$}}(C3);

 \end{tikzpicture}}; 
 \end{tikzpicture}
 \] 
We refer to \cite{FranjouPirashvili} for the properties of the six functors. We will only be interested in the left hand side of the recollement, in particular the functor $i^*$ relating $\Mod B/Be_- B$ and $\Mod B$. Note that with Lemma~\ref{Cor:Bprops}, one can write $i^*= -\otimes_B R/(f)$, $i^!=\Hom_B(R/(f), -)$.  Moreover, $j^*=\Hom_B(e_-B, - ) \cong - \otimes_B Be_-$. \\
One also easily verifies the following statements: Let $M$ be a $B$-module and $C$ be a $R/(f)=B/Be_-B$-module. Then $i_*C=(0 \qurep{}{} C)$ and $i^*(M) = \cok u_M$. The other functors have similar simple expressions.

Recall that an associative ring $\Lambda$ is called \emph{Iwanaga--Gorenstein} if it is noetherian on both sides and the injective dimension of $\Lambda$ as a left and right $\Lambda$-module is finite. 

\begin{remark} We will use the following Theorem \ref{Thm:Eisenbud} for $R$ either a polynomial ring or a power series ring over $K$. However, it can be stated more generally for regular rings, for ease of notation and to keep extra assumptions, such as existence of ranks, at a minimum, we choose $R$ to be an integral domain.
\end{remark} 

\begin{theorem} \label{Thm:Eisenbud}
Let $R$ be a commutative regular ring assume that $R$ is an integral domain. Let $f \in R$, $f \neq 0$ (in particular, $f$ is a non-zero divisor) and $B$ with relations defined as in \eqref{Eq:Bquiver}. \\
Then the ring $B$ is Iwanaga--Gorenstein and $M=(M_+\qurep{v_M}{u_M}M_-)$ is in $\CM(B)$ if and only if $i^*M$ is in $\CM(R/(f))$, where $i^*$ is coming from the recollement as described above. The functor $i^*$ induces an equivalence of categories
 \[\label{eqn:equiv}  \CM(B)/ \langle e_-B \rangle \simeq \CM(R/(f)) \ , \]
 where $e_{-}B$ is the ideal in the category $\CM(B)$ generated by the object $e_{-}B$.
\end{theorem}

 \begin{proof}
The statement that $B$ is Iwanaga--Gorenstein can either be shown directly using properties of the recollement or one can refer e.g. to \cite[Prop.~1.1(3)]{GotoNishida}.\\
Now assume that $M=(M_+\qurep{v_M}{u_M}M_-)$ is in $\CM(B)$. Recall that this means that $M_+$ and $M_-$ are projective over $R$. 
Set $C:=\coker(u_M)=i^*(M)$.  Because $f$ is a non-zero-divisor in $R$, multiplication by $f$ on $M_-$ is injective and so is $u_M$ as $v_M u_M = f \id_{M_-}$. Therefore 
 \begin{equation} \label{Eq:Cres} 0 \xto{} M_- \xto{ \ u_M \ } M_+ \xto{} C \xto {} 0
 \end{equation}
 is a projective resolution of $C$ over $R$. Note that a simple rank calculation shows that the ranks of $M_-$ and $M_+$ have to coincide and that $C$ cannot be projective, thus has projective dimension $1$ over $R$. This implies $C=i^*(M) \in \CM(R/(f))$ by the Auslander--Buchsbaum formula, and so $i^*$ defines a functor $\CM(B)$ to $\CM(R/(f))$.

 Conversely, take any Cohen--Macaulay module $C$ over $R/(f)$ and let \eqref{Eq:Cres} be a projective resolution of $C$ over $R$ with $M_+$ and $M_-$ projective $R$-modules. One can find $u_M$ and $v_M$ 
 such that
 \begin{equation} \label{Eq:MFoverB} {\begin{tikzpicture}[baseline=(current  bounding  box.center)] 

 \node (C1) at (0,0) {$0$} ;
 \node (C1o) at (0,1.5)  {$0$};
 \node (C2) at (2,0)  {$M_-$};
 \node (C2o) at (2,1.5)  {$M_-$};
 \node (C3) at (4,0)  {$M_+$};
 \node (C3o) at (4,1.5)  {$M_-$};
 \node (C4) at (6,0)  {$C$};
 \node (C4o) at (6,1.5)  {$0$};
 \node (C5) at (8,0) {$0$} ;
 \node (C5o) at (8,1.5)  {$0$};


 \draw [->,bend left=25,looseness=1,pos=0.5] (C2) to node[]  [left]{$f$} (C2o);
 \draw [->,bend left=20,looseness=1,pos=0.5] (C2o) to node[] [right] {$\id$} (C2);
 \draw [->,bend left=25,looseness=1,pos=0.5] (C3) to node[]  [left]{$v_M$} (C3o);
 \draw [->,bend left=20,looseness=1,pos=0.5] (C3o) to node[] [right] {$u_M$} (C3);
 \draw [->,bend left=25,looseness=1,pos=0.5] (C4) to node[]  [left]{$0$} (C4o);
 \draw [->,bend left=20,looseness=1,pos=0.5] (C4o) to node[] [right] {$0$} (C4);

 \draw[->]  (C1) to  [] node[left] [above]{}(C2);
 \draw[->]  (C2) to  [] node[left] [above]{$u_M$}(C3);
 \draw[->]  (C3) to  [] node[left] [above]{}(C4);
 \draw[->]  (C4) to  [] node[left] [above]{}(C5);

 \draw[->]  (C1o) to  [] node[left] [above]{}(C2o);
 \draw[->]  (C2o) to  [] node[left] [above]{$\id$}(C3o);
 \draw[->]  (C3o) to  [] node[left] [above]{}(C4o);
 \draw[->]  (C4o) to  [] node[left] [above]{}(C5o);
 \end{tikzpicture}} 
 \end{equation}
 is a short exact sequence of $B$-modules, that is, $u_M v_M=v_M u_M=f$. 
 The leftmost column of this diagram corresponds to the $B$-module $(M_-\qurep{f}{\id}M_-)=j_{!}j^*M $, which is isomorphic to a direct summand of $(e_-B)^m$ for some $m \geq 0$, since $M_-$ is projective over $R$. In particular, that $B$-module is projective and the $B$-module $M=(M_+\qurep{v}{u}M_-)$ is in $\CM(B)$ and $\coker(u)=i^*(M)=C$. This shows that $i^*$ is a dense functor from $\CM(B)$ to $\CM(R/(f))$. Note that we just established that there is a short exact sequence of functors
 \begin{equation} \label{Eq:exactFunctors}
  0 \xto{}  j_! j^* \xto{}  \id_{\CM(B)} \xto{} i_* i^* \xto{} 0 
  \end{equation}
 from $\CM(B)$ to $\fmod(B)$, where $\fmod(B)$ stands for the category of finitely generated $B$-modules. Here $\id_{\CM(B)} \xto{} i_* i^* $ is the restriction of the unit of the adjunction $(i^*,i_*)$ to $\CM(B)$. Further, $i^*(e_- B)=0$, whence $i^*$ factors through the quotient $\CM(B)/ \langle e_- B \rangle$. \\
  From the exact sequence \eqref{Eq:exactFunctors} one easily sees that the functor $\CM(B)/ \langle e_-B \rangle \xto{} \CM(R/(f))$ induced by $i^*$ is fully faithful.
 \end{proof}


Interpreting Theorem \ref{Thm:Eisenbud} in terms of matrix factorizations, note that $\CM(B) \simeq MF(f)$,  the category of matrix factorizations of $f$. Let $\cali$ be the ideal in the category $MF(f)$ generated by the matrix factorization $R\qurep{f}{\id}R$. 
 If $f$ is a non-zero divisor in $R$ then by the above result, the functor coker$(u): MF(f) \longrightarrow \CM(R/(f))$ induces an equivalence of categories
 \begin{equation} \label{Eisenbud:MF} MF(f)/\cali \simeq \CM(R/(f)),
 \end{equation}
 which is a reformulation of \cite[Section 6]{Eisenbud80}.

\subsubsection{Reformulation in terms of the skew group ring} \label{Subsub:Knoerrer}
Let $R$ and $f\in R$ be as in Theorem \ref{Thm:Eisenbud}.
Let $T:=R[Z]/(Z^2-f)$, so that $\Spec(T)$ is the double cover of $\Spec(R)$ ramified over $V(f)=\{ f=0 \}$. The canonical $R$-involution on $T$ that sends $Z$ to $-Z$ defines a group action of $\mu_2=\langle \sigma \mid \sigma^2=1 \rangle$ on $T$. Let $T * \mu_2$ be the corresponding twisted group algebra.  In the next lemma, we describe the quiver structure of $T*\mu_2$.

\begin{lemma}  \label{Lem:Bequiv} 
Consider $T*\mu_2$, as just described, and suppose 2 is invertible. Then $T*\mu_2 \cong R\langle Z, \delta_\sigma \rangle / \langle Z^2 - f, \delta_\sigma Z + Z \delta_\sigma , \delta_\sigma^2-1 \rangle$. Furthermore, with
 $e_\pm = \frac{1}{2}(1 \pm \delta_\sigma)$ and $u =  \frac{1}{2}(1 + \delta_\sigma)Z$ and $v = \frac{1}{2}(1 - \delta_\sigma)Z$,  one has %
 $$ {\begin{tikzpicture}
\node (B) at (-0.45,0) {$T*\mu_2 \cong R\,$};
 
 \node (C1) at (0.9,0) {$e_+$} ;
 \node (C2) at (5.1,0)  {$e_- $};

 \draw [thick ] (0.7,-0.7) to [round left paren ] (0.7,0.7);
 \draw [thick ] (5.3,-0.7) to [round right paren] (5.3,0.7);

 \draw [->,bend left=25,looseness=1,pos=0.5] (C1) to node[]  [above]{$v$} (C2);
 \draw [->,bend left=20,looseness=1,pos=0.5] (C2) to node[] [below] {$u$} (C1);

 \node (DD) at (5.7,0) {.};
      \end{tikzpicture}}
 $$  with relations $uv=f e_+, vu =f e_-$. 
In particular, $T * \mu_2$ is the path algebra of a quiver as in \eqref{Eq:Bquiver}.

 \end{lemma}

\begin{proof}
All statements are easily verified using the fact that $B=T*\mu_2$ is a free $R$-module with basis $1, Z, \delta_\sigma, Z\delta_\sigma$ and the path algebra is a free $R$-module with basis $e_+, e_-, u, v$.
\end{proof}

\begin{remark} \label{Rmk:Knoerrer}
With $T=R[Z]/(Z^2-f)$ and the isomorphism of $T*\mu_2$ with the path algebra $B$ as in \eqref{Eq:Bquiver}  (Lemma \ref{Lem:Bequiv}) and Theorem \ref{Thm:Eisenbud} we have
 \begin{equation}\label{eqn:equiv}  \CM(T*\mu_2)/ \langle e_-B \rangle \simeq \CM(R/(f)) \simeq MF(f)/\cali \ ,
 \end{equation}
which is a reformulation of the equivalence \eqref{Eisenbud:MF}.
 Assuming that $2$ is a unit in $R$, and expressing the recollement in terms of $T*\mu_2$,
 one regains the functors in \cite{KnoerrerCohenMacaulay}. In particular, Theorem \ref{Thm:Eisenbud} implies Horst Kn\"orrer's result
\[ \CM(T *\mu_2 ) \simeq  MF(f) \, \]
as established in \cite[Prop. 2.1]{KnoerrerCohenMacaulay}. 
\end{remark}

\subsection{The skew group ring and Bilodeau's isomorphisms}
\label{Sub:Bilodeau}
The results here were inspired by work of Jos\'ee Bilodeau \cite{Bilodeau}.  In the following, $\mathbb{K}$ is a commutative ring and $G$ a finite group such that the order $|G|$ of $G$ 
is invertible in $\mathbb{K}$. Set $e_{G}=\tfrac{1}{|G|}\sum_{g\in G}g\in \mathbb{K}G$, the idempotent in the group 
algebra that belongs to the trivial representation of $G$. Similarly, for a subgroup $H\leqslant G$ 
we set $e_{H}= \tfrac{1}{|H|}\sum_{h\in H}h\in \mathbb{K}G$ and say that this idempotent element in
$\mathbb{K}G$ is defined by $H$.

If $\Gamma,H\leqslant G$ are complementary subgroups in that $\Gamma\cap H=\{1\}$, where 
$1\in G$ is the identity element, and $H\Gamma =G$, then every element $g\in G$ can be written 
uniquely as $g=h \gamma$ with $h\in H, \gamma\in \Gamma$, and also uniquely as 
$g=\gamma'h'$ with $\gamma'\in \Gamma, h'\in H$. 

Note that one has $e_{G}= e_{H}e_{\Gamma}=e_{\Gamma}e_{H}$ in $\mathbb{K}G$.

\begin{lemma}
\label{Gammainv}
Let $M$ be a left $\mathbb{K}\Gamma$--module. The $\mathbb{K}$--submodule
$M^{\Gamma}=\{m\in M\mid  \gamma m = m \text{\ for each\ } \gamma\in \Gamma\}$ equals 
$e_{\Gamma}M$.
\end{lemma}

\begin{proof}
If $\gamma m= m$ for each $\gamma\in\Gamma$, then 
$\left(\sum_{\gamma \in \Gamma} \gamma\right)m=|\Gamma| m$,
that is, $e_{\Gamma}m = m$, and so $M^{\Gamma}\subseteq e_{\Gamma}M$. On the other hand, 
$\gamma e_{\Gamma}=e_{\Gamma}$ for each $\gamma \in \Gamma$, thus, 
$e_{\Gamma}M\subseteq M^{\Gamma}$.
\end{proof}

\begin{cor}
\label{invsubring}
If $\Gamma$ acts through $\mathbb{K}$--algebra automorphisms on a $\mathbb{K}$--algebra $S$, then 
$T: = S^{\Gamma}= e_{\Gamma}S$ is a $\mathbb{K}$--subalgebra of $S$.
\end{cor}
\begin{proof}
This is obvious from the description 
$T=S^{\Gamma}=\{s\in S\mid \gamma s = s \text{\ for each\ } \gamma\in \Gamma\}$.
\end{proof}

\begin{lemma} \label{Lem:HtensorS}
With notation as before, let $\Gamma\leqslant G$ be a normal subgroup and set $T= S^{\Gamma}$ 
and $H=G/\Gamma$. The quotient group $H$ acts naturally on $T$ through $\mathbb{K}$--algebra 
automorphisms and one can form $T*H$ accordingly. There is a natural isomorphism
$S\otimes_\mathbb{K} \mathbb{K}H \cong Ae_{\Gamma}$ as right $T*H$--modules and a $\mathbb{K}$--algebra isomorphism 
$T{*} H\cong e_{\Gamma}Ae_{\Gamma}$, where $A=S*G$, as before.  

\end{lemma}

\begin{proof} Note that for $\Gamma$ normal in $G$ it holds that $g e_{\Gamma}=e_{\Gamma}g$ for all $g\in G$, 
thus, $e_{\Gamma}$ is then a central idempotent.
Further, $\gamma e_{\Gamma} = e_{\Gamma} =  e_{\Gamma} \gamma$,
whence the element $g e_{\Gamma}= e_{\Gamma} g$ depends solely on the coset $g \Gamma$. 
In that way, $h e_{\Gamma} = e_{\Gamma} h$ is a well--defined element of $\mathbb{K}G$ for any $h\in H$.

Accordingly, the map $S\otimes_{\mathbb{K}} \mathbb{K}H\to Ae_{\Gamma}$ that sends 
$s\otimes h\mapsto s(he_{\Gamma})\in Ae_{\Gamma}$ is well defined. It is bijective as
for $a=\sum_{g\in G}s_{g}\delta_{g}\in A$ one has
$$
ae_{\Gamma} = \sum_{g\in G}s_{g}\delta_{g}e_{\Gamma} = \sum_{g\Gamma\in H}\sum_{\gamma\in\Gamma}s_{g\gamma}\delta_{g}\delta_{\gamma}e_{\Gamma} = \sum_{h=g\Gamma\in H}\left(\sum_{\gamma\in\Gamma}
s_{g\gamma}\right)(he_{\Gamma})$$
whence $ae_{\Gamma}\mapsto \sum_{h=g\Gamma\in H}(\sum_{\gamma\in\Gamma}
s_{g\gamma})\otimes h$ yields the inverse map. It also follows from this calculation that
$\sum_{h'\in H}t_{h'}\delta_{h'}\in T*H$ acts from the right on $Ae_{\Gamma}$ by 
\begin{align*}
ae_{\Gamma}\left(\sum_{h'\in H}t_{h'}\delta_{h'}\right)&= \left(\sum_{h=g\Gamma\in H}s_{h}he_{\Gamma}\right)
\left(\sum_{h'\in H}t_{h'}\delta_{h'}\right)\\
&=\sum_{h,h'\in H}s_{h}h(t_{h'})(hh'e_{\Gamma})\\
&=\sum_{h''\in H}\left(\sum_{hh'=h''}s_{h}h(t_{h'})\right)h''e_{\Gamma}\,,
\end{align*}
where we have used that $e_{\Gamma}t = te_{\Gamma}$ and $e_{\Gamma}h = he_{\Gamma}$
for $t\in T$ and $h\in H$.

Transporting this structure to $S\otimes_\mathbb{K}  \mathbb{K}H$ under the bijection onto $Ae_{\Gamma}$, we obtain that
$(s\otimes h)\sum_{h'\in H}t_{h'}\delta_{h'} = \sum_{h'}sh(t_{h'})\otimes hh'$ defines the right
$T*H$--module structure on $S\otimes_\mathbb{K} \mathbb{K}H$ that makes the bijection above $T*H$--linear.

Furthermore, that bijection is $\Gamma$--equivariant with respect to the left $\Gamma$--actions
$\gamma(s\otimes h) = \gamma(s)\otimes h$ and $\gamma (a e_{\Gamma}) = 
\delta_{\gamma}a e_{\Gamma}\in Ae_{\Gamma}\subseteq A$. Taking $\Gamma$--invariants
returns the isomorphism of right $T*H$--modules
\begin{align*}
(S\otimes H)^{\Gamma} &\cong S^{\Gamma}\otimes H = T\otimes H
\intertext{and}
(Ae_{\Gamma})^{\Gamma} &= e_{\Gamma}Ae_{\Gamma}\,,
\intertext{whence}
T\otimes H&\cong e_{\Gamma}Ae_{\Gamma}\,.
\end{align*}

For all $h \in H$ choose a lift $h'$ in $G$ so that $h'\Gamma =h$.
The morphism
\begin{eqnarray*} T*H & \xrightarrow{} & e_\Gamma A e_\Gamma \\
  \sum_{h \in H} t_h \delta_h & \mapsto & e_\Gamma t_h \delta_{h'}e_\Gamma
\end{eqnarray*} 
is well defined since $\delta_{h'}e_\Gamma$ does not depend on the choice of coset
representative.  It is clearly an algebra homomorphism since $e_\Gamma t_h =e_\Gamma t_h$ and $e_\Gamma \delta_{h'} = \delta_{h'}e_\Gamma$. It is bijective since it identifies with the morphism above from $T\otimes H \to e_\Gamma A e_\Gamma.$
\end{proof}

\begin{remark}\label{rem:SstarH}
If $\Gamma\leqslant G$ admits a complement, necessarily isomorphic to $H$, then the natural 
$\mathbb{K}$--algebra homomorphism $T*H\to S*H$ induces the $T*H$--module structure on 
$S*H\cong S\otimes_\mathbb{K} \mathbb{K}H$ described in the above proof.
\end{remark}
Now we come to the key result.
\begin{proposition}\label{Phi}
Let $\Gamma,H\leqslant G$ be complementary subgroups with $\Gamma$ normal in $G$.
With $G$ acting through $\mathbb{K}$--algebra automorphisms on some $\mathbb{K}$--algebra $S$ and with
$T=S^{\Gamma}$, 
the group $H$ acts naturally on $\End_{T}(S)$ through
algebra automorphisms and there is an isomorphism of $\mathbb{K}$--algebras
$\Phi\colon \End_T(S)*H\xto{\ \cong\ }\End_{T{*} H}(S*H)$, where $S*H$ is considered a right $T*H$--module.
\end{proposition}

\begin{proof}
If $h\in H$ and $\alpha\in \End_{T}(S)$, then $(h\alpha)(s) = h(\alpha(h^{-1}(s)))$ defines the action 
of  $H$ on $\End_{T}(S)$ through algebra automorphisms. Namely, $h\alpha$ is $T$--linear because
\begin{align*}
(h\alpha)(st)&= h(\alpha(h^{-1}(st)))\\
&=h(\alpha(h^{-1}(s)h^{-1}(t)))
\intertext{as $H$ acts through algebra automorphisms on $S$,}
&=h(\alpha(h^{-1}(s))h^{-1}(t))
\intertext{as $\alpha$ is $T$--linear and $h^{-1}(t)\in T$,}
&=h(\alpha(h^{-1}(s)))h(h^{-1}(t))\\
&=(h\alpha)(s)t\,.
\intertext{That $H$ acts through algebra automorphisms on $\End_{T}(S)$ follows from}
 (h(\alpha\beta))(s) &= h(\alpha\beta(h^{-1}(s)))\\
 &=h(\alpha(h^{-1}(h\beta h^{-1}(s))\\
 &= (h\alpha)(( h\beta)(s))\,.
\end{align*}
Accordingly one can form the twisted group algebra  $\End_{T}(S)*H$ as in Definition 
\ref{skewgroupring}.

The map $\Phi$ sends $\alpha=\sum_{h\in H}\alpha_{h}\delta_{h}$, 
with $\alpha_{h}\in  \End_T(S)$, to the map 
$\Phi(\alpha)\colon S*H\to S*H$ defined by 
\begin{align*}
\Phi(\alpha)\left(\sum_{h'\in H}s_{h'}\delta_{h'}\right) &= \left(\sum_{h\in H}\alpha_{h}\delta_{h}\right)
\left(\sum_{h'\in H}s_{h'}\delta_{h'}\right)\\
&=\sum_{h,h'\in H}\alpha_{h}(h(s_{h'}))\delta_{h}\delta_{h'}\\
&=\sum_{h''\in H}\left(\sum_{hh'=h''}\alpha_{h}(h(s_{h'}))\right)\delta_{h''}\,.
\end{align*}
To show that $\Phi$ is a homomorphism of $\mathbb{K}$--algebras, with
$\beta=\sum_{h'\in H}\beta_{h'}\delta_{h'}\in \End_{T}(S)*H$ one finds first
\begin{align*}
\alpha\beta &= \sum_{h,h'\in H}\alpha_{h}h(\beta_{h'})\delta_{hh'}
\intertext{and then}
\Phi(\alpha\beta)\left(\sum_{h''\in H}s_{h''}\delta_{h''}\right) &=
\Phi\left(\sum_{h,h'\in H}\alpha_{h}h(\beta_{h'})\delta_{hh'}\right)\left(\sum_{h''\in H}s_{h''}\delta_{h''}\right)\\
&=\sum_{h,h',h''\in H}(\alpha_{h}h(\beta_{h'}))((hh')(s_{h''}))\delta_{hh'h''}\,,
\intertext{whereas}
\Phi(\alpha)\Phi(\beta)\left(\sum_{h''\in H}s_{h''}\delta_{h''}\right) 
&=\Phi(\alpha)\left(\sum_{h',h''\in H}\beta_{h'}(h'(s_{h''}))\delta_{h'h''}\right)\\
&=\sum_{h,h',h''\in H}\alpha_{h}(h(\beta_{h'}(h'(s_{h''})))\delta_{hh'h''}\\
&=\sum_{h,h',h''\in H}\alpha_{h}(h(\beta_{h'})(h(h'(s_{h''}))))\delta_{hh'h''}\\
&=\sum_{h,h',h''\in H}(\alpha_{h}h(\beta_{h'})((hh')(s_{h''}))\delta_{hh'h''}\,.
\end{align*}
Thus, $\Phi(\alpha\beta)=\Phi(\alpha)\Phi(\beta)$ as claimed.

To check that $\Phi(\alpha)$ constitutes an $T{*} H$--linear endomorphism of
$S*H$ it suffices to note that there is a commutative diagram of homomorphisms of $\mathbb{K}$--algebras
\begin{align*}
\xymatrix{
\End_{T}(S)*H\ar[rr]^-{\Phi}&&\End_{T*H}(S*H)\\
&T*H\ar[ul]^{\vp}\ar[ur]_{\psi}
}
\end{align*}
where $\vp$ is induced by the $\mathbb{K}$--algebra homomorphism $T\to \End_{T}(S)$ that sends $t\in T$
to $\lambda_{t}$, the (left) multiplication by $t$ on $S$, and $\psi$ represents left multiplication by 
$T*H$ on $S*H$.
Indeed,
\begin{align*}
\Phi\vp(t\delta_{h})\left(\sum_{h'\in H}s_{h'}\delta_{h'}\right) &= \Phi(\lambda_{t}\delta_{h})\left(\sum_{h'\in H}s_{h'}\delta_{h'}\right)\\
&=\sum_{h'\in H}t h(s_{h'})\delta_{hh'}\\
&=(t\delta_{h})\left(\sum_{h'\in H}s_{h'}\delta_{h'}\right) \\
&= \psi(t\delta)\left(\sum_{h'\in H}s_{h'}\delta_{h'}\right) \,.
\end{align*}
Finally, we show that $\Phi$ is an isomorphism by exhibiting the inverse.
Let $f:S*H\to S*H$ be a right $T*H$--linear map. Then
\begin{align*}
f\left(\sum_{h\in H}s_{h}\delta_{h}\right)&= \sum_{h\in H}f(s_{h}\delta_{1})\delta_{h}
\end{align*}
as $f$ is $T*H$--linear. Therefore, $f$ is uniquely determined by 
$f(s_{h}\delta_{1}) = \sum_{h\in H}f_{h}(s)\delta_{h}$, where in turn $f_{h}(s)\in S$ is uniquely determined as
the $\delta_{h}$ form a basis of the (right) $S$--module $S*H$. Now $f$ is $T$--linear on the right, 
whence necessarily for any $s\in S, t\in T$ the expression
\begin{align*}
f(st\delta_{1})&=  \sum_{h\in H}f_{h}(st)\delta_{h}
\intertext{equals}
f(s\delta_{1})t&=  \left(\sum_{h\in H}f_{h}(s)\delta_{h}\right)t\\
&=\sum_{h\in H}f_{h}(s)h(t)\delta_{h}
\end{align*}
Comparing coefficients of $\delta_{h}$ it follows that $f_{h}(st) = f_{h}(s)h(t)$ 
for each $h\in H$. This implies that the map $\alpha_{h}(s) = f_{h}(h^{-1}(s))$ is in $\End_{T}(S)$
and $\Psi(f) = \sum_{h\in H}\alpha_{h}\delta_{h}$ yields the inverse of $\Phi$.
Indeed, 
\begin{align*}
\Phi\Psi(f)\left(\sum_{h'\in H}s_{h'}\delta_{h'}\right) &=
\Phi(\sum_{h\in H}\alpha_{h}\delta_{h})(\sum_{h'\in H}s_{h'}\delta_{h'}) \\
&=\sum_{h''\in H}\left(\sum_{hh'=h''}\alpha_{h}(h(s_{h'}))\right)\delta_{h''}\\
&=\sum_{h''\in H}\left(\sum_{hh'=h''}(f_{h}(h^{-1}(h(s_{h'}))))\right)\delta_{h''}\\
&=\sum_{h, h'\in H}f_{h}(s_{h'})\delta_{h}\delta_{h'}\\
&=f\left(\sum_{h'\in H}s_{h'}\delta_{h'}\right)\,.
\end{align*}
One checks analogously that $\Psi\Phi(\alpha)=\alpha$ for any $\alpha\in \End_{T}(S)*H$.
\end{proof}

To sum up, let us interpret the preceding result in terms of $A=S*G$:

\begin{proposition} \label{Prop:BintermsofA}
  Let $G$ be a finite group, and let $\Gamma$ be a split normal subgroup with complement $G/\Gamma \cong H \leq G$.  Let $G$ act linearly on $S=\mathbb{K}[V]$, and $A=S*G$, $T =S^\Gamma$, $e_\Gamma = \frac{1}{|\Gamma|}(\sum_{\gamma \in \Gamma}\delta_\gamma),$ then 
\begin{enumerate}[\rm(a)]
\item \label{TstarH}
$T*H \cong e_{\Gamma}A e_{\Gamma}$, as an isomorphism of $\mathbb{K}$--algebras.
\item \label{SstarH}
$S*H \cong Ae_{\Gamma}$, as an isomorphism of right $e_{\Gamma}A e_{\Gamma}$--modules.
\item \label{fromLam} $A = S*G\cong (S*\Gamma)*H$, as $\mathbb{K}$--algebras.
\end{enumerate}
Let $\lambda : S*\Gamma \to \End_{T}(S)$ be the natural homomorphism.  Then
the composition of the sequence of $\mathbb{K}$--algebra homomorphisms
\begin{align*}
(S*\Gamma)*H\cong A\to \End_{e_{\Gamma}A e_{\Gamma}}(Ae_{\Gamma})\cong 
\End_{T*H}(S*H)\xto[\ \cong\ ]{\Psi} \End_{T}(S)*H\,.
\end{align*}
is $\lambda*H$. 
\end{proposition}
\begin{proof}
  Statement (\ref{TstarH}) is in Lemma~\ref{Lem:HtensorS}.
Statement (\ref{SstarH}) is noted in Remark~\ref{rem:SstarH}.
  Statement (\ref{fromLam}) is~\cite[Ex.~1.11]{lam2001first}. 
In the sequenece of maps above, we see that 
left multiplication by elements of $A$ defines a $\mathbb{K}$--algebra homomorphism
\[
A\to \End_{e_{\Gamma}A e_{\Gamma}}(Ae_{\Gamma}).\]
The first isomorphism in 
\[\End_{e_{\Gamma}A e_{\Gamma}}(Ae_{\Gamma}) \cong \End_{T{*}H}(S*H)
\xto[\ \cong\ ]{\Psi}\End_{T}(S)*H\, 
\]
follows from (\ref{TstarH}) and (\ref{SstarH}) above and the second isomorphism is Proposition~\ref{Phi}.
Moreover, as $\Gamma, H$ are complementary subgroups in $G$ and any skew group ring is isomorphic to its
opposite, the sequence of ring homomorphisms above identifies with
$\lambda * H.$
\end{proof}

\begin{remark} \label{StarHmakessense}
  By Auslander's Theorem~\ref{thm:Auslander}, when $\mathbb{K}$ is a field, if $G\leqslant \GL(V)$ and $S=\mathbb{K}[V]$ or $S=\mathbb{K}[[V]]$,
then $\lambda$ as in Proposition~\ref{Prop:BintermsofA}, and as a consequence also $\lambda*H$, are isomorphisms if $\Gamma$ contains no pseudo--reflections in its linear action on $S$. Thus the above result extends Auslander's theorem to the case where $G \leqslant \GL(V)$ is a pseudo-reflection group. Here $\Gamma=G \cap \SL(V)$ is small and $H$ is the quotient $G/\Gamma$ in the exact sequence $1 \xrightarrow{} \Gamma \xrightarrow{} G \xrightarrow{} G/\Gamma \xrightarrow{} 1$. 
\end{remark}

\subsection{Intermezzo: Specializing to reflection groups}

\subsubsection{The Invariant ring $S^\Gamma$ in terms of $S^G$}

In the following let $V$ be a finite dimensional vector space over $K$ and $G \leqslant \GL(V)$ be a true reflection group. Set $\Gamma:= G \cap SL(V)$ and $H:=\det G \cong \{\pm 1\}=\langle \sigma \rangle$. This means that we have an exact sequence of groups
$$1 \longrightarrow \Gamma \longrightarrow G \xrightarrow{\det|_G} H \longrightarrow 1.$$
This sequence splits (by definition $G$ is generated by pseudo-reflections).
Let $S=K[x_1, \ldots, x_n]$, $S^\Gamma$ the invariant ring of $\Gamma$, $S^G=K[f_1, \ldots, f_n] \subseteq S$ the invariant ring of $G$,  and $J=\det \left((\frac{\partial f_i}{\partial x_j})_{ij} \right)$ the Jacobian of $G$. Note that, since $G$ is generated by order $2$ reflections, $J$ is equal to $z$, the polynomial defining the hyperplane arrangement of $G$ and the discriminant of $G$ is $\Delta = z^2 \in R$.  \\

\begin{lemma} \label{Lem:RasTmod}
The invariant ring
 $S^\Gamma$ satisfies $S^\Gamma \cong S^G \oplus JS^G$  
 as an $S^G$-module and $S^\Gamma \cong S^G[J]/(J^2 - \Delta)$ as rings.
\end{lemma}

\begin{proof}
This  follows from Stanley \cite{StanleyInvariants}: let $S^G_{\chi}$ be the set of invariants relative to the linear character $\chi$, i.e., $S^G_\chi = \{ f \in S: g(f)=\chi(g) f$ for all $g \in G\}$.  In Lemma 4.1 loc.~cit.~it is shown that
\[S^\Gamma = S^G_{\mathrm{triv\phantom{^{1}}}}\! \! \oplus S^G_{\det^{-1}}\]
as $S^G$-modules, where $\mathrm{triv}$ denotes the trivial character and $\mathrm{det}^{-1}$ denotes the inverse of the determinantal character.  Since $S^G_{\mathrm{triv}} = S^G$ and $S^G_{\det^{-1}}$ is generated by $J=z$ over $S^G$ (see either \cite{StanleyInvariants} or \cite[Chapter 6]{OTe}), it follows that 
$$S^\Gamma \cong S^G \oplus JS^G \ .$$
From Stanley's description of $S^\Gamma$ as $S^G$-module, we also see how $H=G/\Gamma \cong \mu_2=\langle \sigma \rangle$ acts on $S^G[J]$: $\sigma$ is the identity on $S^G$ and $\sigma(J)=\det^{-1}(\sigma)(J)=-J$, since the Jacobian is a semi-invariant for $\det^{-1}$ of the reflection group. 
\end{proof}

\begin{corollary}
The skew group ring $S^\Gamma*H$ is isomorphic to the path algebra of a quiver, as in \eqref{Eq:Bquiver} (in the notation of Section \ref{Subsub:Knoerrer}: $R=S^G$, $\mu_2=H$, $f=\Delta$, $Z=J$, and $T=S^\Gamma$)
\end{corollary}

\begin{proof}
By Lemma \ref{Lem:RasTmod}, $S^\Gamma \cong S^G[J]/(J^2-\Delta)$. The rest follows as in Section \ref{Subsub:Knoerrer}.
\end{proof}

\begin{remark} If $G$ is a pseudo-reflection group, then by \cite{StanleyInvariants}, the module of relative invariants $S^G_{\det}$ is generated by $z$, the reduced equation for the hyperplane arrangement and $S^G_{\det^{-1}}$ is generated by the Jacobian $J$. Then the relation for the discriminant is $zJ= \Delta$ (see \cite{OTe}, Examples 6.39, 6.40 and Def. 6.44). 
\end{remark}

\subsubsection{The hyperplane arrangement $S/(J)$}

Let $G$ be a true reflection group in $\GL(V)$, and let $H =\langle\sigma\rangle$ be the split subgroup $H\cong\det G = \mu_2,$ with complement
$\Gamma = \SL(V) \cap G$.\\
 In the following we use the notation as suggested in Section \ref{Subsub:Knoerrer}. Let $S=K[V]$, $T=S^\Gamma$ and $R=S^G$. Further write $A=S*G$, $B=T*H$ and set $e=\frac{1}{|G|}\sum_{g \in G}\delta_g$, $e_\Gamma=\frac{1}{|\Gamma|}\sum_{\gamma \in \Gamma}\delta_\gamma$, $e_{-}=\frac{1}{2}(1 - \delta_\sigma)$ and the (inverse) determinantal idempotent $\overline{e}= \frac{1}{|G|}\sum_{g \in G} \det^{-1}(g) \delta_g$.
Here  we show how the module $S/(J)$ over the discriminant $R/(\Delta)$ can be seen as the image of the $B$-module  $Ae_\Gamma$.  

\begin{proposition} \label{Prop:Smodz}
Denote by $i^*=-\otimes_{B}B/Be_{-}B:\Mod(B) \xto{} \Mod(B/Be_{-}B)$ the standard recollement functor. Then $i^* Ae_{\Gamma} \cong S/(J)$ as $B/Be_{-}B\cong R/(\Delta)$-module.
\end{proposition}

\begin{proof}
First compute 
$$i^* Ae_\Gamma  = Ae_\Gamma \otimes_B B/Be_{-}B \cong  Ae_\Gamma / Ae_\Gamma e_{-}B.$$
Since $e_- e_\Gamma = e_\Gamma e_- =\overline{e}$ and $B \cong e_\Gamma A e_\Gamma$ (see Proposition~\ref{Prop:BintermsofA}), this is 
$$ A e_\Gamma / A  e_\Gamma e_{-} B \cong  A e_\Gamma /  A  e_\Gamma  e_- e_\Gamma A e_\Gamma \cong A e_\Gamma /  A  (e_\Gamma e_{-}) (e_- e_\Gamma) A e_\Gamma \cong  A e_\Gamma /  A \overline{e} A e_\Gamma \cong  (A/A\overline{e}A)e_\Gamma.$$
Consider the trivial idempotent $e$ in $A$. Since $\Gamma$ is of index $2$ in $G$ and $H$ is the cokernel of $\Gamma \xto{} G$, it follows that $e +\overline{e} = e_\Gamma$ and with $\overline{A}=A/A\overline{e}A$ one sees that $\overline{A} e_\Gamma \cong  \overline{A}e$. From  Lemma \ref{Lem:quotienrelative} it follows that $\overline{A}e \cong S/(J)$, since $J$ generates the $R$-module of relative invariants for $\chi=\det^{-1}$.
\end{proof}

\subsection{The main theorem}

\begin{theorem} \label{Thm:main}
Let $G \leqslant \GL(V)$ be a finite true reflection group with $H =\langle \sigma \rangle \cong \det G  = \mu_2.$
and set $\Gamma= G \cap SL(V)$. Let $T=S^\Gamma$, $R=S^G \subseteq S$,  $J$ the Jacobian of $G$ and the discriminant $\Delta = J^2 \in R$. Further denote by $A= S*G$ the skew group ring, $\overline{A} = A/ Ae_{\chi}A$, with $e_{\chi} \in A$ an idempotent for a linear representation $\chi$, and $B=T*H$. Then:
\begin{enumerate}[\rm(a)]
\item \label{catEquiv} Then there is an equivalence of categories
$$\CM (R / \Delta) \simeq \CM(B)/ \langle e_{-}B \rangle,$$
where $e_{-}$ is the idempotent $e_- = \frac{1}{2}(1 - \delta_{\sigma})$ in $B$. 
\item \label{skgrprng} The skew group ring $A$ is isomorphic to $\End_{B}( A e_{\Gamma})=\End_{B}(S*H)$ , where $e_{\Gamma}=\frac{1}{|\Gamma|}\sum_{\gamma \in \Gamma}\delta_\gamma$. 
\item \label{mainthm} The quotient algebra $\overline{A} = A/ Ae_{\chi}A$ is isomorphic to $\End_{R / \Delta}(i^*(Ae_\Gamma))$, where $i^*$ comes from the standard recollement of $\fmod B$, $\fmod Be_-B$ and $\fmod B/Be_-B$. 
\item \label{iAgammaisSmodJ} The $R/( \Delta)$-module $i^*(Ae_\Gamma )$ is isomorphic to $S/(J)$, which implies that 
$$\overline{A}   \cong \End_{R/(\Delta)}(S/(J)).$$
\end{enumerate}
\end{theorem}

\begin{proof}
  Without loss of generality we may assume that $e_\chi=\overline{e}=\frac{1}{|G|}\sum_{g \in G}\det^{-1}(g)\delta_g$, cf.~Cor.~\ref{Cor:quotientsiso}. As noted in Remark~\ref{Rmk:Knoerrer},
  $$i^*\!: \Mod (B) \xrightarrow{- _B\otimes B/Be_-B  } \Mod(B/Be_-B),$$
  induces an equivalence
\[ \CM(B)/\langle e_- B  \rangle  \simeq \CM (B/ Be_- B). \]

Since $B/Be_-B$ is isomorphic to $R/(\Delta)$ (see Lemma~\ref{Cor:Bprops}, \eqref{DiskQuiv}), it follows that 
$$i^*: \CM(B)/\langle  e_- B \rangle \xrightarrow{\simeq}    \CM(R/(\Delta)),$$
establishing statement (\ref{catEquiv}).
Statement (\ref{skgrprng}) follows from Proposition~\ref{Prop:BintermsofA} and Remark~\ref{StarHmakessense}.
\\By Prop.~\ref{Prop:BintermsofA}, we have that $A = S*G$ is isomorphic to $\Hom_B(S \otimes H, S \otimes H)$.  
Since $i^*$ is an equivalence, it follows that
$$i^*(\Hom_B(S \otimes H, S \otimes H)) \cong \Hom_{R/\Delta}(i^*(S\otimes H), i^*(S \otimes H)).$$
  Now using $S \otimes H \cong  Ae_\Gamma$ (as right $B$-module) from Lemma \ref{Lem:HtensorS} yields that $i^*(S \otimes H)=i^* A e_\Gamma= S/(J)$ by Prop.~\ref{Prop:Smodz}, establishing the first statement of (\ref{iAgammaisSmodJ}). 
Thus in total we get 
$$i^*A  {\cong}  \End_{R/(\Delta)}(S/(J))$$
in $\CM(R/(\Delta))$.
  To complete the proof of (\ref{iAgammaisSmodJ}), we need to establish (\ref{mainthm}).  To this end, we first claim that there is a ring isomorphism
  \begin{equation} \label{claim} \overline{A} \cong \End_{\CM(B)/\langle  e_- B \rangle} (Ae_\Gamma).
  \end{equation}
Note that the known equivalence of categories noted in Remark~\ref{Rmk:Knoerrer},
induces an isomorphism of rings
$$\overline{A} \cong \End_{\CM(B)/\langle  e_- B \rangle } (Ae_\Gamma) \xrightarrow{\cong} \End_{R/(\Delta)}(i^*(Ae_\Gamma))$$
which gives us statement (\ref{mainthm}). To establish the claim (\ref{claim}) we need to show that the ideal $\langle e_-B\rangle$ in $A \cong \End_{\CM(B)}(Ae_\Gamma)$ is equal to $AeA$. 

So we compute the image of $A=\Hom_B( A e_\Gamma, A e_\Gamma )$ in $\CM(B)/ \langle e_-B \rangle$: we have to identify all morphisms $A e_\Gamma \xrightarrow{} A e_\Gamma$ that factor through copies of $e_-B$. These are sums of elements of the form $\alpha \circ \beta$ with  $\alpha \in \Hom_B(e_- B,  Ae_\Gamma)$ and $\beta \in \Hom_B( Ae_\Gamma, e_-B)$. Since $e_-$ is an idempotent, it follows e.g. from \cite[Lemma 4.2]{Assem06} that the first Hom is isomorphic (as right $e_-Be_- = R$-modules)
\[ \Hom_B(e_-B, Ae_\Gamma )\cong  \Hom_B(B, Ae_\Gamma)e_-\cong  A e_\Gamma e_- =A \overline{e},\]
 since $e_-e_\Gamma =e_\Gamma e_- =\overline{e}$. For the other Hom, note that $e_- B=e_- e_\Gamma A e_\Gamma = \overline{e} A e_\Gamma$ and thus $\Hom_B( Ae_\Gamma, e_-B)=\Hom_B( Ae_\Gamma, \overline{e} A e_\Gamma)$.
For each $\overline{e}\beta \in \Hom_B(A e_\Gamma, Ae_\Gamma)$ one sees that the natural map $\Phi: \overline{e}\Hom_B(A e_\Gamma, Ae_\Gamma) \xto{} \Hom_B(A e_\Gamma, \overline{e}Ae_\Gamma)$ sending $\overline{e}\beta$ to $(a e_\Gamma \mapsto \overline{e}\beta(a e_\Gamma))$ is surjective and moreover injective. Thus $\Phi$ is an isomorphism.
It follows that 
\[
\Hom_B( Ae_\Gamma, e_-B)  \cong   \overline{e} \Hom_B( Ae_\Gamma, A e_\Gamma)  \cong   \overline{e}A  \]
as rings. In total we get 
\[
 \Hom_B(Ae_\Gamma ,  A e_\Gamma)/ \langle e_- B \rangle   \cong  A/\left( (A\overline{e})(\overline{e}A) \right)   \cong  A/A\overline{e}A \ . 
\]
\end{proof}

This theorem immediately yields that $A/A\overline{e}A$ is a noncommutative resolution of the discriminant $R/(\Delta)$.

\begin{remark}
By Example \ref{Ex:isotypical-triv} $R/(\Delta)$ is a direct summand of $S/(J)$. Using \cite[Thm.~5.3]{DoFI} it follows that the centre of $\Abar$ is equal to $Z(\End_{R/(\Delta)}(S/(J))=R/(\Delta)$.
\end{remark}

\begin{corollary} \label{Cor:NCRdisc}
Notation as in the theorem. If $G \not \cong \mu_2$, then $A/Ae_{\chi}A\cong \End_{R/(\Delta)}(S/(J))$ yields a NCR of $R/(\Delta)$ of global dimension $n$. If $G \cong \mu_2$, then $A/Ae_{\chi}A \cong R/(\Delta)$ is a NCCR of $R/(\Delta)$.
\end{corollary}

\begin{proof}
By the theorem $A/Ae_{\chi}A \cong \End_{R/(\Delta)}(S/(J))$.  By Cor.~\ref{Cor:quotientsiso} $A/Ae_{\chi}A \cong A/AeA$.  By Corollary \ref{Cor:gdimA} the global dimension of $A/AeA$ is $n$ if $G \not \cong \mu_2$. For the remaining case, cf.~Rmk.~\ref{Rmk:mu2} and note that $R/(\Delta)$ is regular.
\end{proof}

\begin{corollary}[McKay correspondence] The nontrivial irreducible graded $G$-representations are in $1-1$-correspondence to the graded indecomposable projective $\overline{A}$-modules, that are in $1-1$-correspondence to the isomorphism classes of graded $R/(\Delta)$-direct summands of $S/(J)$.  Moreover, we also obtain $1-1$ correspondences of these objects up to grading shifts.
\end{corollary}

\begin{proof} Take $\overline{A}=A/AeA$. Similar as in Lemma \ref{Lem:ProjCorresp} one has functors $\alpha, \beta$ between $\gr P(\overline{A})$ and $\gr \Mod(KG)$. This yields a bijection between the irreducible graded representations of $KG$ (except the trivial ones) and graded indecomposable projective $\overline{A}$ modules. On the other hand, we can uniquely decompose $S/(J)=\bigoplus_{i}M_i^{a_i}$ as a finite direct sum of CM-modules over $R/(\Delta)$ by Krull-Schmidt for the graded category as in~\cite[Cor.~of Lemma 3, Thm.~1]{Atiyah}. Then the indecomposable graded projective $\End_{R/(\Delta)}(S/(J))$-modules are of the form $\Hom_{R/(\Delta)}(S/(J),M_i)$, which yields the second bijection.
\end{proof}

\begin{remark} If one prefers, similar results can be established by passing to the Henselization $S'$ of $S$ at the origin where $\overline{A}\otimes S'$ is semi-perfect, applying results of~\cite{Vamos}, or by passing to the power series ring.\end{remark}

\begin{example}(The normal crossings divisor as discriminant and its skew group ring)
This example was our main motivation for investigating the relationship between $A/AeA$ and $\End_{R/(\Delta)}(S/(J))$: The reflection group
$G=(\mu_2)^n$ acts on $V=K^n$ via the reflections $\sigma_1, \ldots, \sigma_n$ with
\begin{align*}
\sigma_i(x_j)& =\begin{cases}  \phantom{-}x_j \textrm{ if } &i  \neq j \\ 
                                                         -x_j \textrm{ if } &i  =j \  .
                                                         \end{cases}
\intertext{So $G$ can be realized as the subgroup of $GL(V)$ generated by the diagonal matrices}
s_i & = \begin{pmatrix} 1 & 0 & 0 & 0 & 0 \\
0 & \ddots & 0 &0 &0 \\
0 & 0 & -1 &0 &0 \\
0 & 0 & 0 & \ddots & 0 \\
0 & 0 & 0 & 0 & 1
\end{pmatrix} \ .
\end{align*}
It is easy to see that the invariant ring $R=S^G=K[x_1^2, \ldots, x_n^2]=K[f_1, \ldots, f_n]$. 
Then the Jacobian determinant $J=z$ of the basic invariants $(f_1(x), \ldots, f_n(x))$ is $J=2^n x_1 \cdots x_n$. We may omit the constant factor $2^n$ for the remaining considerations. The hyperplane arrangement corresponding to $G$ is the normal crossing divisor $S/(J)=K[x_1, \ldots, x_n]/(x_1 \cdots x_n)$. 
The discriminant $\Delta$ 
is given by $\Delta=J^2=f_1 \cdots f_n$. So the coordinate ring of the discriminant is $R/(\Delta)=K[f_1, \ldots, f_n]/(f_1 \cdots f_n)$. \\ 
By Theorem \ref{Thm:main}, the ring $\overline{A}=A/AeA \cong \End_{R/(\Delta)}(S/(J))$ yields a NCR of $R/(\Delta)$. Here we can explicitly compute the decomposition of $S/(J)$ as $R/(\Delta)$-module: 
$$S/(J) \cong \bigoplus_{I \subsetneq [n]}x^{I}\cdot  \left(R/ (f^{[n] \backslash I}) \right),$$
where $[n]$ denotes the set $\{1, \ldots, n\}$ and $f^L=\prod_{l \in L}f_l$ for a subset $L \subseteq [n]$.
This holds because $ S \cong \bigoplus_{I \subsetneq [n]} R x^I$ as $R$-module and $\Ann_{R}(S/(J))=J^2=\Delta$. Thus $(Rx^I)/(f^{[n]}) \cong (R/( f^{[n] \backslash I})) \cdot x^I$ for any $I \subsetneq [n]$, and it follows that $S/(J)$ is a faithful $R/(\Delta)$-module. \\
In \cite[Thm.~5.5]{DFI} it was shown that the module $M=\bigoplus_{I  \subsetneq [n]}R/(\prod_{i \in I} f_i)$ gives a noncommutative resolution of global dimension $n$ of the normal crossing divisor $R/(\Delta)$. This was proven by showing that $\End_{R/(\Delta)}(M)$  is isomorphic to the order 
\begin{equation} \label{Eq:orderNC} \left( x^{J \backslash I} K[x_1, \ldots, x_n]\right)_{I,J} \subset K[x_1,\ldots,x_n]^{2^{n} \times 2^n}, \text{ where $I,J \subseteq [n]$.}
\end{equation}. 
On the other hand, one can also compute that the skew group ring $A$ is in this case is $A= (K[x_1, \ldots, x_n] * (\mu_2)^n ) \cong \bigotimes_{i=1}^n \Lambda_1$, where $\Lambda_1$ is the skew group ring $K[x] * \mu_2$.  Forming the quotient by $AeA$ yields the order \eqref{Eq:orderNC}. 

\end{example}

\subsubsection*{Coda: Results in dimension $2$} \label{Sub:dim2}

If $G \leqslant \GL(V)$, $\dim V=2$, is a  true reflection group, then the relation between $R=S^G$, $T=S^\Gamma$ and $R/(\Delta)$ can be interpreted in context of the classical McKay correspondence, cf.~Section \ref{Sec:ClassicalMcKay}: in this case  $T$ is isomorphic to $K[x,y,z]/(z^2 + \Delta(x,y))$, where $\{ z^2 + \Delta=0 \}$ is an Kleinian surface singularity. Moreover, $T$ is of finite CM-type, that is, there are only finitely many isomorphism classes of indecomposable CM-modules. By Herzog's Theorem \cite{Herzog78}, $S$ is a representation generator for $T$, that is, $\add_T(S)=\CM(T)$. \\

In the following we show that with Theorem \ref{Thm:main} we recover that $R/(\Delta)$ is an ADE-curve and furthermore we show that the hyperplane arrangement $S/(J)$ yields a natural representation generator for $R/(\Delta)$:
\begin{corollary} \label{Thm:dim2}
Let $G \leqslant \GL(V)$, $\dim V=2$, be a true reflection group, with invariant ring $S^G=R$ and discriminant $R/(\Delta)$. Then $R/(\Delta)$ is of finite CM-type and consequently $\Spec(R/(\Delta))$ is an ADE curve singularity. Moreover, $\add_{R/(\Delta)}(S/(J)) = \CM(R/(\Delta))$.
\end{corollary}

\begin{proof}
By Corollary~\ref{Cor:NCRdisc} we have that $\overline{A}= A/AeA \cong \End_{R/(\Delta)}(S/(J))$ has global dimension $2$. Moreover, by Example~\ref{Ex:isotypical-triv}, we see that $R/(\Delta)$ is a direct summand of $S/(J)$.
Since $R/(\Delta)$ is Gorenstein,
one can use the Auslander lemma,  cf. \cite{IyamaRejective,DFI} to see that $R/(\Delta)$ is of finite CM-type, and thus $\add(S/(J))=\CM(R/(\Delta))$. The only Gorenstein curves of finite CM-type are the ADE-curves, see \cite{GreuelKnoerrer}. 
\end{proof}

\section{Isotypical components and matrix factorizations} \label{Sec:Decomposition}

Let $G \leqslant \GL(V)$ be any finite pseudo-reflection group.  In this section we study direct sum decompositions of $S/(J)$ and $\overline{A}$. Moreover, the Hilbert--Poincar\'e series of the direct summands of $S/(J)$ as a $R/(\Delta)=S^G/(\Delta)$-modules are computed. Thus we also able to compute the ranks of these direct summands over $R/(\Delta)$ in case $\Delta$ is irreducible. In the case of $G=S_n$ we can even give a more explicit description using Young diagrams. We also compute the rank of $\overline{A}$ for any finite pseudo-reflection group  in two ways:  using the codimension $1$ structure and with Hilbert--Poincar\'e series (in case $\Delta$ is irreducible).

\subsection{Hilbert--Poincar\'e series of isotypical components of $S/(J)$}

Here we look at the Hilbert--Poincar\'e series of the direct summands $M_i$ of $S/(J)$: recall from Section \ref{Sec:Basics} that $M_i$ was defined to be the $R/(\Delta)$-module $\Hom_{KG}(V_i, S/(J))$, where $V_i$ is an irreducible $G$-representation. Further we have $S_i=\Hom_{KG}(V_i,S)$ and $S_i'=\Hom_{KG}(V_i',S)=\Hom_{KG}(V_i\otimes \det,S)$. From the exact sequence \eqref{Eq:isotypical} it follows that
 $$H_{M_i}(t) = H_{S_i}(t) - t^{m} H_{S'_i}(t) \ . $$
  Let $K_{S_i}(t)$ and $K_{S_i'}(t)$ be the numerator polynomials of the Hilbert-Poincar\'e series of $S_i$ and $S'_{i}$ respectively and $H_R(t)=\frac{1}{\prod_{i=1}^n (1-t^{d_i})}$ and $H_{R/(\Delta)}(t)=\frac{1-t^{m+m_1}}{\prod_{i=1}^n (1-t^{d_i})}$ the Hilbert--Poincar{\'e} series of $R$ and $R/(\Delta)$ respectively. Then $H_{M_i}(t)$ can be written as
\begin{equation} \label{Eq:HSforV} H_{M_i}(t)  =H_R(t)\left( K_{S_i}(t)- t^m K_{S'_i}(t) \right) 
  = H_{R/(\Delta)}(t)\frac{\left( K_{S_i}(t)- t^m K_{S'_i}(t) \right)}{1 - t^{m+m_1}} \ .
 \end{equation}
 
 \begin{remark}
 The numerator polynomials $K_{S_i}$ of the $H_{M_i}$ are called \emph{fake degree polynomials}, see e.g.~\cite{Carter}, or \emph{generalized Kostka polynomials}, see \cite{GarsiaProcesi}. 
 \end{remark}

\begin{example} \label{Ex:HSforSn}
In the case of $G=S_n$, the irreducible representations of $G$ correspond to partitions $\lambda$ of $n$ and each partition $\lambda$ is given by a Young diagram, see e.g.~\cite{FultonHarris}. Then the corresponding Hilbert--Poincar{\'e} series for the $\lambda$-isotypical component $S_\lambda$ of $S$ is given as
\begin{equation} \label{Eq:Kirillov1}H_{S_\lambda}(t)=\prod_{k=1}^n\frac{t^{f_k}}{1-t^{h_k}} \ ,
\end{equation}
where $f_k$ denotes the length of the leg of the hook of the $k$-cell and $h_k$ denotes the length of the hook of the $k$-cell (see  \cite[Thm.~1]{Kirillov}) [Note here: for $f_k$ the $k$-cell itself is not counted and for the hooklength it is counted once, cf.~\cite{FultonHarris}]. 
\end{example}

\subsection{Ranks of the isotypical components of $S/(J)$}  The ranks of the $M_i=\Hom_{KG}(V_i, S/(J))$ over $R/(\Delta)$ can be computed by evaluating $H_{M_i}(t)$ in $t=1$, at least when $\Delta$ is irreducible:

\begin{lemma} \label{Lem:rank}
Let $R=K[x_1, \ldots, x_n]$ be graded by $\deg x_i=d_i \in \NN$, let $\Delta \in R$ be a quasi-homogeneous polynomial, $R/(\Delta)$ be a domain, and let $M$ be a finitely generated CM module over $R/(\Delta)$.  Then 
$$\rank_{R/(\Delta)}(M)=\lim_{t \rightarrow 1}\frac{H_M(t)}{H_{R/(\Delta)}(t)}.$$
\end{lemma}

\begin{proof} Let $S'=K[y_1, \ldots, y_{n-1}]$, where the $y_i$ form a system of parameters of $R/(\Delta)$, then $M$ is a finitely generated module over $S'$ 
Note that the Hilbert--Poincar{\'e} series of $M$ (and of $R/(\Delta)$) does not change if we consider both modules over $S'$. If $M$ is a graded CM-module over the graded CM ring $R/(\Delta)$, then $\rank(M)=\frac{e_{S'}(M)}{e_{S'}(R/(\Delta))}$, see \cite[Theorem 18]{Northcott68} (cf.~also Thm.~4.7.9 in \cite{BrunsHerzog93}).  
Here $e_{S'}(-)$ denotes the multiplicity of a module over $S'$.  By work of William Smoke \cite{Smoke}, one can interpret $e_{S'}(M)$ as $\lim_{t \rightarrow 1}(\chi_{S'}(K) H_M(t))$, where $\chi_{S'}(M)$ is the so-called generalized multiplicity of $M$, also cf.~\cite{Stanley-Hilbert}, and $\chi_{S'}(K)$ is equal to $\prod_{i=1}^{n-1}(1-t^{d_i'})$, where $d_i'=\deg y_i$. Since both $M$ and $R/(\Delta)$ have ranks, both limits $\lim_{t \rightarrow 1}(\chi_{S'}(K) H_M(t))$ and $\lim_{t \rightarrow 1}(\chi_{S'}(K) H_{R/(\Delta)}(t))$ exist and thus 
\[ \rank_{R/(\Delta)}(M)=\frac{\lim_{t \rightarrow 1}(\chi_{S'}(K) H_M(t))}{\lim_{t \rightarrow 1}(\chi_{S'}(K) H_{R/(\Delta)}(t))}=\lim_{t \rightarrow 1}\frac{H_M(t)}{H_{R/(\Delta)}(t)} \ . \]
\end{proof}

\begin{proposition} \label{RankisotypicalDiscriminant} With notation as above, let $V_i$ be an irreducible representation of $G$. Then the rank over $R/(\Delta)$ of the $V_i$-isotypical component of $S/(J)$, $M_i$, is given by
$$\rank_{R/(\Delta)}M_i=\frac{1}{m+m_1}\left( m\dim V_i' + \frac{d K_{S_i'}}{d t}(1)-\frac{d K_{S_i}}{d t}(1) \right),$$
where $V'_i$ stands again for the twisted representation $V_i \otimes \det$. If $G$ is a true reflection group, this simplifies to
$$\rank_{R/(\Delta)}M_{i}=\frac{1}{2}\left( \dim V_i' + \frac{\frac{d K_{S_i'}}{d t}(1)-\frac{d K_{S_i}}{d t}(1)}{m} \right).$$
\end{proposition}

\begin{proof}
Using expression \eqref{Eq:HSforV} and Lemma \ref{Lem:rank} for $H_{M_i}(t)$ we get
$$\rank_{R/(\Delta)}M_i=\lim_{t \rightarrow 1}\left(\frac{ H_{R/(\Delta)}(t) \frac{K_{S_i}(t) - t^m K_{S'_i}(t)}{1-t^{m+m_1}}}{H_{R/(\Delta)}(t)} \right)=\lim_{t \rightarrow 1}\left(\frac{K_{S_i}(t) - t^m K_{S'_i}(t)}{1-t^{m+m_1}} \right).$$
By the rule of l'Hospital this limit is equal to 
$$\lim_{t \rightarrow 1} \left(\frac{\frac{d K_{S_i}}{d t}(t) -m t^{m -1} K_{S'_i}(t) - t^m \frac{d K_{S'_i}}{d t}(t)}{-(m+m_1)t^{m+m_1-1}} \right).$$
Evaluating this expression in $t=1$ yields 
the above expression.
If $G$ is generated by order $2$ reflections, then $m=m_1$ and also $\det=\det^{-1}$, so one obtains the second formula.
\end{proof}

\begin{proposition} In case of $G=S_n$ and an irreducible representation $\lambda$ the rank of the $\lambda$-isotypical component $M_{\lambda}$ of $S/(J)$ is given by
\begin{equation} \label{Eq:Kirillov}\rank_{R/(\Delta)}(M_\lambda) = \dim (V_\lambda) \left(\frac{1}{2} +\frac{A-F}{2 m}\right) \ , 
\end{equation}
where $F=\sum_{k}f_k$ is the total footlength and $A=\sum_k a_k$ is the total armlength of the Young diagram corresponding to $\lambda$.
\end{proposition}

\begin{proof}
The rank of $M_\lambda$ over $R/(\Delta)$ is given as
\begin{align*}
\rank_{R/(\Delta)}M_\lambda & = \lim_{t \rightarrow 1} \frac{H_{M_\lambda}(t)}{H_{R/(\Delta)}(t)} =
\lim_{t \rightarrow 1} \frac{H_{S_\lambda}(t) - t^m H_{S_{\lambda'}}(t)}{H_{R/(\Delta)}(t)} \ ,
\end{align*}
where $\lambda'$ is the conjugate partition to $\lambda$. Note that the hooklengths of the conjugate partitions are the same, that is, the hooklength $h_k$ of the $k$-cell in $\lambda$ is the same as the hooklength $h'_k$ of the corresponding $k$-cell in $\lambda'$. On the other hand, one has that the footlength $f_k$ in $\lambda$ is equal to the armlength $a'_k$ in $\lambda'$ and vice versa. Moreover, these are connected to the hooklength via $h_k=f_k + a_k +1$. Now substitute Kirillov's formula \eqref{Eq:Kirillov1} in the above equation:
\begin{align*}
\rank_{R/(\Delta)}M_\lambda & = \lim_{t \rightarrow 1} \left( \frac{\frac{1}{\prod_{k=1}^n (1-t^{h_k})} (t^F - t^{m+A})}{\frac{1-t^{2m}}{\prod_{k=1}^n (1-t^{d_k})}} \right) =  \lim_{t \rightarrow 1} \left[ \left( \frac{t^F-t^{m+A}}{1-t^{2m}} \right) \cdot \left( \frac{\prod_{k=1}^n (1-t^{d_k})}{\prod_{k=1}^n (1-t^{h_k})} \right) \right] \ ,
\end{align*}
where $d_k$ are the degrees of the basic invariants of $S_n$. Now using l'Hospital's rule for the two factors in the product yields: 
\begin{equation*}
\rank_{R/(\Delta)}M_\lambda = \frac{m+A-F}{2m} \cdot \prod_{k=1}^n \frac{d_k}{h_k} \ .
\end{equation*}
Since $d_k=k$ for all $k=1, \ldots, n$, the product $\prod_{k=1}^n \frac{d_k}{h_k} = \frac{n!}{\prod_{k=1}^n h_k} = \dim(V_\lambda)$ by the hooklength formula, see e.g.~\cite{FultonHarris}. This yields the formula in \eqref{Eq:Kirillov}. 
\end{proof}

\subsection{Identifying isotypical components}  

The main result of this section is to identify the module of logarithmic vector fields $\Theta_{R}(-\log \Delta)\cong \Theta_S^G$ and its exterior powers $\Theta^m_R(-\log \Delta)=\bigwedge^m(\Theta_{R}(-\log \Delta))\cong(\bigwedge^m \Theta_S)^G$ as isotypical components of the natural representation $V$ and its exterior powers $\bigwedge^m V$ and their corresponding matrix factorizations. The modules of logarithmic differential forms and logarithmic residues were 
first defined and studied by Kyoji Saito in \cite{Saito80}.    \\

We start with recalling some facts from linear algebra, and introducing the notation for logarithmic vector fields, where 
we follow \cite{OTe}.  \\

\label{sit:linalg} 
Recall the following result from linear algebra: Let $\vp:P\to Q$ be a linear map between
finite projective modules of same rank $m$ over some commutative ring $C$.
With $\Lambda^{i}$ the $i^{th}$ exterior power over $C$ and 
$|P|=\det P = \Lambda^{m}P, |Q|=\det Q= \Lambda^{m}Q$ the invertible
$C$--modules given by the top exterior powers of $P$ and $Q$, respectively, one has
isomorphisms of $C$--modules $\Lambda^i P\cong |P|\otimes_{C}\Lambda^{m-i}P^{*}$ and
$\Lambda^i Q\cong|Q|\otimes_{C}\Lambda^{m-i}Q^{*}$ induced from the nondegenerate pairing
$\Lambda^{i}\otimes_{C}\Lambda^{m-i}\to \Lambda^{m}$ . Consider the composition
\[
\vp^{\adj}\vp\colon \Lambda^i P\xto{\ \Lambda^i \vp\ } \Lambda^i Q \cong  
|Q|\otimes_{C}\Lambda^{m-i}Q^{*}\xto{\ |Q|\otimes_{C} \Lambda^{m-i}\vp^{*}\ }
|Q|\otimes_{C}\Lambda^{m-i}P^{*}\cong
|Q|/ |P|\otimes_{C}\Lambda^iP\,,
\]
where the adjugate morphism $\vp^{\adj}$ is the composition of the maps to the right of $\vp$,
while $|Q|/|P|$ is shorthand for the invertible $C$--module $|Q|\otimes_{C}|P|^{-1}$.
The top exterior power of $\vp$ defines the $C$--linear map $\Lambda^{m}\vp:|P|\to |Q|$ 
and the associated $C$--linear section 
$\det \vp = \Lambda^{m}\vp\otimes_{C} |P|^{-1}:C\to |Q|/|P|$ of the invertible line bundle $|Q|/|P|$
is the determinant of $\vp$. The Laplace expansion of the determinant  
then translates into
\begin{align*}
\vp^{\adj}\vp = (\det\vp)\id_{P}\colon \wedge^iP\xto{\quad} (|Q|/|P|)\otimes_{C}\wedge^iP\,.
\end{align*}

 We maintain our usual set-up: $G\leqslant \GL(V)$ is a finite group generated by pseudo-reflections as
subgroup of $\GL(V)$, and $S=\Sym_{K}(V)$ denotes the polynomial ring defined by $V$ over $K$,
with $R=S^{G}$ the invariant subring. Recall that $R\cong \Sym_{K}W$ is a polynomial ring
in its own right, with $W\cong R_{+}/R_{+}^{2}$ the graded $K$--vector space generated 
by the classes of the basic invariants $f_{i}\in R_{+}$.

We denote by $\Omega^{1}_S$ the K\"ahler differential forms on $S$ over $K$ and by 
$\Theta_S=\Hom_S(\Omega^{1}_S,S)$ its $S$--dual, isomorphic to the $S$--module
of $K$-linear derivations, or vector fields, on $S=\Sym_K(V)$. We define $\Omega^1_R$ and $\Theta_R$ similarly by replacing $V$ with $W$.

Restricting a derivation on $S$ to $V=\Sym^{1}_{K}V\subset S$ yields canonical isomorphisms
\[
\Theta_{S} = \Hom_{S}(\Omega^{1}_{S},S)\xto{\ \cong\ } \Hom_{K}(V, S)
\cong S\otimes_{K}V^*\,,\quad
D\mapsto D|_{V}\,.
\]
Similarly,
$$\Theta^i_S = \Hom_{K}(\Lambda^iV, S) = S\otimes_K \Lambda^iV^*$$
$$\Omega^i_S = \Hom_{K}(\Lambda^iV^*, S) = S\otimes_K \Lambda^iV.$$
and these hold if we replace $S$ with $R$ and $V$ with $W$.

If a group $G$ acts on $S$ through $K$--algebra automorphisms then $G$ also acts naturally on 
$\Omega^{1}_S$ and $\Theta_S$, respectively.

Let again denote $R=S^G$, then $\Theta_S^G$ is the $R$-module of $G$-invariant 
derivations and $(\Omega^{1}_S)^G$ the $R$-module of $G$-invariant differential forms.
Employing the isomorphisms above, it follows that 
\[
\Theta_S^G\xto{\ \cong\ } \Hom_{K}(V,S)^{G}\cong \Hom_{KG}(V,S)\, ,
\]
or, in other words, that the $V$-isotypical component of $S$ is $\Theta_S^G\otimes V$.
\begin{lemma} \label{lem:Visotyp}
If the defining representation $V$ is an  irreducible $G$--representation then the evaluation map
\begin{align*}
\ev:\Theta_S^G\otimes_{K}V\to S
\end{align*}
identifies $\Theta_S^G$ with the isotypical component of $S$ that belongs to $V$.
In particular, the evaluation map is a split $R$--monomorphism.\qed
\end{lemma}

  We have the Jacobian map of $S$-modules $\Omega^{1}_{R} \otimes_{R} S \xto{\jac} \Omega^{1}_{S}$ defined by the inclusion of $K$--algebras $R\ \into\  S$. This gives the Zariski--Jacobi sequence $0 \xto{} \Omega^{1}_{R} \otimes_{R} S \xto{\jac} \Omega^{1}_{S} \xto{} \Omega^1_{S/R} \xto{} 0$, 
see e.g.~ \cite[Thm~25.1]{Matsumura86}.
Note that $\jac$ is injective because $ \Omega^{1}_{R} \otimes_{R} S$ is a free 
$S$--module as  $R$ is smooth over $K$, while the potential kernel is supported on the critical 
locus of the morphism $\Spec S\to \Spec R$, thus must be zero as the morphism is generically 
smooth.

Applying $(\ )^{*}=\Hom_S(-,S)$ yields the map $\jac^*:\Theta_{R}\otimes S \rightarrow \Theta_S.$
\[ 
\tag{$\dagger$}
\label{dualZJ}
0 \xleftarrow{} T^{1}_{S/R} \xleftarrow{} \Theta_{R}\otimes S \xleftarrow{ {\jac}^* } 
\Theta_{S} \xleftarrow{} 0\,,
\]
where $T^{1}_{S/R}\cong \Ext^{1}_{S}(\Omega^{1}_{S/R}, S)$ is the first tangent cohomology of 
$S$ over $R$. 
 In particular, the determinant of the (transposed) Jacobian matrix is given by the $S$--linear 
co-section
\begin{align*}
\det(\jac^{*}): S\lto \Theta^n_{R}\otimes_{S}(\Theta^{n}_{S})^{*}
\cong S\otimes |V|/|W| \ ,
\end{align*}
where $|V|/|W|$  is again shorthand for $\det V \otimes_K (\det W)^{-1}$.

Taking $G$--invariants is exact and applied to the short exact sequence (\ref{dualZJ}) above
it returns
\[ 
0 \xleftarrow{} j_\Delta \xleftarrow{} \Theta_{R} \xleftarrow{\mathrm{(jac}^*)^G} 
\Theta^{G}_{S} \xleftarrow{} 0\,.
\]
In \cite[Cor.~6.57]{OTe} it is shown that the $R$--linear inclusion $\mathrm{(jac}^*)^G$ 
identifies $\Theta_{S}^{G}$ with 
$\Theta_{R}(-\log \Delta)=\{ \theta \in \Theta_{R}: \theta(\Delta) \in \Delta R \}$, 
the $R$-module of logarithmic vector fields along the discriminant $\Delta$.
We have the natural inclusions 
\begin{align*} \mu^*:  \Theta_R(-\log \Delta)\cong \Theta_S^G & \xto{(\jac^*)^G} \Theta_R \\
\zeta^*:  \Theta_R(-\log \Delta) \otimes S  \cong \Theta_S^G\otimes S & \xto{\phantom{(\jac^*)^G}} \Theta_S \ .
\end{align*}
Accordingly, $j_{\Delta}=\Coker \left(\mu^*\right)$ can be identified with the 
Jacobian ideal of the discriminant, 
\[
j_{\Delta}\cong \left\{D(\Delta)+(\Delta)\mid D\in \Theta_{R} \right\}\subseteq R/(\Delta)\,,
\]
and the determinant of $\mu^{*}$ is the discriminant of $G$.
It yields the $R$--linear co-section
\begin{align*}
\det(\mu^*): R\lto \Theta^n_R(-\log \Delta) \otimes_{R} (\Theta^{n}_{R})^{*}
\cong R\otimes |W|/|W'|\,,
\end{align*}
where $W'$ is a graded $K$--vector space so that $\Theta_R(-\log \Delta)\cong R\otimes (W')^{*}$.
In particular, $R \otimes |W'|$ is a free $R$--module of rank $1$ generated in degree $-c$, where 
$c=\sum_{i=1}^{n}c_{i}$ is the sum of the \emph{co--degrees\/} $0=c_{1}\leqslant \cdots
\leqslant c_{n}$, so that $\Theta_R(-\log \Delta) \cong \oplus_{i=1}^{n}R(-c_{i})$.

As $\Theta_{R}\cong \oplus_{i=1}^{n}R(d_{i})$, the degree of the discriminant is 
$|\Delta| = \sum_{i=1}^{n}(d_{i} + c_{i})$.

\begin{example}
If $G$ is a \emph{duality group\/}, then $d_{i}-c_{i}=d_{1}$, while $d_{i}+c_{n-i+1}= d_{n}$, the
\emph{Coxeter number\/} of $G$. Thus for such a group, $|\Delta| = h\cdot n =
\sum_{i=1}^{n}(2d_{i}-d_{1})$.
Coxeter and Shephard groups are duality groups, and for Coxeter groups $d_{1}=2$ 
so that for these groups $|\Delta| = 2\sum_{i=1}^{n}(d_{i}-1)$, twice the number of reflections in that 
group, as it should be.
\end{example}

  By \cite[Thm.~6.59]{OTe} the map $\Theta_R(-\log \Delta) \otimes_R S \xto{} \Theta_S$ is an inclusion as well and identifies in this way 
the $S$--modules $\Theta_R(-\log \Delta)\otimes_RS \cong \Theta_S(-\log z)$, where 
$\Theta_S(-\log z)\subseteq \Theta_{S}$ is the $S$--module of logarithmic vector fields 
along the hyperplane arrangement given by $\{z=0\}\subseteq \Spec S$.

Using the same analysis as before, it follows that $z$ has degree $|z| = \sum_{i=1}^{n}(c_{i}+1)$,
equal, by definition, to the sum  of the \emph{co-exponents\/} of $G$, equal as well to the number of
mirrors or reflecting hyperplanes defined by $G$ (this number has been denoted earlier as $m_1$).

We note the following facts.
\begin{proposition}
  \begin{enumerate}[(a)]
  \item   If the defining representation $V$ of the pseudo-reflection group $G\leqslant \GL(V)$ is irreducible, then $\Lambda^i V$ are irreducible for $1 \leq i \leq \dim(V)$. 
    \item $\Omega_R^{i} \cong (\Omega_S^{i})^G \cong (S\otimes \Lambda^iV)^G$ 
      \item $\Theta_R^i(-\log \Delta) \cong (\Theta^i_S)^G \cong (S\otimes \Lambda^iV^*)^G$
  \end{enumerate}
  \end{proposition}
 \begin{proof}
   The first statement is well known, e.g., \cite[Thm.~4.6]{GeckMalle}.  The second is \cite[Theorem.~6.49]{OTe}, and the third statement follows immediately from \cite[Prop.~6.70]{OTe}, which are both special cases of Solomon's theorem \cite[Prop.~6.47]{OTe}.
\end{proof}

 Now we come to main goal of this section to identify some of the $R$--direct summands of $S/(J)$.  By the above, we have the following pair of dual commutative diagrams: 
  \[ \xymatrixcolsep{3pc}
  \xymatrix{  
 \Theta_S \ar@{=}[r] & \Theta_S \ar[d]^{\mathrm{jac}^*} & & \Omega_S \ar[d]_{\zeta} \ar@{=}[r] & \Omega_S \\
 \Theta_R(-\log \Delta)\otimes_R S \ar[r]^-{\mu^*\otimes S} \ar[u]^{\zeta^*}& \Theta_R \otimes_R S & & \Omega_R(\log \Delta)\otimes_R S & \Omega_R\otimes_R S \ar[l]_-{\mu \otimes S} \ar[u]_{\mathrm{jac}} \\
 \Theta_R(-\log \Delta) \ar[r]^-{\mu^*} \ar[u] & \Theta_R \ar[u] & & \Omega_R(\log \Delta) \ar[u] & \Omega_R \ar[l]_-{\mu} \ar[u]
 }
 \]
Here the top squares are commutative diagrams of $S$-modules and the bottom squares are commutative diagrams of $R$-modules. 
 Let $\iota: \Lambda^i V \rightarrow \Omega^i_S$ and $\iota^*: \Lambda^{n-i}V^* \otimes \Theta^{n-i}_S$ be the natural inclusions.
 The above maps give us the following pair of commutative diagrams where the vertical maps of the top two
   squares are the multiplication in the exterior algebra,  $a\otimes b \mapsto a\wedge b.$ 
    \begin{equation} \label{thetafig}
 \xymatrixcolsep{8pc} \xymatrix{
 \Theta_S^n \ar[r]^{\Lambda^n\jac^*} & \Theta_R^n \otimes_R S \\
 \Theta^i_S \otimes \Theta_S^{n-i} \ar[u]^{\wedge} \ar[r]^{\Lambda^i\jac^*\otimes \Lambda^{n-i}\jac^*} & \Theta^i_R \otimes \Theta^{n-i}_R\otimes_R S \ar[u]_{\wedge \otimes S} \\
 \Theta^i_R(-\log \Delta) \otimes_R \Lambda^{n-i}V^* \ar[r]^{\Lambda^i\mu^*\otimes \Lambda^{n-i}V^*} \ar[u]^{\Lambda^i\zeta^* \otimes \iota^*} &  \Theta^i_R \otimes_R \Lambda^{n-i}V^* \ar[u]_{\Theta^i_R\otimes (\Lambda^{n-i}\jac^*\circ \iota^*)}
 }
 \end{equation}

\begin{equation} \label{omegafig}
 \xymatrixcolsep{9pc} \xymatrix{
  \Omega^n_S  \ar[r]^{\Lambda^n \zeta}  & \Omega^n_R(\log \Delta)\otimes_R S \\
 \Omega_S^{n-i} \otimes \Omega_S^i  \ar[u]^{\wedge} \ar[r]^{\Lambda^{n-i}\zeta \otimes \Lambda^i \zeta} & \Omega^{n-i}_R(\log \Delta) \otimes \Omega_R^i(\log \Delta) \otimes_R S \ar[u]_{\wedge \otimes S} \\
 \Omega^{n-i}_R\otimes \Lambda^iV    \ar[r]^{\Lambda^{n-i}\mu \otimes \Lambda^iV} \ar[u]^{\Lambda^{n-i} \jac \otimes \iota} &  \Omega^{n-i}_R(\log \Delta) \otimes \Lambda^iV \ar[u]_{\Omega_R^{n-i}(\log\Delta) \otimes (\Lambda^i\zeta \circ \iota)}
 }
\end{equation}
  The top squares of these diagrams commute since if $\phi:P\rightarrow Q$
  is a map of free $S$-modules, then $\Lambda^\bullet \phi : \Lambda^\bullet P \rightarrow \Lambda^\bullet Q$ is an homomorphism of $S$-algebras.

 We know that $\Lambda^n \Omega_R^1(\log \Delta) \cong \Omega^n_R(\log \Delta)
 \cong \Omega_R^n(|\Delta|)$, since $\Delta$ is a free divisor, and $zJ = \Delta$ so we obtain the following maps
 $$\Omega_R^n \otimes_R S \xto{\hspace{0.5cm}J\hspace{0.5cm}} \Omega^n_S \xto{\hspace{0.5cm}z\hspace{0.5cm}} \Omega^n_R(|\Delta|) \otimes_R S.$$

 Now we apply the functor $\Hom_{KG}(\Lambda^i V,-)\otimes_K\Lambda^iV$
 to this sequence.  We first simplify the terms 
 \begin{align*}
 \Hom_{KG}(\Lambda^i V,\Omega^n_S) \cong (\Omega^n_S\otimes \Lambda^i V^*)^G  
& \cong (\Omega^{n-i}_S)^G \cong \Omega^{n-i}_R \cong \Omega^n_R \otimes \Theta^i_R \ ,\\
   \Hom_{KG}(\Lambda^i V,\Omega^n_R\otimes S) & \cong   (\Omega^n_R \otimes S \otimes \Lambda^iV^*)^G \\
& \cong   \Omega^n_R\otimes (\Theta^i_S)^G\\
&  \cong  \Omega^n_R\otimes \Theta_R^i(-\log \Delta)\\
&  \cong  \Omega^n_R(\log \Delta)(-|\Delta|) \otimes \Theta_R^i(-\log \Delta)\\
   & \cong  \Omega_R^{n-i}(\log \Delta) \ . \end{align*}
Here we have used the fact that $\Omega^n(\log \Delta) \cong \Omega^n(|\Delta|)$.
Now applying the functor with its natural transformation to the identity functor yields the following commutative diagram:
\[  \xymatrixcolsep{5pc} \xymatrix{ 
\Omega_R^n \otimes S  \ar[r]^-{\Lambda^n\jac^*\otimes \Omega^n_R\otimes \Lambda^nV} & \Omega^n_S \ar[r]^{\Lambda^n\zeta} & \Omega^n_R(|\Delta|) \otimes S \\
\Omega_R^{n-i}(\log \Delta)(-|\Delta|)\otimes \Lambda^i V \ar[r] \ar[u] & \Omega^{n-i}_R \otimes \Lambda^i V \ar[u] \ar[r]^{\Lambda^{n-i}\mu\otimes \Lambda^iV} & \Omega_R^{n-i}(\log \Delta) \otimes \Lambda^i V \ar[u]  \\
 \Omega^n_R\otimes \Theta_R^i(-\log \Delta) \otimes \Lambda^{i} V \ar[r]^-{\Omega^n_R \otimes \Lambda^i \mu^*\otimes \Lambda^{i}V} \ar[u]^{\sim}  & \Omega^n_R \otimes \Theta^i_R  \otimes  \Lambda^i V \ar[r] \ar[u]^{\sim} &  \Omega^n_R\otimes \Theta_R^i(-\log \Delta)(|\Delta|) \otimes \Lambda^{i} V \ar[u]^{\sim}
} 
\]
where we have presented two isomorphic interpretations of the bottom row.  It is clear that this diagram commutes since the left square with the bottom row is the outer square of the diagram \eqref{thetafig} tensored with $\Omega_R \otimes \Lambda^nV$ after applying the isomorphisms $\Lambda^{n-i}V^*\otimes V^n \cong \Lambda^iV$ and $\Theta^i_R\otimes \Omega^n_R \cong \Omega^{n-i}_R$, and the upper right square is the diagram \eqref{omegafig}. Note that the cokernel of $\Lambda^{n-i}\mu: \Omega_R^{n-i} \rightarrow \Omega_R^{n-i}(\log \Delta)$ is the \emph{$(n-i)$-th logarithmic residue}, see \cite{Saito80}. We call the cokernel $\Lambda^i\mu^*:\Theta^i_R(-\log \Delta)\rightarrow \Theta^{i}_R$ the \emph{$i$-th logarithmic co-residue of $\Delta$.}
It is clear that the vertical maps are the evaluations of the natural transformation. 
  Lastly, since the maps on the bottom row are uniquely determined by commuting with the diagram, we obtain the following result.
\begin{theorem} \label{logRes}
  For all $i$ with $1 \leq i \leq \dim V$,  there is a matrix factorization of $\Delta$ given by the pair of maps $\Lambda^{n-i}\mu$ and $\Lambda^i\mu^*\otimes \Omega^n_R$.  The cokernels of these maps are the logarithmic residues and co-residues with a degree shift $|\Omega^n_R|$, 
  which occur as $R/(\Delta)$-direct summands of $S/(z)$ and $S/(J)$ respectively, with multiplicity $\binom{n}{i} = \dim \Lambda^iV$. In particular, the first logarithmic residue $\coker(\mu)$ and $\coker(\mu^*) = j_\Delta$ are summands of $S/(z)$ and $S/(J)$ respectively, of multiplicity $n$. 
  \end{theorem}

\begin{example}
  Let $G = G(r,1,n) \cong \mu_r \wr S_n$ be the full monomial group acting in the usual way on $S=K[x_1,\ldots,x_n]$.  Let  $p_i = \frac{1}{ir}\sum_j x_j^{ri}$ be the $i^{th}$ power sum function of the $x_i^m,$ for $i\geq 1$.
  One choice of generators for the invariants is $p_1,\ldots,p_n$.  So $R=K[p_1,\ldots,p_n]$ as in~\cite[Section 6]{STo}.
  It is now easy to compute that $\jac = ( x_i^{jr-1})_{ij}$ in terms of the bases $dp_i$ and $dx_j$ of
  $\Omega_R$ and $\Omega_S$ respectively.  A basis of $\Theta_R(-\log \Delta) \cong \Theta_S^G$
  is given by $\theta_i = \sum_j x^{(i-1)r+1}_j \frac{\partial}{\partial x_j}$ as seen in~\cite[Appendix B.1]{OTe},
  and so $\zeta^* = ( x^{(i-1)r+1}_j )_{ij}$ in terms of the bases $\theta_i$ and $\frac{\partial}{\partial x_j}$
  of $\Theta_R(-\log \Delta)$ and $\Theta_S$ respectively.
  Now we can compute $\mu = \zeta \jac = r((i+j-1)p_{i+j-1})_{ij}$ in terms of the bases $dp_i$ and $\theta_i$.
  From this it is easy to compute
  $$J = \det(\jac) = (x_1\cdots x_n)^{r-1} \prod_{i<j} (x_i^r-x_j^r)$$
  $$z =\det(\zeta) =  x_1\cdots x_n\prod_{i<j} (x_i^r-x_j^r).$$
Lastly, the maps $\Lambda^i\mu^*$ and $\Lambda^{n-i}\mu$ will determine a matrix factorization for $\Delta$ for each $i$.
\end{example}

\section{Extended example: $S_4$ and the swallowtail}

Consider the case of $G=S_4$ acting on $K^3$. We will give an explicit description of the direct summands of $S/(J)$ over the discriminant. \\
For this example, $S=k[x,y,z]$. Let $s=-x-y-z$ and $\sigma_i(s,x,y,z)$ be the elementary symmetric function. Then $R=k[u,v,w]$
where $u = 6\sigma_2$, $v=4\sigma_3$ and $w=3\sigma_4$ and  $$J=(x-y)(x-z)(y-z)(2x+y+z)(2y+x+z)(2z+x+y).$$ 
A generator of the discriminant ideal $(\Delta)$ can be computed as the determinant of the matrix $(Jac)^T(Jac)$, where $Jac=\left(\frac{\partial f_i}{\partial x_j} \right)$ is the Jacobian matrix, cf.~\cite{SaitoReflexion,OTe}. An explicit equation is: 
\[ \Delta = -  v^4 - 2 u^3 v^2  + 9 u^4 w + 6 u v^2 w - 6 u^2 w^2 + w^3. \]
$\Spec(R/(\Delta))$ is called the {\it swallowtail}. Its singular locus consist of two curves: a parabola (the ``self-intersection locus'') and a cusp, meeting at the origin, see Fig.~\ref{Fig:swallowtail}. \\

Now let us sketch the computation of the matrix for multiplication by $J$.
Consider the map induced by multiplication by $J$ on $S$
$$S \xto{J} S$$
We know that $S$ is a free $R$-module and that $J^2=\Delta \in R$, so
$$S \xto{J} S \xto{J} S $$
is a matrix factorization of $\Delta$ over $R$ by definition.
We wish to decompose $S/(J)$ into indecomposable CM-modules over $R/(\Delta)$.
We can use the grading and the $G$-action to provide information about the decomposition.

First recall that $R=K[f_1,f_2,f_3]$ and let $(R_+)$ be the ideal in $S$ generated by $f_1,\ldots,f_n$.  Recall that (see Section \ref{Sub:isotypical})
$S/(R_+) \cong KG$
as $G$-representations, and
$S/(R_+) \otimes_K R \cong S$ as graded $RG$-modules.

In this example $G=S_4$.  Let us call the irreducible representations
$$K,V,W,V',\det$$
corresponding to the partitions 
$$\ytableausetup{centertableaux,boxsize=0.3em} 4=\ydiagram{4} \ , \ 3+1=\ydiagram{3,1}\ , \ 2+2=\ydiagram{2,2} \ , \ 2+1+1=\ydiagram{2,1,1} \ ,\ 1+1+1+1=\ydiagram{1,1,1,1} \ .$$
  We have that
$$S/(R_+) \simeq K(0) \oplus V(-1) \oplus V(-2) \oplus W(-2) \oplus V(-3) \oplus V'(-3)
\oplus V'(-4) \oplus W(-4) \oplus V'(-5) \oplus \det(-6) \ , $$
where the number in $(-)$ indicates the degree shift.

By Section \ref{eqn:Sdecomp}, $S$ decomposes into isotypical components via the isomorphism
$$S \cong \bigoplus_{V_i \text{ irreps of } G}\Hom_{KG}(V_i,S) \otimes_K V_i$$
which gives us that the map $ S \xto{J} S $
decomposes into components of the form
$$\Hom_{KG}(U,S)  \otimes_K U  \xto{J} \Hom_{KG}(U\otimes \det,S)  \otimes_K U\otimes \det$$
for each irreducible representation $U$ of $G$.

So for our example $S_4$ we have the following components 
\begin{align}
\label{rep1} K(0) \otimes R   & \rightarrow \det(-6) \otimes R \\
\label{rep2} \det(-6) \otimes R & \rightarrow K(0) \otimes R \\
\label{rep3}(V(-1)\oplus V(-2) \oplus V(-3))\otimes R & \rightarrow (V'(-3) \oplus V'(-4) \oplus V'(-5))\otimes R \\
\label{rep4} (V'(-3) \oplus V'(-4) \oplus V'(-5))\otimes R & \rightarrow  (V(-1)\oplus V(-2) \oplus V(-3))\otimes R \\
\label{rep5} (W(-2) \oplus W(-4))\otimes R & \rightarrow (W(-2) \oplus W(-4))\otimes R 
\end{align}
where the maps are the $R$-linear maps given by multiplication by $J$ restricted to each component.
Combining the first two components of lines \eqref{rep1} and \eqref{rep2} we obtain the matrix factorization
$$ R \xto{ } JR \xto{} R$$ 
where both maps are multiplication by $J$.  The cokernels of the two maps are $0$ and $R/\Delta$ respectively.
By choosing bases of  $V(-1)\oplus V(-2) \oplus V(-3)$ and $V'(-3) \oplus V'(-4) \oplus V'(-5) $ we can express multiplication by $J$ in the other components as matrices with entries in $R$.  From \eqref{rep3} and \eqref{rep4} we get a pair of $9\times 9$ matrices
\begin{align*} M_1:(V(-1)\oplus V(-2) \oplus V(-3))\otimes R & \rightarrow (V'(-3) \oplus V'(-4) \oplus V'(-5))\otimes R \\
M_2:(V'(-3) \oplus V'(-4) \oplus V'(-5))\otimes R & \rightarrow  (V(-1)\oplus V(-2) \oplus V(-3))\otimes R
\end{align*}
By choosing bases appropriately one can show that both matrices are Kronecker products with the $3\times 3$ identity matrix $I_3$ so $M_1 = A \otimes I_3$ and $M_2=B \otimes I_3$.  Similarly, for \eqref{rep5} we can compute a matrix 
\[M_3:(W(-2) \oplus W(-4))\otimes R \rightarrow (W(-2) \oplus W(-4))\otimes R\]
 and $M_3 = C \otimes I_2$ for some choice of basis.

We can identify the matrices $A,B,C$ involved in the matrix factorizations of $\Delta$ by using
Bradford Hovinen's thesis \cite[Thm.~4.4.7]{Hovinen}, where the graded rank one CM-modules over $R/\Delta$ are classified (via matrix factorizations).
Within his classification we have
$$M_{2,0}=\coker \begin{pmatrix} w + u^2 & v^2+4u^3 \\ v^2+4u^3 \phantom{22}&  w^2 + 6uv^2 -7u^2w + 16 u^4 \end{pmatrix},$$
and
\[ M_{4,-3,-2}=\coker \begin{pmatrix} -w - u^2 & 0 & v^2 -5uw - u^3 \\  v & -w + 3u^2 & 0 \\ u & v & -w - u^2 \end{pmatrix}, \]
which is is the matrix factorization of the normalization $\widetilde{R/(\Delta)}$.

The result is that $S/(J)$ is a direct sum of $4$ nonisomorphic CM-modules corresponding to the nontrivial irreducible representations of $S_4$. One can calculate the ranks explicitly or use the formulas in Section \ref{Sub:isotypical}:

\begin{theorem}
As a $R/\Delta$-module, 
$$
\ytableausetup{centertableaux,boxsize=0.25em}
S/(J) \cong M_{\ydiagram{4}}\oplus M^3_{\ydiagram{3,1}} \oplus M^3_{\ydiagram{2,1,1}} \oplus M_{\ydiagram{2,2}}^2,$$
where $M_{\ydiagram{4}}\cong R/(\Delta)$, $M_{\ydiagram{2,1,1}} \cong M_{4,-3,-2}$, the Jacobian ideal of $R/(\Delta)$ (also isomorphic to the normalization of $R/\Delta$), $M_{\ydiagram{3,1}}$ is the syzygy of $M_{\ydiagram{2,1,1}}$, i.e., the module of logarithmic derivations along $\Delta$, and $M_{\ydiagram{2,2}} \cong M_{2,0}$, which is isomorphic to the ideal defining the singular cusp in $\Delta=0$. \\
The ranks of the modules over $R/(\Delta)$ are $\rank(M_{\ydiagram{4}})=\rank(M_{\ydiagram{2,1,1}})=\rank(M_{\ydiagram{2,2}})=1$ and $\rank(M_{\ydiagram{3,1}})=2$.
\end{theorem}
In particular, this shows that $\rank_{R/(\Delta)}(S/(J))=12$, and thus $\rank_{R/(\Delta)}(\End_{R/(\Delta)}(S/(J))=\rank(\overline{A})=144$, by Example~\ref{Ex:ranks}. 
In Fig.~\ref{Fig:swallowtail} below the curves corresponding to the modules $M_{\ydiagram{2,1,1}}$ and $M_{\ydiagram{2,2}}$ are sketched on the swallowtail from two different perspectives.

\begin{figure}[!h]   
\begin{tabular}{c@{\hspace{1.cm}}c}
\includegraphics[width=0.4 \textwidth]{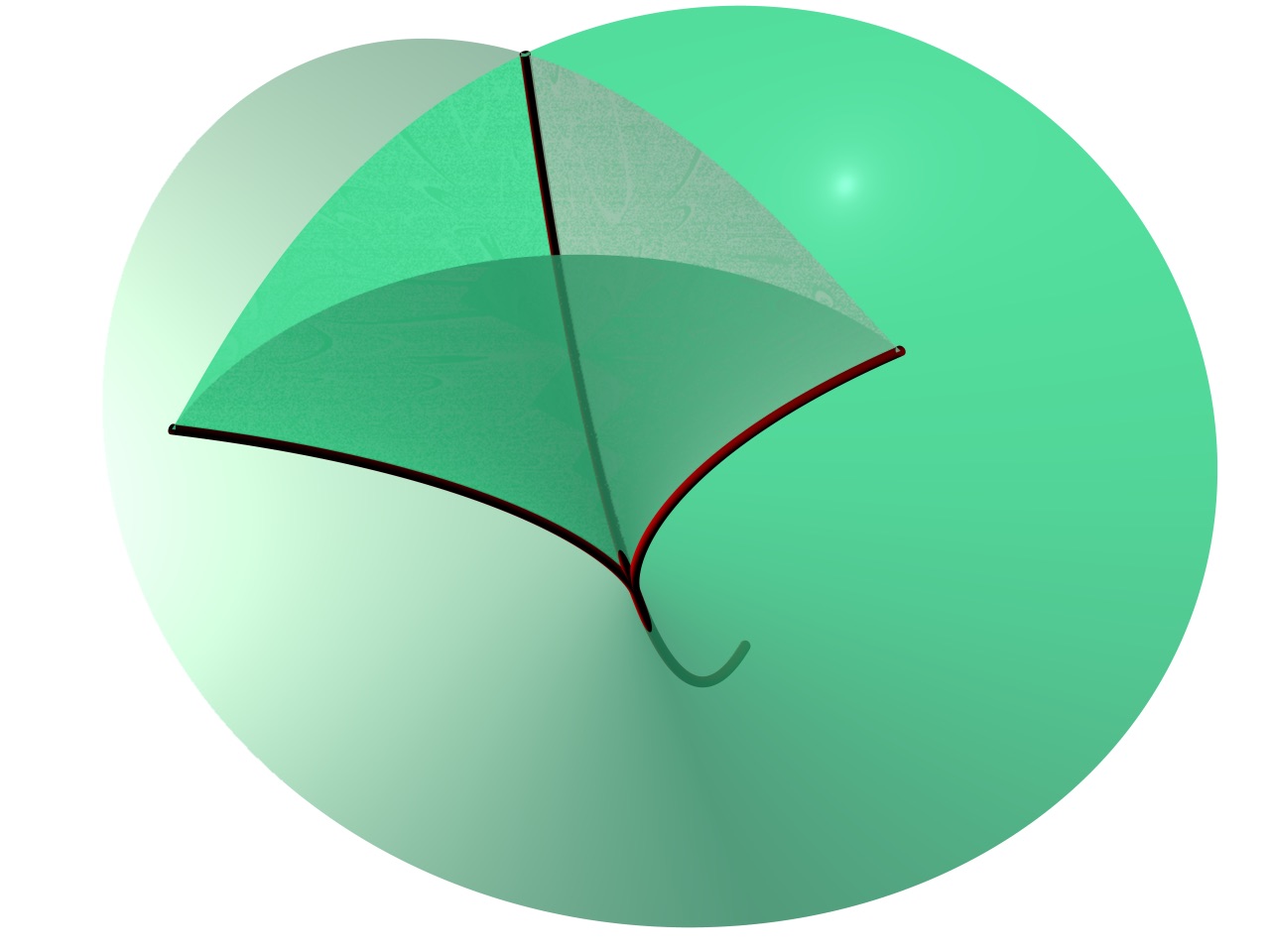}& 
\includegraphics[width=0.4 \textwidth]{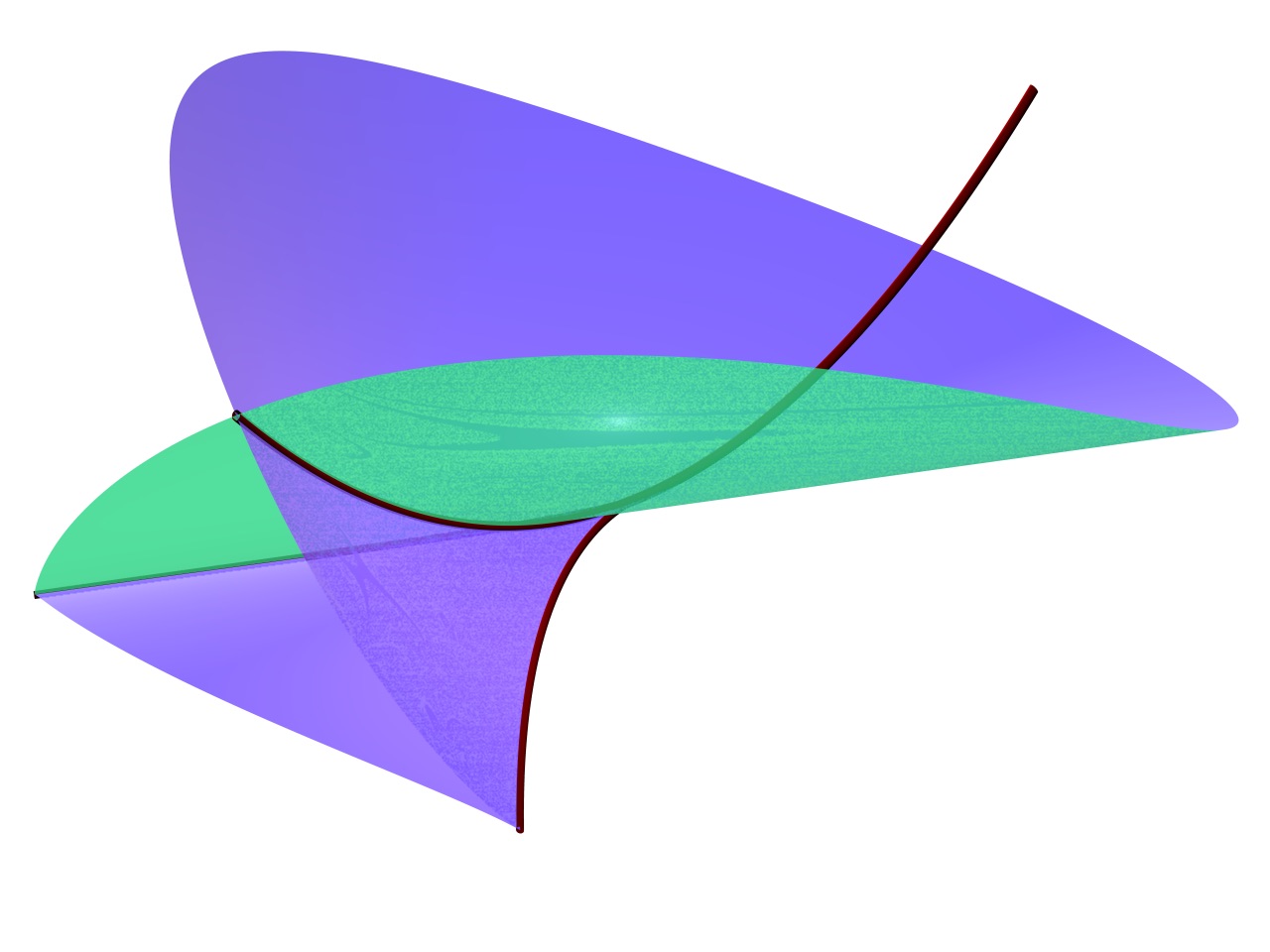}
\end{tabular}
\caption{ The swallowtail.} 
\label{Fig:swallowtail}
\end{figure}
For this example, one can also draw the McKay quiver, see \eqref{Eq:S4-McKayQuiver}. The quiver of $\overline{A}$ is obtained from \eqref{Eq:S4-McKayQuiver} by deleting the vertex and incident arrows corresponding to the determinantal representation $\ydiagram{1,1,1,1} \ $.

\begin{equation} \label{Eq:S4-McKayQuiver}
{\begin{tikzpicture}[baseline=(current  bounding  box.center)] 
\ytableausetup{centertableaux,boxsize=0.4em}
\node (triv) at (0,0) {$\ydiagram{4}$};

\node (R1) at (2,0) {$\ydiagram{3,1}$} ;
\node (R2) at (4,0)  {$\ydiagram{2,1,1}$};
\node (det) at (6,0)  {$\ydiagram{1,1,1,1}$};
\node (W) at (3,-1.5)  {$\ydiagram{2,2}$};

\draw [->,bend left=25,looseness=1,pos=0.5] (triv) to node[]  [above]{$ $} (R1);
\draw [->,bend left=25,looseness=1,pos=0.5] (R1) to node[]  [above]{$ $} (triv);

\draw [->,bend left=20,looseness=1,pos=0.5] (R1) to node[] [below] {$ $} (R2);
\draw [->,bend left=20,looseness=1,pos=0.5] (R2) to node[] [below] {$ $} (R1);

\draw [->,bend left=25,looseness=1,pos=0.5] (R2) to node[] [below] {$ $} (det);
\draw [->,bend left=25,looseness=1,pos=0.5] (det) to node[] [below] {$ $} (R2);

\draw [->,bend left=20,looseness=1,pos=0.5] (R1) to node[] [below] {$ $} (W);
\draw [->,bend left=20,looseness=1,pos=0.5] (W) to node[] [below] {$ $} (R1);
\draw [->,bend left=20,looseness=1,pos=0.5] (R2) to node[] [below] {$ $} (W);
\draw [->,bend left=20,looseness=1,pos=0.5] (W) to node[] [below] {$ $} (R2);

\path[->,every loop/.style={looseness=5.8}] (R2)
         edge  [in=120,out=60,loop]  ();
         
 \path[->,every loop/.style={looseness=8}] (R1)
         edge  [in=120,out=60,loop]  ();

\node (DD) at (6.3,0) {.};
\end{tikzpicture}} 
\end{equation}


\section{Acknowledgements}
This research was supported through the program ``Research in Pairs'' by the Mathematisches Forschungsinstitut Oberwolfach in 2015, as well as by the program ``Oberwolfach Leibniz Fellows'' in 2016, 2017. Support by the Institut Mittag-Leffler (Djursholm, Sweden) is gratefully acknowledged. The authors thank both institutes for the perfect working conditions and inspirational atmosphere. The authors also want to thank the anonymous referee for his/her thorough reading of the paper, and the many useful comments and suggestions. \\

\newcommand{\etalchar}[1]{$^{#1}$}
\def\cprime{$'$} \def\cprime{$'$}


\end{document}